\documentclass[11pt]{article}
\usepackage{fullpage}
\usepackage{amsmath,amssymb,amsthm}
\DeclareFontFamily{U}{stix2bb}{\skewchar\font127 }
\DeclareFontShape{U}{stix2bb}{m}{n} {<-> stix2-mathbb}{}
\DeclareMathAlphabet{\mathbb}{U}{stix2bb}{m}{n}
\usepackage[english]{babel}
\usepackage[left=3cm,right=3cm,top=3cm,bottom=4cm]{geometry}
\usepackage{caption}
\usepackage{tikz,tikz-cd}
\usepackage{graphicx}

\usetikzlibrary{snakes}
\usetikzlibrary{decorations.pathmorphing,shapes}
\usetikzlibrary{positioning}
\usepackage{braket}
\usepackage{mathtools,stmaryrd}

\SetSymbolFont{stmry}{bold}{U}{stmry}{m}{n}
\usepackage{lmodern,bm,bbm,mathrsfs}
\usepackage{manfnt}
\usepackage[scr=esstix]{mathalfa}
\usepackage{relsize}

\usepackage{xcolor}
\definecolor{lightorange}{RGB}{255,220,160}
\definecolor{lightblue}{RGB}{210,235,255}
\usepackage{enumitem}

\usepackage[backref=page, colorlinks=true, linkcolor=blue, citecolor=red, menucolor=green]{hyperref}
\usepackage[dvipsnames]{xcolor}

\numberwithin{equation}{section} 

\newtheorem*{theorem}{Theorem}
\newtheorem*{corollary}{Corollary}
\newtheorem*{lemma}{Lemma}
\newtheorem*{proposition}{Proposition}

\theoremstyle{definition}
\newtheorem*{definition}{Definition}
\newtheorem*{remark}{Remark}

\newtheorem*{example}{Example}

\newtheorem*{assumption}{Assumption}

\renewcommand{\bar}[1]{\overline{#1}}
\newcommand{\un}[1]{\ensuremath{\underline{#1}}}
\newcommand{\wt}[1]{\ensuremath{\widetilde{#1}}}
\newcommand{\wh}[1]{\ensuremath{\widehat{#1}}}
\newcommand{\cat}[1]{\mathscr{#1}}
\renewcommand{\hat}[1]{\widehat{#1}}

\renewcommand{\tilde}[1]{\widetilde{#1}}
\renewcommand{\vec}[1]{\bm{{#1}}}
\newcommand\numberthis{\addtocounter{equation}{1}\tag{\theequation}}

\newcommand*\al{\alpha}
\newcommand*\be{\beta}
\newcommand*\ga{\gamma}

\newcommand*\lam{\lambda}
\newcommand*\sig{\sigma}

\newcommand{\bA}{\mathbb{A}}
\newcommand{\bB}{\mathbb{B}}
\newcommand{\bC}{\mathbb{C}}
\newcommand{\bE}{\mathbb{E}}
\newcommand{\bF}{\mathbb{F}}
\newcommand{\bG}{\mathbb{G}}
\newcommand{\bGm}{{\mathbb{G}_m}}

\newcommand{\bH}{\mathbbm{H}}
\newcommand{\bk}{\mathbbm{k}}
\newcommand{\bK}{\mathbb{K}}
\newcommand{\bfK}{\mathbf{K}}
\newcommand{\bL}{\mathbb{L}}
\newcommand{\bM}{\mathbb{M}}
\newcommand{\bN}{\mathbb{N}}
\newcommand{\bP}{\mathbb{P}}
\newcommand{\bQ}{\mathbb{Q}}
\newcommand{\bR}{\mathbb{R}}
\newcommand{\bT}{\mathbb{T}}
\newcommand{\bV}{\mathbb{V}}
\newcommand{\bZ}{\mathbb{Z}}
\newcommand{\cA}{\mathcal{A}}

\newcommand{\cE}{\mathcal{E}}
\newcommand{\cF}{\mathcal{F}}

\newcommand{\cH}{\mathcal{H}}

\newcommand{\cK}{\mathcal{K}}
\newcommand{\cL}{\mathcal{L}}
\newcommand{\cM}{\mathcal{M}}

\newcommand{\cO}{\mathcal{O}}

\newcommand{\cV}{\mathcal{V}}

\newcommand{\fB}{\mathfrak{B}}
\newcommand{\fc}{\mathfrak{c}}
\newcommand{\fC}{\mathfrak{C}}
\newcommand{\fe}{\mathfrak{e}}

\newcommand{\fF}{\mathfrak{F}}

\newcommand{\fM}{\mathfrak{M}}

\newcommand{\fN}{\mathfrak{N}}

\newcommand{\fv}{\mathfrak{v}}

\newcommand{\fX}{\mathfrak{X}}
\newcommand{\fY}{\mathfrak{Y}}
\newcommand{\fZ}{\mathfrak{Z}}
\newcommand{\scE}{\mathscr{E}}

\newcommand{\scL}{\mathscr{L}}

\newcommand{\sA}{\mathsf{A}}

\newcommand{\sE}{\mathsf{E}}
\newcommand{\si}{\mathsf{i}}

\newcommand{\sS}{\mathsf{S}}
\newcommand{\sT}{\mathsf{T}}

\newcommand{\sv}{\mathsf{v}}

\newcommand{\sZ}{\mathsf{Z}}
\newcommand{\sz}{\mathsf{z}}

\newcommand{\loc}{\mathrm{loc}}

\newcommand{\pl}{\mathrm{pl}}
\newcommand{\pt}{\mathrm{pt}}
\newcommand{\red}{\mathrm{red}}

\newcommand{\vac}{\mathbf{1}}
\newcommand{\vir}{\mathrm{vir}}

\newcommand\blackbullet{\raisebox{-.1ex}{\scalebox{1.5}{$\bullet$}}}

\newcommand*\Art{\mathscr{A}\!\textit{rt}}
\DeclareMathOperator{\Aut}{Aut}
\DeclareMathOperator{\End}{End}
\DeclareMathOperator{\Ext}{Ext}
\newcommand*\mExt{\mathcal{E}\textnormal{xt}}

\DeclareMathOperator{\ch}{ch}

\newcommand*\Coh{\textnormal{Coh}}

\newcommand*\cs{\textnormal{cs}}

\DeclareMathOperator{\fix}{\mathsf{fix}}
\DeclareMathOperator{\Fr}{Fr}
\DeclareMathOperator{\fr}{fr}
\DeclareMathOperator{\cFr}{\mathcal{F}{\it r}}

\DeclareMathOperator{\GL}{GL}
\DeclareMathOperator{\Hilb}{Hilb}

\DeclareMathOperator{\Hom}{Hom}

\DeclareMathOperator{\id}{id}

 \DeclareMathOperator{\gr}{gr}
\DeclareMathOperator{\im}{im}

\DeclareMathOperator{\JS}{JS}

\DeclareMathOperator{\PGL}{PGL}

\newcommand*\PPerf{\mathscr{P}\!\textit{erf}}

\DeclareMathOperator{\PVP}{PVP}
\newcommand*\rig{\textnormal{pl}}

\newcommand*\QJS{Q^{\JS}}

\newcommand{\RHom}{\textnormal{RHom}}
\DeclareMathOperator{\rk}{rk}
\newcommand*\Sh{\textnormal{Sh}}

\DeclareMathOperator{\tot}{tot}

\DeclareMathOperator{\vdim}{vdim}

\DeclarePairedDelimiterX{\lseries}[1]{(}{)}{\mkern-2mu\delimsize(#1\delimsize)\mkern-2mu}
\DeclarePairedDelimiterX{\pseries}[1]{[}{]}{\mkern-2mu\delimsize[#1\delimsize]\mkern-2mu}

 \newcommand{\bbLambda}{%
  \tikz[
    baseline=0pt,
    x=1em,
    y=1em,
    line cap=round,
    line join=round
  ]{
    \path[use as bounding box] (0,0) rectangle (.82,.73);

    \draw[line width=.055em]
      (.06,.02) -- (.41,.70) -- (.76,.02);

    \draw[line width=.055em]
      (.17,.02) -- (.44,.55);

    \draw[line width=.055em]
      (.06,.02) -- (.17,.02);
  }%
}

\tikzset{%
  vertex/.style={shape=circle,fill=black,minimum size=6pt,inner sep=0},
  framing/.style={shape=rectangle,fill=black,minimum size=6pt,inner sep=0},
  baseline={([yshift=-0.8ex]current bounding box.center)}
}

\title{Wall-crossing for equivariant DT4 invariants}
\author{Arkadij Bojko, Nikolas Kuhn, Henry Liu, Felix Thimm}
\date{\today}

\begin{document}

\maketitle
\begin{abstract}
We prove the wall-crossing formula conjectured by Gross--Joyce--Tanaka for equivariant enumerative invariants of CY4 categories equipped with framing functors. We also establish a version for stable pairs with fixed-determinant obstruction theories, as used in earlier applications by the first-named author. The main technical ingredient, which we develop in this work, is a construction of well-behaved CY4 pullback virtual classes using Jouanolou devices.
\end{abstract}

\tableofcontents

\section{Introduction}
Invariants associated to moduli commonly depend on a choice of a parameter that we will henceforth call \textit{stability conditions}. \textit{Enumerative wall-crossing} studies how such invariants change as one crosses walls between different chambers in the space of stability conditions. In the case of sheaves, or more generally $\bC$-linear exact categories, this phenomenon was studied, for example, in \cite{Tha96, mochizuki, KS, JoyceSong, GJT, Joyce2021, Bo25, KLT25, HL-nonabelian, karpov-moreira-nal-wc}.
 
Here, we focus on enumerative invariants of Calabi--Yau 4 (CY4) categories which has its roots in \cite{DT} while \cite{BJ,OT} laid the formal foundations of the subject.\footnote{Some examples were already studied in \cite{CL}.} Wall-crossing of these invariants has been conjectured in \cite[§4.4]{GJT} (see also \cite[§2.5]{Bo24}) and has been studied in \cite{Bo25} where it was also proved for CY4 quivers and local CY fourfolds. Here, we build on techniques from \cite{Bo25} together with ideas from \cite{KLT23, KLT25} and prove the general statement of equivariant wall-crossing in CY4 categories in the presence of framing functors. The proof uses the CY4 pullback machinery developed in §\ref{sec:sym-pullback}.

In particular, we prove \cite[Conjecture 4.11]{GJT} and \cite[Conjecture 2.9]{Bo24} making the results in \cite{Bo24, Bo21} into robust theorems.  Further applications of this work will include addressing conjectures about curve and surface counting invariants stated, for example, in \cite{CK3, CT1, CT2, CT2-1, BKP2}.  
\subsection{Equivariant Invariants in Lie Algebras}
\subsubsection{}
Let $\cat{A}$ be a CY4 abelian category, in the sense explained in §\ref{sec:CY4-category}, with a possible action of a torus $\sT$. When defining invariants, one fixes two pieces of data first:
\begin{itemize}
\item Consider the \textit{even classes} \begin{equation}
\label{eq:Keven}
K^0_e(\cA) = \big\{\alpha\in K^0(\cA):\chi(\alpha,\alpha)\in 2\bZ\big\}
\end{equation}
in the Grothendieck group of $\cat{A}$ and a quotient $K^0_e(\cat{A})\twoheadrightarrow \bar{K}(\cat A)$. We then fix a class $\al\in \bar{K}(\cat A)$.
\item We also choose a \textit{weak stability condition} $\tau$ in the sense of §\ref{sec:weak-stability}. Such weak stabilities will form spaces $W$ as in §\ref{sec:W-space-of-stability-cons}.
\end{itemize}
Let $M_{\al}(\tau)$ be the moduli space parametrizing $\tau$-stable objects in the class $\al$. Relying on the presence of $-2$-shifted symplectic structures, it was described in \cite{BJ} as a real derived manifold. This introduced the question of \textit{orientability}. Denote by $T^{\vir}_{M_{\al}(\tau)}$ the natural virtual tangent bundle with Serre duality
$
T^{\vir}_{M_{\al}(\tau)}\cong \big(T^{\vir}_{M_{\al}(\tau)}\big)^\vee[-2]
$
induced by the $-2$-shifted symplectic structure. A choice of orientation for the associated derived manifold is equivalent to a choice of a trivialization 
$
\det\big(T^{\vir}_{M_{\al}(\tau)}\big)\cong \cO_{M_{\al}(\tau)}
$
compatible with Serre duality as made precise in \eqref{eq:EEorient} in Definition~\ref{sp:def:obstruction-theory}. The existence of orientations for perfect complexes on CY4 folds was addressed in \cite{CGJ,bojko, JU,karpov-thimm} and for representations of CY4 quivers in \cite{Bo25, Schmier}.

Once orientations are determined, the machinery of \cite{BJ} produces a virtual fundamental class $\big[M_{\al}(\tau)\big]^{\vir}\in H_*\big(M_{\al}(\tau),\bZ\big)$
assuming properness of the moduli space. The more algebro-geometric approach of \cite{OT} produces equivariant virtual fundamental cycles 
\begin{equation}
\label{eq:MasigvirOT}
\big[M_{\al}(\tau)\big]^{\vir}\in A^{\sT}_*\big(M_{\al}(\tau),\bZ[2^{-1}]\big)\qquad\textnormal{and}\qquad [\widehat{\mathcal{O}}_{M_{\al}(\tau)}^{\vir}]\in G^{\sT}_0\big(M_{\al}(\tau),\bZ[2^{-1}]\big)
\end{equation}
where $G^{\sT}_0(-)$ denotes the equivariant Grothendieck group of coherent sheaves. The Oh--Thomas formulation additionally supports the virtual equivariant localization formula proved in \cite{OT} and virtual pullbacks established in \cite{Park21}. Our proof of wall-crossing relies on these two results. 

\subsubsection{}
The wall-crossing itself is formulated in terms of equivariant homology or K-homology of the moduli stack $\fM_{\cat{A}}$ of objects in $\cat{A}$. Two approaches to constructing such theories are recalled in §\ref{sec:background}:
\begin{itemize}
\item operational K-homology $\bfK_\circ^\sT(-)$ following \cite{Liu2022,KLT25} explained in Definition \ref{def:operational-k-homology} and
\item equivariant homology $H^{\sT}_*(-)$ and K-homology $K_0^{\sT}(-)$ from \cite{BB1}, \cite[Appendix B]{Bo25} as summarized in §\ref{sec:Khan-bivariant} and §\ref{sec:homology-to-khomology}, respectively.
\end{itemize}
Henceforth, we give them the generic label $H^\sT(-)_\loc$. Their necessary properties are briefly summarized in §\ref{sec:homology-theories-LA}.

One such property is the existence of equivariant pushforwards along any equivariant morphism of Artin stacks. Let $ \fM^{\pl}_{\cat{A}} := \fM_{\cat{A}}\fatslash B\bG_m$ be the rigidified stack as in \cite[App. A]{AOV} or \cite{Romagny}. Then there are open embeddings 
$$
j: M_{\al}(\tau)\to \fM^{\pl}_{\cat{A}}\,.
$$
However, to pushforward the classes \eqref{eq:MasigvirOT} along $j$, one first needs to localize as explained in §\ref{sec:universal-invariants-shorthand} and §\ref{def:bivariant-hom-invariants}. We denote the resulting invariants by 
$$
j_*\sZ_{M_{\al}(\tau)}\in H^\sT(\fM^{\pl}_{\cat{A}})_{\loc}\,.
$$
\subsubsection{}
In \cite[§3.2]{Joycehall}, Joyce introduced graded vertex algebras on the homology $H_*(\fM_{\cat{A}})$. Their quotient by the translation operator carries a Lie algebra structure due to \cite{Borcherds}.  This Lie algebra was partly identified with $H_*(\fM^{\pl}_{\cat{A}})$ and used to formulate a wall-crossing conjecture for the invariants $\sZ_{M_{\al}(\tau)}$ in \cite[§4.4]{GJT}.

In §\ref{sec:homology-theories-LA}, we give an overview of how the equivariant homology theories $H^{\sT}_*(-)$, $\bfK_\circ^\sT(-)$, and $K_0^{\sT}(-)$ give rise to vertex algebra-like structures on $H^\sT(\fM_{\cat{A}})_\loc$ which in turn lead to Lie algebras on the quotient denoted by $L(\fM_{\cat{A}})_\loc$. Note that $H^{\sT}_*(\fM_{\cat{A}})$ for $\sT=\{1\}$ is identical to the homology used in \cite{GJT} and carries the same vertex algebra structure. For the theories from §\ref{sec:equivariant-homology-bivariant}, there is still a partial identification of $L(\fM_{\cat{A}})_\loc$ with $H^\sT(\fM^{\pl}_{\cat{A}})_\loc$ by Proposition \ref{prop:quotient-by-T}, but more generally, one can use a natural morphism
$$
\Pi^\pl_*:L(\fM_{\cat{A}})_\loc\to H^\sT(\fM^{\pl}_{\cat{A}})_\loc
$$
as in Lemma \ref{lem:pl-group-functoriality}. We choose lifts of $j_*\sZ_{M_{\al}(\tau)}$ in $L(\fM_{\cat{A}})_{\loc}$ along $\Pi^\pl_*$. 
\subsection{Wall-Crossing Results}
\subsubsection{}
The invariants $\sZ_{M_{\al}(\tau)}$ can be generalized to cases with strictly semistables using the approach in §\ref{sec:sst-inv}, the philosophy behind which was explained in \cite[§5.3]{Bo25}. The construction, a priori, depends on a choice of a framing functor $\Fr$ from Definition \ref{bg:def:framing-functor}. Framing functors are defined on smaller subcategories $\cat{A}^{\Fr}\subset \cat{A}$ and assign a vector space $\Fr(E)$ to each $E\in \cat{A}$. Its dimension should be constant for a fixed class $\al$ and is denoted by $\fr(\al)$.

Our next theorem proves that the resulting invariants are independent of choices of $\Fr$. In the following we use $\fM_{\al}(\tau)$ to denote the moduli substacks of $\tau$-semistable objects in class $\al$ contained in $\fM_{\cat{A}}$. We only consider a fixed subset of so-called emergent classes $\al$ denoted by $\scE(\cat{A})\subset \bar{K}(\cat{A})$ as in Definition \ref{def:categoryA} (d).
\begin{theorem}[Semistable invariants] \label{thm:sst-invariants}
  Suppose $W$ is a space of weak stability conditions on $\cat{A}$ with fixed $\scE(\cat{A})$ for which
  Assumption~\ref{ass:stab} holds. Then, for any $\tau\in W$, there exist canonically defined classes
  \begin{equation}
    \sz_{\alpha}(\tau)\in L(\fM_{\cat{A}})_{\loc}\qquad\textnormal{for all }\al\in \scE(\cat{A})\,,
  \end{equation}
  which satisfy the following properties:
  \begin{enumerate}[label = (\roman*)]
  \item \label{item:vss-support} $\sz_\alpha(\tau)$ is supported on
    $\fM_\alpha(\tau)$;
  \item \label{item:vss-no-strictly-semistables} 
    if $\fM_{\alpha}(\tau)$ contains no strictly semistables, then  $\Pi^\pl_*\sz_{\al}(\tau) = \sZ_{M_{\al}(\tau)}$;
  \item \label{item:vss-isomorphic-moduli} if $\tau'\in W$ is another weak
    stability condition and $\fM_\alpha(\tau) = \fM_\alpha(\tau')$
    for a given $\alpha$, then $\sz_\alpha(\tau) = \sz_\alpha(\tau')$;
  \item \label{item:vss-pairs-relation} for any framing functor $\Fr$ from Definition \ref{bg:def:framing-functor} such that $\fM_\alpha(\tau) \subset
    \fM_\alpha^{\Fr}$, in the notation of
    Definition~\ref{def:pair-invariant}, we have
    \begin{equation} \label{eq:sstable-def-intro}
     \fr(\al) \breve\sZ_{\alpha,1}^{\Fr}(\tau^{\JS}) = \sum_{\substack{n>0 \\ \alpha = \alpha_1+\cdots+\alpha_n\\ \forall i: \,\phi(\alpha_i) = \phi(\alpha)\\ \;\;\fM_{\alpha_i}(\tau) \neq \emptyset}} \frac{\fr(\al_1)}{n!} \left[\sz_{\alpha_n}(\tau), \left[\cdots,\left[\sz_{\alpha_2}(\tau), \sz_{\alpha_1}(\tau)\right]\cdots\right]\right]
    \end{equation}
    in $L(\fM_{\cat{A}})_{\loc}$, where the left-hand side is defined in \eqref{eq:pairs-invariant}.
  \end{enumerate}
\end{theorem}
For torsion-free sheaves on projective Calabi--Yau fourfolds, this theorem is proved in \S \ref{sec:independencequantumLefschetz} by completing the argument from \cite[§1.4, §6.5]{Bo25}. It relied on a quantum Lefschetz type argument, but needed Joyce--Song wall-crossing to hold, which is proved using Example \ref{ex:JS-wall-crossing-setup}. For more general invariants, we use the approach introduced in \cite[§9.2]{Joyce2021} based on \cite[§7.2.2]{mochizuki}. Additionally, we need to use the results and constructions in Theorem \ref{sp:thm:symm-pullback-localization}, Proposition \ref{prop:comparison}, and Corollary \ref{cor:flag-pushforward}, which are necessary due to the absence of CY4 obstruction theories on moduli spaces represented by \eqref{eq:Vshapedquiver}. We return to this point in §\ref{sec:CY4-pullback-intro}.
\subsubsection{}
\label{sec:general-wall-crossing-intro}

Our main theorem is now the following general wall-crossing formula for the semistable invariants defined above.
\begin{theorem}
\label{thm:general-wall-crossing-intro}
 Suppose $W$ is a space of weak stability conditions on $\cat{A}$ with fixed $\scE(\cat{A})$ for which
  Assumption~\ref{ass:stab} holds. Let $\tau,\tau'\in W$, then the wall-crossing formula
    \begin{equation} \label{eq:general-wcf}
  \sz_\alpha(\tau') = \sum_{\substack{n>0\\\alpha = \alpha_1 + \cdots + \alpha_n\\\forall i:\, \fM_{\alpha_i}(\tau)\neq \emptyset}}\tilde U\left(\alpha_1,\dots,\alpha_n;\tau,\tau'\right)\left[\left[\cdots\left[\sz_{\alpha_1}(\tau),\sz_{\alpha_2}(\tau)\right],\cdots\right],\sz_{\alpha_n}(\tau)\right]
\end{equation}
holds in $L(\fM_{\cat{A}})_{\loc}$. Here  $\tilde U\left(\alpha_1,\dots,\alpha_n;\tau,\tau'\right)$ are the coefficients defined, for example, in \cite[§3.5]{GJT}.
\end{theorem}

In \cite{Bo25}, this was proved under the limiting conditions that required the existence of CY4 obstruction theories on moduli spaces represented by the quivers \eqref{wc:eq:flag-quiver} and \eqref{wc:fig:wc-ms-quiver}. This included $\cat{A}$ consisting of representations of CY4 quivers and compactly supported sheaves on local CY fourfolds. For more general sheaves, \cite[§6.3]{Bo25} shows that the required obstruction theories may fail to exist. Using the machinery developed in Section~\ref{sec:sym-pullback}, this restriction is removed here entirely.

Using $H_*(-)$ for $\sT = \{1\}$, this in particular addresses \cite[Conjecture 4.11]{GJT}. We give a brief overview of the strategy of the proof in §\ref{sec:strategy} as it follows along the lines of \cite[§10, §11]{Joyce2021}\footnote{This work, in turn, was based on the approach in \cite{mochizuki}.} which was already summarized in \cite[§8.1]{Bo25} and \cite[§4]{KLT25}.
\subsubsection{}
Another setting we consider is wall-crossing for stable pairs with fixed determinant obstruction theories. For now we choose a projective CY4 fold $X$ with a heart $\cat{B}$ of $D^b(X)$ and a subset $\scE(\cat{B})\subset \bar{K}(\cat{B}) = \bar{K}\big(\Coh(X)\big)$. Let $W$ be a space of stability conditions on $\cat{B}$ and let $O\in Coh(X)$. We consider the category $\cat{B}_O$ of pairs $P=(V\otimes O\xrightarrow{s} F)$ where $V$ is a vector space, $F$ a positive codimension sheaf on $X$, and $s$ a morphism of sheaves. The class of such a $P$ will be $(\dim(V), \llbracket F\rrbracket)\in \bZ_{\geq 0}\times \bar{K}\big(\Coh(X)\big)$. Suppose that for each $\tau\in W$ and $\beta\in \scE(\cat{B})$ we are given a stability condition $\tau^p$ on $\cat{B}_O$ and a $(d,\al)$ with $d=0,1$ such that 
\begin{equation}
\label{eq:semistable-pair-intro}
\left\{\begin{array}{c}\tau^p\textnormal{-semistable pairs }P  \\ 
\textnormal{of class }(d,\al)
\end{array}
\right\}^{\pl}\cong \begin{cases}
         \fM_{\be}(\tau)_{\det(O)}  &\text{if }d=1\,,\\
         \\
    \big(\fM_{\be}(\tau)\big)^{\pl}  &\text{if }d=0 \,.
        \end{cases}
\end{equation}
Here, the subscript $(-)_{\det(O)}$ denotes the moduli substack of objects with fixed determined $\det(O)$. Minor modifications of the proof of Theorem \ref{thm:general-wall-crossing-intro} shows the following.
\begin{theorem}
    In the situation described above, suppose that the data $\cat{B},\scE(\cat{B})$, and $W$ satisfy Assumption \ref{ass:pairWC}. Then for any $\tau\in W$ and $\beta\in \scE(\cat{B})$, there are canonically defined classes $\sz_{\beta}(\tau) \in L(\fM_{\cat{B}})_{\loc}$ satisfying the properties from Theorem \ref{thm:sst-invariants}. For any two weak stability conditions $\tau,\tau'\in W$ and  $\be\in \scE(\cat{B})$, the wall-crossing formula
        \begin{equation} \label{eq:general-wcf-pairs}
  \sz_\beta(\tau') = \sum_{\substack{n>0\\\beta = \beta_1 + \cdots + \beta_n\\\forall i:\, \fM_{\beta_i}(\tau)\neq \emptyset}}\tilde U\left(\beta_1,\dots,\beta_n;\tau,\tau'\right)\left[\left[\cdots\left[\sz_{\beta_1}(\tau),\sz_{\beta_2}(\tau)\right],\cdots\right],\sz_{\beta_n}(\tau)\right]
\end{equation}
holds in $L(\fM_{\cat{B}})_{\loc}$. 
\end{theorem}
For local CY4 folds, this was proved in \cite[Theorem 1.8]{Bo25}.
\begin{remark}
  Let $\fN_O$ be the moduli stack of $\fB_O$.  Using the identification \eqref{eq:semistable-pair-intro}, one can consider the classes $\sz_{d,\al}(\tau^p)=\sz_{\beta}(\tau)$ in $L(\fN_
  O)_{\loc}$, instead. Then, the wall-crossing formula \eqref{eq:general-wcf-pairs} still holds there.
\end{remark}
\subsubsection{}
A particular scenario of the above fixed-determinant stable pair wall-crossing is discussed in Example \ref{ex:JS-wall-crossing-setup}. Note that this example also allows higher rank sheaves $F$ inside of the pairs $V\otimes O \to F$ as long as $X$ is strict (for now). This example proves \cite[Conjecture 2.9]{Bo24} which was applied to Hilbert schemes there.
\begin{corollary}
In $L(\fN_{\cO_X})$, the following wall-crossing formula holds:
$$\sum_{n\geq 0}j_*\big[\Hilb^n(X)\big]^{\vir}q^n = \exp\left(\sum_{m>0}\Big[[\fM_{mp}]^{\operatorname{inv}},-\Big]q^m\right)e^{(1,0)} $$
where $[\fM_{np}]^{\operatorname{inv}} = \sz_{n p}(\tau)$ for the class $p\in K^0(\Coh(X))$ of a sky-scraper sheaf, $\tau$ the Gieseker stability, and $e^{(1,0)}$ the point-class of $(\cO_X\to 0)$.
\end{corollary}
Assumption \ref{ass:pairWC} was set up in a way that it should also allow to prove the analogous Quot-scheme formula from \cite{Bo21} once the first-named author finds some time to address the assumptions.
\subsection{General Methods}
\subsubsection{}
A core ingredient in our approach to wall-crossing are the auxiliary framed moduli stacks explained in \S\ref{sec:auxiliary-stacks}, which are determined by a quiver $Q$ and a vector of framing functors $\vec\Fr$. The quiver vertices $Q_0$ are decomposed as $Q_0 = Q_0^o \sqcup Q_0^f$ denoted by $\blackbullet$ and $\blacksquare$, respectively. Each framing of $\vec\Fr$ is uniquely attached to a vertex of $Q^o_0$. The auxiliary framed moduli stacks $\fM^{Q(\Fr)}$ then consist of objects $(E,\vec{V},\vec{\rho})$ where $E\in \cat{A}^{\Fr}$ and $(\vec{V}, \vec\Fr(E),\vec{\rho})$ is a representation of the quiver $Q$. 

Such auxiliary spaces come up in four types, each for a different purpose:
\begin{itemize}
  \item The auxiliary stack of the quiver $Q^{\JS}$ of Definition~\ref{def:pair-invariant} helps to define invariants when there are strictly semistable objects. This is important since it necessarily happens at any stability condition on a wall.
  \item The auxiliary stack of the quiver $Q\wedge Q$ of Definition~\ref{sst:def:wedge-quiver} is used to show that the above invariants are independent of the chosen framing functor.
  \item The auxiliary stack of the quiver $\bar{Q}(r)$ of \eqref{wc:eq:flag-quiver} is used to break down a potentially complicated wall-crossing into simple wall-crossing steps.
  \item The auxiliary stack of the quiver $\hat{Q}$ of \eqref{wc:fig:wc-ms-quiver} is used to define the master space $\bM$, which allows us to solve the simple wall-crossings via localization.
\end{itemize}
The auxiliary spaces come equipped with smooth maps to the base moduli stack, and for all these uses, we require a well-behaved lift of the virtual class to stable loci in the auxiliary stacks.
  
\subsubsection{}\label{sec:CY4-pullback-intro}
In \cite[§6.3]{Bo25}, the first-named author explained why a direct construction of CY4 obstruction theories on the above auxiliary framed stacks is not possible. Instead, drawing on the approach developed in \cite{KLT23,KLT25}, we use Jouanolou devices to define virtual fundamental classes that satisfy the virtual equivariant localization formula.
More generally, let $\fM$ be a locally finite Artin stack with a $\bT:=\bG_m\times \sT$-action and a $\bT$-equivariant CY4 obstruction theory $\bE\to \bL_{\fM}$ in the sense of Definition \ref{sp:def:obstruction-theory}. If $M\to \fM$ is a smooth $\bT$-equivariant morphism from a quasi-compact, separated algebraic $\bT$-space with the $\bT$-equivariant resolution property, then Theorem \ref{sp:thm:symm-pullback-localization} introduces classes
\begin{equation}
\label{eq:CY4-pullback-classes-intro}
  [\widehat{\mathcal{O}}_{M}^{\vir}]\in G^{\bT}_0\big(M,\bZ[2^{-1}]\big)\qquad\textnormal{and}\qquad  [\tilde{\mathcal{O}}_{M^{\bG_m}}^{\vir}]\in G_0^\bT\big(M^{\bG_m},\bZ[2^{-1}]\big)
\end{equation}
independent of the $\bT$-equivariant Jouanolou device $A\to M$. These classes satisfy the Oh-Thomas equivariant localizations formula from Theorem \ref{thm:eqvirloc} with respect to the first factor $\bG_m$.

The proof of this result uses functoriality of Park's virtual pullback diagrams as recalled in §\ref{sec:pvpinfty}. This was formulated and shown in \cite[Appendix A]{Bo25} on the level of stable $\infty$-categories. This setting also simplifies the question of whether the isotropy condition is preserved due to \cite[Lemma A.16]{Bo25}.
\subsubsection{}
Once the $\bG_m$-localization formula from Theorem \ref{sp:thm:symm-pullback-localization} is applied to the enhanced master space $\bM$ from Definition \ref{wc:def:horizontal-master-space}, we use Proposition \ref{prop:comparison} to compare the orientations and invariants on the fixed-point loci. This established the horizontal flag wall-crossing formula from Theorem \ref{thm:horizontal-flag-wc} which takes place in $L(\fM^{\bar{Q}(\Fr)})_\loc$ for the following quiver
$$
\includegraphics{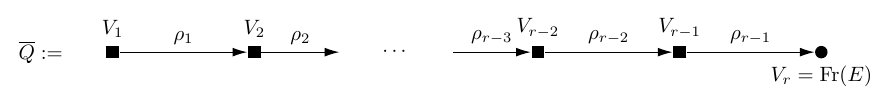}
$$
The final step in recovering Theorem \ref{thm:general-wall-crossing-intro} is projecting the horizontral flag wall-crossing formula along the natural map $\fM^{\bar{Q}(\Fr)}\to \fM_{\cat A}$ which forgets everything except of the objects $E\in \cat{A}^{\Fr}$.  This uses the flag pushforward formula from Corollary \ref{cor:flag-pushforward} which applies to the first class in \eqref{eq:CY4-pullback-classes-intro}. This argument is explained in §\ref{sec:flagprojection}.
\subsection{Notation, Conventions, Acknowledgements}

\subsubsection{}

We work over $\bC$. All Artin stacks, i.e. algebraic stacks, are
assumed to be locally of finite type. 

\subsubsection{}

During the course of this project, we benefited from fruitful discussions with Emile Bouaziz, Dominic Joyce, and Adeel Khan.

A.B. was supported by the Research Start-up Fund of the Shanghai Institute for Mathematics and Interdisciplinary Sciences (No. 2302-SRFP-2026-0009) and by NSFC RFIS Grant No. W2533004 (“Wall-crossing for enumerative invariants and Virasoro constraints”). N.K. was supported by Research Council of Norway grant number 302277 -
”Orthogonal gauge duality and non-commutative
geometry” and by EPSRC grant number EP/X040674/1. H.L. was
supported by World Premier International Research Center Initiative
(WPI), MEXT, Japan.

The authors were informed by I. Karpov and M. Moreira that their group is developing an alternative approach to the problem of CY$4$ wall-crossing, similar to their previous wall-crossing paper \cite{karpov-moreira-nal-wc}, based on non-abelian localization.
\subsubsection{}
Here, we summarize some notation, used throughout this work, regarding Chern classes and equivariant Euler classes:
\begin{itemize}
\item We write $K_\sT^\circ(\fX) \coloneqq K_0(D_{\cat{Perf},\sT}(\fX))$ and $G_0^\sT(\fX) \coloneqq K_0(D_{\cat{Coh},\sT}(\fX))$\footnote{This is called $K_\sT(\fX)$ in \cite{KLT25}.};
\item $\rk(\Theta)$ without any further subscripts denotes the usual rank of a K-theory class $\Theta$;
\item $c_k(\Theta)$ is the $k$-th Chern class and $c_{\rk}(\Theta) =  c_{\rk(\Theta)}(\Theta)$;
\item $\sT$ will denote a finite-dimensional algebraic torus and $\bG_m$ the 1-dimensional one;
\item Since, as discussed in Section~\ref{sec:homology-theories-LA} is relatively independent of the underlying homology theory, we denote the Euler class in both homology and K-theory by $\fe(-)$;
\item for a $\bG_m$-weight $t=e^u$, the equivariant Euler class $\fe_{\bG_m}(t\Theta)$ is expanded as
$$
\fe_{\bG_m}(t\Theta)=u^{\rk(\Theta)}c_{u^{-1}}(\Theta)
$$
for $c_{u^{-1}}(\Theta) = \sum_{k\geq 0}c_k(\Theta)u^{-k}$, which naturally generalizes to any $\sT$-weight;
\item for a $\sT$-equivariant vector bundle $E$, we also write $\fe_{\sT}(E)=c_{\rk}(E)$;
\item For a perfect complex $\Theta$, we use $\Lambda_t(\Theta) = \sum_{k\geq 0}\Lambda^k(\Theta)t^k$ where $\Lambda^k(\Theta)$ are perfect complexes defined for example in \cite[Chapter I, §4.2.2.2]{Illusie1971}.
\item $\fc_k(\Theta)$ is the $k$-th K-theoretic Conner--Floyd--Chern class,
  i.e.
  \[ \fc_k(\Theta) \coloneqq \text{coefficient of } (1 - t)^{\rk(\Theta) - k} \text{ in } \Lambda_{-t}(\Theta^\vee), \]
  and $\fc_{\rk}(\Theta) \coloneqq \fc_{\rk(\Theta)}(\Theta)$;
  \item for a $\bG_m$-weight $t$, the equivariant Connor-Floyd-Euler class $\fe_{\bG_m}(t\Theta)$ is expanded as
$$
\fe_{\bG_m}(t\Theta) = (1-t^{-1})^{\rk(\Theta)}\sum_{k\geq 0}\fc_k(\Theta)(1-t^{-1})^{-k}\,,
$$
which naturally generalizes to any $\sT$-weight;
\item we also use $\hat{\fc}_k(\Theta)
  \coloneqq \fc_k(\Theta) \otimes \det(\Theta)^{1/2}$ and $\hat{\fe}_{\sT}(\Theta) = \fe_{\sT}(\Theta)\otimes\det(\Theta)^{1/2}$;
  \item for a $\sT$-equivariant vector bundle $E$, we also write $\fe_{\sT}(E)=\fc_{\rk}(E)$ and $\hat{\fe}_{\sT}(E)=\hat{\fc}_{\rk}(E)$;
\end{itemize}

\section{CY4 Pullback and PVP Diagrams}
\label{sec:sym-pullback}

\subsection{CY4 Obstruction Theories and PVP Diagrams}

\subsubsection{}

Throughout this subsection, let $\fM$ be an Artin stack over a smooth
equidimensional base $\fB$, all in $\cat{Art}_\sT$. When $\fM$ is additionally an algebraic space, it will be written as $M$. Let
$D^-_{\cat{QCoh},\sT}(\fM)$ be its derived category of bounded-above
$\sT$-equivariant complexes with quasi-coherent cohomology
\cite{Olsson2007}. Let $\bL_{\fM/\fB} \in D^-_{\cat{QCoh},\sT}(\fM)$
denote the cotangent complex \cite{Illusie1971}.

\subsubsection{}
Following \cite[Definition 3.1]{Bo25} and \cite[§2.5.2]{KLT25}\footnote{This work discusses the equivariant CY3 version which requires different conditions from the current ones.}, where the former also applies to stacks as explained in \cite[§3.2]{Bo25}, we summarize the necessary conditions of obstruction theories for the purpose of defining invariants.
 \begin{definition}
 \label{def:viradmcl} \label{sp:def:obstruction-theory}
  A $\sT$-equivariant\footnote{Since all obstruction theories used in equivariant settings in this
    paper are $\sT$-equivariant, we will omit writing
    ``$\sT$-equivariant''.} {\it obstruction theory} on $\fM$ relative
  to $\fB$ is an object $\bE$ with a morphism
  \begin{equation}
   \label{eq:EEobtheory}
  \phi\colon \bE \to \bL_{\fM/\fB} 
\end{equation}  
in $D^-_{\cat{QCoh},\sT}(\fM)$, such that $\cH^i(\phi)$ for $i\geq 0$ are isomorphisms and $\cH^{-1}(\phi)$ is surjective. 
For $\bE$ perfect of tor-amplitude $[-3,1]$, we say that $(\bE,\phi)$ is a \textit{CY4 obstruction theory} if it satisfies the following conditions $\sT$-equivariantly:
\begin{enumerate}[wide, align=left]
    \item[\textit{(Symmetry)}] $\bE$ is called a \textit{symmetric complex} if there is an isomorphism\footnote{Note that this data is equivalent to a non-degenerate symmetric form $q$ in terms of \cite{Park21}.}
    \begin{equation}
    \label{eq:sigma}
    \begin{tikzcd} \mathbb{i}_q:\bE\arrow[r,"\sim"]&{\bE^\vee[2]}\end{tikzcd}\,, \qquad \mathbb{i}_q^\vee[2] = \mathbb{i}_q\,.
    \end{equation}
    \item[\textit{(Evenness)}] The rank of $\bE$ is even.  
    \item[\textit{(Orientability)}] There exists a trivialization, called an \textit{orientation},
    $
    o: \cO_\fM \xrightarrow{\sim}\det(\bE)
    $
satisfying
\begin{equation}
\label{eq:EEorient}
(o^*)^{-1}\circ o^{-1}= \det(\mathbb{i}_q):\begin{tikzcd}\det(\bE)\arrow[r,"\sim"]& \det(\bE)^*\end{tikzcd}\,.
\end{equation}
In this case, we will say that $(\bE,\phi)$ is orientable for given a given symmetry isomorphism \eqref{eq:sigma}.
   \item[\textit{(Cone Isotropy)}] 
   Let $\fC_\bE$ be the virtual normal cone associated to $\bE$ as defined in \cite[Proposition 2.4]{Behrend1997}, which is an abelian cone stack over $\fM$. Recall from \cite[Proposition 1.7]{Park21} (also see \cite[§3.2]{Bo25} for a discussion of this construction in the stack case) that the self-duality isomorphism $\mathbb{i}_q$ determines a quadratic function
    \begin{equation}
    \label{eq:quadratic}
    q_\bE: \fC_\bE \to \bA^1_\fM.
    \end{equation}
     By \cite[Prop. 2.6]{Behrend1997}, the obstruction theory $\phi$ provides an embedding of the intrinsic normal cone $\fC_{\fM/\fB}$  into the virtual normal cone. We say $\bE$ \textit{satisfies the isotropic condition} if the restriction of $q_\bE$ to $\fC_{\fM/\fB}$, i.e.
  \begin{equation*}
    \fC_{\fM/\fB}\hookrightarrow \fC_\bE \xrightarrow{q_\bE} \bA_\fM^1
  \end{equation*}
  vanishes.
\end{enumerate}
 \end{definition}

Exactly as in \cite[2.5.3]{KLT25}, we have the following Lemma.
\begin{lemma} \label{sp:lem:perfect-symmetric-obstruction-theory}
  Let $M \subset \fM$ be a $\sT$-invariant open locus which is an algebraic space. Then the restriction to $M$ of a symmetric obstruction theory on $\fM$ has tor-amplitude in $[-2,0]$.
\end{lemma}

\subsubsection{}\label{sp:sec:dt4-vir-class}
From now on $G^{\sT}_0(-)$ will denote the equivariant Grothendieck group of coherent sheaves. For a separated algebraic space $\fM=M$ with a CY4 obstruction theory in the sense of Definition~\ref{sp:def:obstruction-theory}, by \cite{Oh2023,Park21}\footnote{In \cite{Oh2023}, it is constructed only for quasi-projective $M$, but we want to allow situations when $M$ is a separated algebraic space which is covered only by the approach in \cite{Park21}.}, we have equivariant virtual cycles
\begin{equation*}
  [M]^{\vir}_{\sT}\in A^{\sT}_*\left(M,\bZ[2^{-1}]\right)\,,\qquad  [\widehat{\mathcal{O}}_{M}^{\vir}]\in G^{\bT}_0\big(M,\bZ[2^{-1}]\big)\,.
\end{equation*}
If $M$ is connected, then changing the choice of orientation $o$ changes the sign of $[M]^{\vir}_{\sT}$. 

Such a class was also defined in earlier work of Borisov--Joyce \cite{BJ}. When $M$ is projective and $\sT$ trivial, \cite{OT2} shows that $[M]^{\vir}:=[M]^{\vir}_{\sT}$ maps to the Borisov--Joyce class in $H_*\big(M,\bZ[2^{-1}]\big)$. This implies that the resulting class, which, by abuse of notation, we also denote by $[M]^{\vir}$, lies in the image of $H_*\big(M,\bZ\big)\to H_*\big(M,\bZ[2^{-1}]\big)$. It also follows from \cite{OT2} that the class from \cite{BJ} vanishes in $H_*\big(M,\bZ[2^{-1}]\big)$ if evenness does not hold. Thus, we only lose information about the $2$-torsion part by working with the class defined in \cite{Oh2023}.

\subsubsection{}\label{sp:sec:loc-setup}

Now we discuss how the localization formula for the above virtual classes works. In this paper, we often consider a situation where we have an action by a torus $\bT$ that splits into $\bT=\bG_m\times \sT$ for some fixed torus $\sT$. So, to fix notation for this section, we denote a general torus by by $\bT$.

Before stating the localization formula, we discuss orientation conventions. We follow here the conventions in \cite[§3.3]{Bo25}. Let $M$ be a moduli space with $\bT$-action and an oriented $\bT$-equivariant CY4-obstruction theory $\bE$. In this case, there is a splitting 
\begin{equation}\label{eq:fixedsplitting}
  \bE|_{M^{\bT}} =  \bE^{f}\oplus \bE^m 
\end{equation}
into a fixed part $\bE^{f}$ and a moving part $\bE^m$. The latter defines $\bN\coloneqq (\bE^m)^\vee$ called the \textit{virtual normal bundle}, which only contains the parts of $\bE|_{M^{\bT}}^\vee$ of non-zero $\bT$-weight.\footnote{In our case of CY$4$ obstruction theories, the virtual normal bundle could equivalently be written as $\bE^m[-2]$.}

Choosing a fixed decomposition into $\bT$-weights\footnote{Usually, we choose a fixed ordering on $\bT$-weights and take positive and negative weights, which explains the choice of notation.}
\begin{equation*}
  \bN = \bN^{>}\oplus \bN^{<},
\end{equation*}
where $\bN^{<} = \big(\bN^{>}\big)^\vee[2]$ and the $\bT$-weights of $\bN^{>}$ and $\bN^{<}$ are disjoint, induces the orientation 
\begin{equation}
\label{eq:oNgeq}
o\big(\bN\big):\cO \to \det\big(\bN^{>}\big)\det\big(\bN^{<}\big)\cong\det(\bN)
\end{equation}
defined in \cite[Def. 3.14(iv) \& (3.9)]{Bo25}. 

Now the fixed part $\bE^f$ is an obstruction theory on $M^\bT$. We now use the given orientation $o: \cO \to \bE$ on $M$ and the orientation $o\big(\bN\big)$ from \eqref{eq:oNgeq} above to define an orientation $o^f$ on $\bE^f$ by
\begin{equation}
\label{eq:EEonMGGmtrivialization}
\cO \xrightarrow{o|_{M^{\bT}}} \det\big(\bE|_{M^{\bT}}\big) \xrightarrow{\epsilon_{\bN,\bE^f}^{-1}}\det(\bN)\det\big(\bE^f\big)\xrightarrow{o(\bN)^{-1}\otimes \id}\det(\bE^f).
\end{equation}
With this orientation, $\bE^f$ becomes an oriented CY4 obstruction theory \cite[(3.10)]{Bo25}, and hence we obtain the virtual fundamental class $\big[M^{\bT}\big]^{\vir}$ for $M^\bT$. 

\subsubsection{}
\label{sec:localization-formula}
Recall that the usual equivariant localization in K-theory is expressed using the equivariant K-theoretic Euler class $\fe_{\bT}(-)$ which is $\Lambda_{-1}(E^*)$ on vector bundles $E$. The localization formula for the virtual classes defined in Section~\ref{sp:sec:dt4-vir-class} from CY4 obstruction theories, instead contains the \textit{symmetrized Euler class} $\widehat{\fe}_{\bT}$, which is $\widehat{\fe}_{\bT}(\alpha) = \fe_{\bT}(\alpha)\det(\alpha)^{\frac{1}{2}}$ for any equivariant K-theory class $\alpha$. 

Now the virtual localization formula can be stated as follows.

\begin{theorem}[{\cite[Theorem 7.1]{Oh2023}}, {\cite[Proposition A.5]{Park21}, \cite[Theorem 3.5]{Bo25}}]\label{thm:eqvirloc}
   Let $M$ be a separated algebraic space with a $\bT$-action and $(\bE,\phi)$ a CY4 obstruction theory. The associated virtual class $\big[M\big]_{\bT}^{\vir}$ defined in Section~\ref{sp:sec:dt4-vir-class} and the virtual class $\big[M^{\bT}\big]^{\vir}$ for $M^\bT$ defined in Section~\ref{sp:sec:loc-setup} satisfy
    \begin{equation}
    \label{eq:vireqlocgeneral}
        [M]^{\vir}_{\bT} = \iota_*\frac{\big[M^{\bT}\big]^\vir}{\fe_{\bT}(\bN^{>})} \in H^{\bT}_*(M)_{\loc}\,.
    \end{equation}
    The analogous K-theoretic localization formula
    \begin{equation*}
      \widehat{\cO}_{M}^{\vir} = \iota_*\frac{\widehat{\cO}_{M^{\bT}}^{\vir}}{\widehat{\fe}_{\bT}(\bN^{>})}\in G^{\bT}_0(M)_{\loc} \,.
    \end{equation*}
    holds in equivariant K-theory.
\end{theorem}

The first version of this formula was originally proved in \cite[Theorem 7.1]{Oh2023}, with a more general version proved in \cite[Proposition A.5]{Park21}, and a careful treatment of orientations in \cite[Theorem 3.5]{Bo25}.

\subsection{General CY4 Pullback \& Localization}
\label{sec:generalCY4pullback}
All algebraic spaces in this section are assumed to be quasi-compact with affine diagonal.
\subsubsection{}

\begin{definition}
Let $\pi:\fN\rightarrow \fM$ a \textit{quasi-smooth} morphism in $\cat{Art}_\sT$ over a smooth equidimensional base $\fB$, so that we have an obstruction theory $\bE_\pi \stackrel{\phi_f}{\rightarrow} \bL_{\pi}$ with  $\bE_\pi$ being perfect of tor-amplitude $[-1,1]$. Additionally, let $\fM$ be given a CY$4$ obstruction theory $\phi_{\fM/\fB}:\bE_{\fM/\fB}\to \bL_{\fM/\fB}$. We call the following diagram \textit{Park's virtual pullback (PVP) diagram}, following \cite[(0.3)]{Park21}:
\begin{equation} \label{eq:pvp-diagram}
  \begin{tikzcd}
    \bE_\pi[-1] \ar[equals]{d} \ar{r} & \bF^\vee[2] \ar{d}{\eta^\vee[2]} \ar{r}{\zeta} & \widehat{\bE}_{\fN/\fB} \ar{d}{\zeta^\vee[2]} \ar{r}{+1} & {} \\
    \bE_\pi[-1] \ar{d}{{\phi_\pi[-1]}} \ar{r} & \pi^*\bE_{\fM/\fB} \ar{d}{\pi^*\phi_{\fM/\fB}} \ar{r}{\eta} & \bF \ar{d}{\psi} \ar{r}{+1} & {} \\
    \bL_\pi[-1] \ar{r} & \pi^*\bL_{\fM/\fB} \ar{r} & \bL_{\fN/\fB} \ar{r}{+1} & {}
  \end{tikzcd}
\end{equation}
where the horizontal rows are distinguished triangles, the vertical arrows induce morphisms between triangles, $\psi$ is an obstruction theory, and we use the symmetry of $\bE_{\fM/\fB}$. We also assume that $\widehat{\bE}_{\fN/\fB}$ is given a fixed self-dual isomorphism $\widehat{\bE}_{\fN/\fB}\cong \widehat{\bE}_{\fN/\fB}^\vee[2]$ compatible with the diagram. If, additionally, $\psi\circ \zeta^\vee[2]: \widehat{\bE}_{\fN/\fB} \to \bL_{\fN/\fB}$ is a CY$4$ obstruction theory, then we say it is a \textit{CY$4$ pullback} of the CY$4$ obstruction theory $\phi_{\fM/\fB}:\bE_{\fM/\fB}\to \bL_{\fM/\fB}$. 
\end{definition}

\subsubsection{}

\begin{remark}\label{sp:rmk-pvpinf}
Note that we can also work with the above PVP diagram in the setting of stable $\infty$-categories as explained in \cite[Appendix A]{Bo25}. Working in this setting takes into account (higher) homotopies filling up all faces and the interior of the diagram, which ends up determining the duality of $\widehat{\bE}_{\fN/\fB}$ uniquely and can be used to prove the isotropy condition for it even in the case of stacks. Details of this process, which we omit here, are worked out by the first-named author in \cite[Prop. A.15 \& Lemma A.16]{Bo25}. Due to the assumptions in §\ref{sec:assab}, our starting CY4 obstruction theories will always have natural $\infty$-lifts in the sense of \cite[Definition A.14, p. 20]{Bo25}. We can thus work in $\infty$-stable categories whenever necessary.
\end{remark}
\subsubsection{}
\label{sec:pvpinfty}
The additional benefit of lifting to stable $\infty$-categories is that one can prove functoriality of \eqref{eq:pvp-diagram}. Consider a commutative diagram 
\begin{equation}
\label{eq:functorialitydiagram}
 \begin{tikzcd}
     \fN_2\arrow[r,"\pi_2"] \arrow[rr,bend left = 50, "\pi"]&\fN_1\arrow[r, "\pi_1"]&\fM
 \end{tikzcd}\,.
\end{equation}
of quasi-smooth morphisms in $\cat{Art}_\sT$ over $\fB$. As in \cite[(A.21)]{Bo25}, suppose that there is an $\infty$-commutative diagram
\begin{equation}
\label{eq:startingpointPVPfunctoriality}
\begin{tikzcd}[column sep=small]
  \pi^*_2 \arrow[rr,bend left = 50, blue] \arrow[d]\bE_{\pi_1}[-1]\arrow[r, blue]&  \arrow[d]\bE_{\pi}[-1]\arrow[r, blue]& \pi^*\bE_{\fM}\arrow[d]\\
   \arrow[rr,bend right = 50]  \pi^*_2\bL_{\pi_1}[-1]\arrow[r]&  \bL_{\pi}[-1]\arrow[r]& \pi^*\bL_{\fM}&{}
\end{tikzcd}
\end{equation}
where vertical arrows determine the corresponding obstruction theories. If a PVP diagram \eqref{eq:pvp-diagram} is given for $\pi_1$ and $\pi_2$ on the level of stable $\infty$-categories, then this data produces uniquely (up to contractible choices) another such diagram for $\pi$. This and similar such results are formulated and proved in \cite[§A.5]{Bo25}. We use them repeatedly below to show that the classes constructed in Theorem \ref{sp:thm:symm-pullback-localization} are independent of choices and that they satisfy compatibilities from Proposition \ref{prop:comparison} and \ref{prop:smoothpushforward}. These results are then applied in the proof of wall-crossing in Section~\ref{sec:horizontalflagWC},  Section~\ref{sec:flagprojection}, and Section~\ref{sec:generalindependence}.
\subsubsection{}
\begin{definition}\label{def:orpullback}
Suppose that we are given a pullback of CY4 obstruction theories specified by \eqref{eq:pvp-diagram}, and suppose  further that there is an orientation $o(\bE_{\fM/\fB}):\cO \xrightarrow{\sim} \det(\bE_{\fM/\fB})$. Then there is an induced orientation $o(\widehat{\bE}_{\fN/\fB})$ of $\widehat{\bE}_{\fN/\fB}$ determined by the PVP diagram, defined as the composition of the consecutive morphisms
\begin{equation}\label{sp:eq:pvp-orientation}
\begin{tikzpicture}[descr/.style={fill=white,inner sep=1.5pt}]
      \matrix (m) [
           matrix of math nodes,
           row sep=2.2em,
            column sep=3.5em,
            text height=1.5ex, text depth=0.25ex
        ]
   {  \cO & \det\big(\bE_\pi^\vee[2]\big)\det\big(\bE_\pi\big) & \\
   \det\big(\bE_\pi^\vee[2]\big)\det\big(\pi^*\bE_{\fM/\fB}\big)\det\big(\bE_\pi\big) &  \det\big(\bE_\pi^\vee[2]\big)\det\big(\bF\big) & \\
            \det\big(\widehat{\bE}_{\fN/\fB}\big)\,, &  & \\
        };
    \path[overlay,->, font=\scriptsize,>=latex]
        (m-1-1) edge node[midway, above] {$o(\bE_\pi)$} (m-1-2) 
        (m-1-2) edge[out=345,in=165]  node[descr,yshift=0.3ex] {$\id\otimes \pi^*o(\bE_{\fM/\fB})\otimes \id$}  (m-2-1)
        (m-2-1) edge node[midway, above] {$\id\otimes \epsilon_{\pi^*\bE,\bE_\pi}$} (m-2-2)
        (m-2-2) edge[out=345,in=165]  node[descr,yshift=0.3ex] {$\epsilon_{\bE_\pi^\vee[2], \bF}$} (m-3-1);
\end{tikzpicture}
\end{equation}
where $o(\bE_\pi)$ is the obvious orientation \cite[Def. 3.14.iv)]{Bo25} and the $\epsilon$-isomorphisms are the isomorphisms of determinant induced by distinguished triangles, as discussed in \cite[(3.31)]{Bo25}.
\end{definition}
In light of functoriality of \eqref{eq:pvp-diagram} summarized in Section~\ref{sec:pvpinfty}, we may ask whether these orientations are compatible with respect to it. 
\begin{lemma}[{\cite[Lemma 3.10]{Bo25}}]
\label{lem:functorialityorientations}
    In the situation of \eqref{eq:functorialitydiagram}, applying \eqref{sp:eq:pvp-orientation} to the PVP diagram for $\pi$ induces the same orientation as the consecutive application of \eqref{sp:eq:pvp-orientation} to the PVP diagrams of $\pi_1$ and $\pi_2$.
\end{lemma}
\subsubsection{}
Here, we discuss a particular situation when \eqref{eq:startingpointPVPfunctoriality} exists automatically.
\begin{lemma}
\label{lem:smoothcompositionPVPstart}
Suppose that the morphisms of algebraic stacks in \eqref{eq:functorialitydiagram} are smooth.
Let $\bE_{\fM}\to \bL_{\fM}$ be an obstruction theory\footnote{More precisely its $\infty$-lift in the sense of Remark \ref{sp:rmk-pvpinf}.}. Then there exists an essentially unique $\infty$-commutative diagram \eqref{eq:startingpointPVPfunctoriality} with $\bE_{\pi} = \bL_{\pi}$ and $\bE_{\pi_1} = \bL_{\pi_1}$. 
\end{lemma}
\begin{proof}
    Illusie's construction gives the morphism $\bL_{\pi}[-1]\to \pi^*(\bL_{\fM})$ in the stable $\infty$-category. Let $C$ be the cone of $\pi^*(\bE_{\fM})\to\pi^*(\bL_{\fM})$. Then since $$\Ext^1(\bL_{\pi}, C) =  \Ext^0(\bL_{\pi}, C)=\Ext^1(\bL_{\pi}, C) = 0\,,$$ the topological space $\underline{\Hom}\big(\bL_{\pi}[-1],C\big)$ of morphisms in the stable $\infty$-category is 2-connected. This gives a $1$-connected space of null-homotopies of the composition $\bL^\vee_{\pi}[-1]\to \pi^*(\bL_{\fM})\to C$ and therefore a 1-connected space of the required maps $\bL^\vee_{\pi}[-1]\to \pi^*(\bE_{\fM})$. Since the left square in \eqref{eq:startingpointPVPfunctoriality} is automatic as its vertical arrows are identities, we obtain the diagram \eqref{eq:startingpointPVPfunctoriality} where all 2-dimensional faces homotopy commute except for the {\color{blue} blue face}. By the above argument applied to $\pi_1$ instead of $\pi$, the space of its homotopies is connected, so we obtain the full diagram in an essentially unique way.
\end{proof}

\subsubsection{}

\begin{lemma}[{\cite[Proposition A.15]{Bo25}, \cite[§2.5]{LiuVW}, \cite[Lemma 5.4]{kuhn-spin}, \cite[2.5.12]{KLT25}}]\label{sp:lemma:affine-symm-pb}
Let $\pi:A\to \fM$ be a smooth morphism from an affine scheme $A$ to an algebraic stack $\fM$ with a CY$4$ obstruction theory $\phi_{\fM}:\bE_{\fM}\to \bL_{\fM}$. Then there is a self-dual commutative diagram
\begin{equation}\label{eq:Jouanolousquare}
  \begin{tikzcd}
    \arrow[d]\Omega_{\pi}[-1]\arrow[r]&\pi^*(\bE_\fM)\arrow[d]\\
    0\arrow[r]&\Omega_{\pi}^\vee[3]
  \end{tikzcd}
\end{equation}
in $D^b(A)$, and hence a CY4 obstruction theory $\widehat{\bE}_A$ that fits into the PVP-diagram
\begin{equation}\label{eq:Parkconst}
        \begin{tikzcd}
\Omega_{\pi}[-1]\arrow[d,equal]\arrow[r]&\bF_A^\vee[2]\arrow[d]\arrow[r]&\arrow[d]\widehat{\bE}_A\\
            \arrow[d,equal]\Omega_{\pi}[-1]\arrow[r]&\arrow[d]\pi^*(\bE_\fM)\arrow[r]&\arrow[d]\bF_A\\
\Omega_{\pi}[-1]\arrow[r]&\pi^*(\bL_{\fM}) \arrow[r,equal]&\bL_{A}
        \end{tikzcd}
\end{equation}
in $D^b(A)$, with orientation induced from the given orientation on $\fM$ as explained in Definition~\ref{def:orpullback}. In fact, the stronger version explained in Remark~\ref{sp:rmk-pvpinf} holds, which is required to prove the cone isotropy property and proves sufficient uniqueness of the diagram \eqref{eq:Parkconst}.
\end{lemma}	
\begin{proof}
As in Lemma \ref{lem:smoothcompositionPVPstart}, we obtain a 1-connected space of morphisms $\Omega_{\pi}[-1]\to \pi^*(\bE_{\fM})$. Because $\Ext^i\big(\Omega_{\pi}[-1],\Omega^*_{\pi}[3]\big) = 0$ for $i\geq-3$, the space $\underline{\Hom}\big(\Omega_{\pi_a}[-1],\Omega^*_{\pi_a}[3]\big)$ is 3-connected, so any such morphism admits an essentially unique null-homotopy $\mu$ which we can symmetrize by $(\mu + \mu^\vee[2])/2$. Thus, we have obtained \eqref{eq:Jouanolousquare} and can apply \cite[Proposition A.15]{Bo25} to it to construct the diagram \eqref{eq:Parkconst}.
\end{proof}
\subsubsection{}
The next result is important to obtain compatibilities between the 
\subsubsection{}

Given a map $f:M\to \fM$, where $M$ is an algebraic space and $\fM$ has an oriented CY4 obstruction theory, there does not need to be a CY4-pullback of the obstruction theory on $\fM$ along $f$. However, if $f$ is smooth and $M$ has the resolution property, we can still construct a natural virtual cycle on $M$ that recovers the Oh-Thomas class when a virtual pull-back diagram exists and that satisfies the expected localization formula in an equivariant setting.   

\begin{theorem}\label{sp:thm:symm-pullback-localization}
Write $\bT\coloneqq \bG_m\times \sT$. Let $\fM$ be an algebraic $\bT$-stack with trivial $\bG_m$-action. Let $\bE_\fM$ be a $\bT$-equivariant CY$4$ obstruction theory on $\fM$. Let $\pi:M \to \fM$ be a $\bT$-equivariant smooth morphism from a quasi-compact, separated algebraic space with a $\bT$-action. Assume that $M$ has the $\bT$-equivariant resolution property. Then
  \begin{enumerate}[label = (\alph*)]
		\item There exists a natural CY$4$ virtual pullback $K$-theory class (and an analogous class in Chow)
		\begin{equation}
    \label{eq:OvirJouanolou}[\widehat{\mathcal{O}}_{M}^{\vir}]\in G^{\bT}_0\big(M,\bZ[2^{-1}]\big)
    \end{equation} 
    constructed using an equivariant Jouanolou device $a:A\to M$ and the orientation from Definition~\ref{def:orpullback} and Section~\ref{sec:OvirJouanoloudef}. Moreover, there is an induced class
    $$\qquad [\tilde{\mathcal{O}}_{M^{\bG_m}}^{\vir}]\in G_0^\bT\big(M^{\bG_m},\bZ[2^{-1}]\big)$$
    for which the orientations are determined in Section~\ref{sp:sec:weight-splitting-orientation} following Section~\ref{sp:sec:loc-setup}. Both classes are independent of the choice of a Jouanolou device.
    \item Suppose that there exists a PVP diagram \eqref{eq:pvp-diagram} for $\pi:M\to \fM$ such that $\phi_{\pi}:\bE_{\pi}\to \Omega_{\pi}$ is an isomorphism. Then the class \eqref{eq:OvirJouanolou} is identified with the one constructed using the resulting CY4 obstruction theory $\widehat{\bE}_M$ of $M$. In this case, the class $[\tilde{\mathcal{O}}_{M^{\bG_m}}^{\vir}]$ is determined by the CY4 obstruction theory $\widehat{\bE}_M|^f_{M^{\bG_m}}$ with the orientation \eqref{eq:EEonMGGmtrivialization}.
		\item Let $\iota:M^{\bG_m}\hookrightarrow M$ be the embedding of the fixed-point locus. Denote the moving part of the complex
    \begin{equation*}
      \left(\pi^*\bE_\fM^\vee\oplus \Omega_\pi[-2]\oplus T_\pi\right)|_{M^\bGm}\cong \left(\pi^*\bE_\fM\oplus \Omega_{\pi}\oplus T_{\pi}[2]\right)|_{M^{\bG_m}}[-2]
    \end{equation*}
    by $\bN_{\iota} = \bN^{>}_{\iota}\oplus \bN^{<}_{\iota}$ where the splitting is into positive and negatives weights. Then the localization formula
    \begin{equation}\label{sp:eq:pullback-localization-formula}
      \widehat{\cO}_{M}^{\vir} = \iota_*\,\frac{\tilde{\cO}_{M^{\bG_m}}^\vir}{\widehat{\fe}_{\bT}(\bN^{>}_{\iota})}
    \end{equation}
	  and its analog in Chow hold.
	\end{enumerate} 
\end{theorem}

\begin{remark}
  In the proof of the wall-crossing formula, we encounter master spaces $M$ where $M^{\bG_m}$ can be equipped with a CY$4$ virtual pullback class from \eqref{eq:OvirJouanolou}. In this case, the construction of the classes $\big[\widehat{\mathcal{O}}_{M^\bGm}^{\vir}\big]$ and $\big[\tilde{\cO}_{M^{\bG_m}}^\vir\big]$ differs, which is why we chose different notation. We study their relation in our situations of interest explicitly in Proposition~\ref{prop:comparison} below.
\end{remark}

\subsubsection{}
\label{sec:OvirJouanoloudef}
\begin{proof}
  Let us begin by assuming the situation in Theorem \ref{sp:thm:symm-pullback-localization} (b). We then consider a ($\bT$-equivariant) Jouanolou device $a:A\to M$, the existence of which is guaranteed by the resolution property for algebraic spaces (see e.g.,  \cite[\S 2.5.3]{KLT23}, \cite{Jong}). We construct a PVP diagram along $a: A\to M$ from $\widehat{\bE}_M$ using Lemma~\ref{sp:lemma:affine-symm-pb} so that the resulting CY4 obstruction theory $\widehat{\bE}_A$ on $A$ produces a class $[\widehat{\mathcal{O}}_A^{\vir}]$. Due to \cite[Theorem B.3]{Park21}, it satisfies 
  \begin{equation}\label{sp:eq:jsp-vir-class}
    [\widehat{\cO}_M^{\vir}]:= (a^*)^{-1} \left(\big[\widehat{\mathcal{O}}_A^{\vir}\big] \otimes \left(\det \Omega_{ a}\right)^{-\frac{1}{2}} \right)\ \in G^{\bT}_0\big(M,\bZ[2^{-1}]\big)\,.
  \end{equation}
  
If we no longer require the existence of a PVP diagram along $\pi$, we can instead apply Lemma~\ref{sp:lemma:affine-symm-pb} to the smooth map $\pi_a\coloneqq \pi\circ a$ to construct a CY$4$ pullback obstruction theory on $A$. We now simply turn \eqref{sp:eq:jsp-vir-class} into the definition of $[\widehat{\cO}_M^{\vir}]$ which may not exist otherwise. 
\subsubsection{}
Let us assume again that the assumption of Theorem \ref{sp:thm:symm-pullback-localization} (b) holds. By Lemma \ref{lem:smoothcompositionPVPstart}, there is an essentially unique starting diagram \eqref{eq:startingpointPVPfunctoriality}.  As explained in \S \ref{sec:pvpinfty}, we may use \cite[Theorem A.18]{Bo25} to compose the PVP diagrams along $\pi$ and $a$ with their upper two rows given by 
$$
  \begin{tikzcd}[column sep=small]
    \wt{\bF}_{\pi}^\vee[2] \ar{d} \ar{r} & \widehat{\bE}_M \ar{d} \ar{r} & \Omega_{\pi} \ar[equals]{d} \ar{r}{+1} & {}  \\
    \pi^*\bE_{\fM} \ar{r} & \wt{\bF}_{\pi} \ar{r} & \Omega_{\pi} \ar{r}{+1} & {}   
  \end{tikzcd}\qquad \textnormal{and}\qquad  \begin{tikzcd}[column sep=small]
    \tilde{\bF}_{a}^\vee[2] \ar{d} \ar{r} & \widehat{\bE}_{A} \ar{d} \ar{r} & \Omega_{a} \ar[equals]{d} \ar{r}{+1} & {}  \\
    \widehat{\bE}_{M} \ar{r} & \tilde{\bF}_{a} \ar{r} & \Omega_{a} \ar{r}{+1} & {}   
  \end{tikzcd}\,,
$$
respectively. By the uniqueness in Lemma \ref{sp:lemma:affine-symm-pb}, the resulting obstruction theory $\widehat{\bE}_{A}$ coincides with the one constructed using CY4 pullback along $\pi_a$.  Therefore the two constructions of $[\widehat{\mathcal{O}}_M^{\vir}]$ from \S \ref{sec:OvirJouanoloudef} produce the same class and the first part of Theorem \ref{sp:thm:symm-pullback-localization} (b) holds.
\subsubsection{}\label{sp:sec:ind-of-choices}
  Before proving the desired localization formula for this class, we discuss its independence of choices. Specifically, let $a_1:A_1\to M$ and $a_2:A_2\to M$ be two different Jouanolou devices. To show that
  \begin{equation}
  \label{eq:OvirJouanolouqual}
    (a_1^*)^{-1} \left(\big[\widehat{\cO}_{A_1}^{\vir}\big] \otimes \left(\det \Omega_{a_1}\right)^{-\frac{1}{2}} \right) = (a_2^*)^{-1} \left(\big[\widehat{\cO}_{A_2}^{\vir}\big] \otimes \left(\det \Omega_{a_2}\right)^{-\frac{1}{2}} \right),
  \end{equation}
  i.e., the two definitions of $[\widehat{\cO}_M^{\vir}]$  coincide, we consider the fiber product
  \begin{equation}
  \label{eq:A_1A_2A_3}
    \begin{tikzcd}
      A_3\arrow[d,"c_2"]\arrow[r,"c_1"]\arrow[dr,"a_3"]& A_1\arrow[d,"a_1"]\\
      A_2\arrow[r,"a_2"]&M.
    \end{tikzcd}
  \end{equation}
  
  Now $A_3$ is another Jouanolou device over $A_1$, $A_2$, and $M$. This allows us to apply Lemma~\ref{lem:smoothcompositionPVPstart} and Lemma Lemma~\ref{sp:lemma:affine-symm-pb} to both $\pi_{a_3}\coloneqq \pi\circ a_3$ and $c_i$ to get compatible (in the sense of \eqref{eq:startingpointPVPfunctoriality}) self dual diagrams
  \begin{equation*}
    \begin{tikzcd}
      \arrow[d]\Omega_{\pi_{a_3}}[-1]\arrow[r]&\pi^*_{a_3}(\bE_\fM)\arrow[d] & \arrow[d]\Omega_{c_i}[-1]\arrow[r]&c_i^*(\widehat{\bE}_{A_i})\arrow[d] \\
      0\arrow[r]&\Omega_{\pi_{a_3}}^\vee[3], & 0\arrow[r]&\Omega_{c_i}^\vee[3].
      \end{tikzcd}
  \end{equation*}
  By combining functoriality from \cite[Theorem A.18.ii)]{Bo25} with uniqueness of \eqref{eq:Parkconst}, we see that they yield the same CY$4$ obstruction theory. The induced orientations are identified by Lemma \ref{lem:functorialityorientations}. Then \cite[Theorem B.3]{Park21} implies that
  \begin{equation*}
    \big[\widehat{\cO}^{\vir}_{A_3}\big] = c_i^* \big[\widehat{\cO}^{\vir}_{A_i}\big]\otimes \det(\Omega_{c_i})^{\frac{1}{2}}.
  \end{equation*}
  Using this for $i=1,2$, we can compute
  \begin{align*}
    (a_3^*)^{-1}\left(\big[\widehat{\cO}^{\vir}_{A_3}\big]\otimes \det(\Omega_{a_3})^{-\frac{1}{2}}\right)
    =& (a_i^*)^{-1}(c_i^*)^{-1}\Big(c_i^* \big[\widehat{\cO}^{\vir}_{A_i}\big]\otimes \det(\Omega_{c_i})^{\frac{1}{2}} \otimes \det(\Omega_{a_i\circ c_i})^{-\frac{1}{2}}\Big) \\
    =& (a_i^*)^{-1}\Big(\big[\widehat{\cO}^{\vir}_{A_i}\big]\otimes \det(\Omega_{a_i})^{-\frac{1}{2}}\Big) \, ,
  \end{align*}
  which proves \eqref{eq:OvirJouanolouqual}.

\subsubsection{}\label{sp:sec:localization-class}
In order to prove the localization formula for the virtual class defined in \eqref{sp:eq:jsp-vir-class}, take the $\bT$-equivariant Jouanolou device $a:A\to M$ used to construct $\widehat{\bE}_A$ on $A$ as in Lemma~\ref{sp:lemma:affine-symm-pb}. By equivariance, we have an induced morphism of fixed loci $a^{\bG_m}:A^{\bG_m}\to M^{\bG_m}$ which is itself a $\sT$-equivariant Jouanolou device. Set $B:= a^{-1}(M^{\bG_m})\xrightarrow{b}M^{\bG_m}$ and consider the commutative diagram
\begin{equation}\label{eq:equivariant-jouanolou-cd}
    \begin{tikzcd}
		A^\bGm\ar[hookrightarrow]{r}{\alpha}\ar[bend left]{rr}{\iota_A}\ar{dr}[swap]{a^\bGm} & B \ar{d}{b}\ar[hookrightarrow]{r}{\beta} &A\ar{d}[swap]{a}\ar[bend left=40]{dd}{\pi_a}\\
		\,&M^\bGm\ar[hookrightarrow]{r}{\iota}& M\ar{d}[swap]{\pi} \\
		\, & & \fM\,,
	\end{tikzcd}
\end{equation}
where $\alpha:A^{\bG_m}\hookrightarrow B$ and $\iota_A:A^{\bG_m}\hookrightarrow A$ are the inclusions of the fixed-point locus into $B$ and $A$, respectively. Using Lemma~\ref{sp:lemma:affine-symm-pb}, we construct the CY$4$ pullback of $\bE_\fM$ along $\pi_a$ to obtain the PVP diagram \eqref{eq:Parkconst}
\begin{equation}\label{eq:localization-pvp}
  \begin{tikzcd}[column sep=small]
    \bF_{\pi_a}^\vee[2] \ar{d} \ar{r} & \widehat{\bE}_{A} \ar{d} \ar{r} & \Omega_{\pi_a} \ar[equals]{d} \ar{r}{+1} & {}  \\
    \pi_a^*\bE_{\fM} \ar{r} & \bF_{\pi_a} \ar{r} & \Omega_{\pi_a} \ar{r}{+1} & {}   
  \end{tikzcd}
\end{equation}
defining the obstruction theory $\widehat{\bE}_A$. A minor modification of the functoriality of PVP diagrams from \S \ref{sec:pvpinfty} allows us to construct $\widehat{\bE}_A$ in two steps instead. We first construct a PVP diagram \eqref{eq:pvp-diagram}
\begin{equation}
\label{eq:tildeEEMdiagram}
  \begin{tikzcd}[column sep=small]
    \tilde{\bF}_{\pi_a}^\vee[2] \ar{d} \ar{r} & \tilde{\bE}_{M}^A \ar{d} \ar{r} & a^*\Omega_{\pi} \ar[equals]{d} \ar{r}{+1} & {}  \\
    \pi_a^*\bE_{\fM} \ar{r} & \tilde{\bF}_{\pi_a} \ar{r} & a^*\Omega_{\pi} \ar{r}{+1} & {}   
  \end{tikzcd}\qquad \textnormal{from}\qquad 
  \begin{tikzcd}
    \arrow[d]a^*\Omega_{\pi}[-1]\arrow[r]&\pi_a^*(\bE_\fM)\arrow[d]\\
    0\arrow[r]&a^*\Omega_{\pi}^\vee[3].
  \end{tikzcd}
\end{equation}
Then, one can combine Lemma \ref{lem:smoothcompositionPVPstart} with the distinguished triangle
$$
a^*\Omega_{\pi} \to \Omega_{\pi_a} \to \Omega_{a} \xrightarrow{+1}
$$
to construct the natural map $\Omega_a[-1]\to \widetilde{\bE}_{M}^A$ (see \cite[(A.35)]{Bo25}). This gives us $\widehat{\bE}_A$ in \eqref{eq:localization-pvp} from
\begin{equation}\label{sp:eq:two-step-localization-pvp}
  \begin{tikzcd}[column sep=small]
    \tilde{\bF}_{a}^\vee[2] \ar{d} \ar{r} & \widehat{\bE}_{A} \ar{d} \ar{r} & \Omega_{a} \ar[equals]{d} \ar{r}{+1} & {}  \\
    \tilde{\bE}_{M}^A \ar{r} & \tilde{\bF}_{a} \ar{r} & \Omega_{a} \ar{r}{+1} & {}   
  \end{tikzcd}\qquad \textnormal{from}\qquad
  \begin{tikzcd}
    \arrow[d]\Omega_{a}[-1]\arrow[r]&\widetilde{\bE}_{M}^A\arrow[d]\\
    0\arrow[r]&\Omega_{a}^\vee[3]
  \end{tikzcd}
\end{equation}
using functoriality. We want to study the fixed and moving parts of \eqref{eq:localization-pvp} pulled back to $A^\bGm$. Taking the fixed part gives us an obstruction theory $\widehat{\bE}_A^f\coloneqq \widehat{\bE}_A|_{A^{\bGm}}^f$ on $A^\bGm$, which we now want to induce an orientation on using \eqref{eq:EEonMGGmtrivialization}.

\subsubsection{}\label{sp:sec:weight-splitting-orientation}
Because the square in \eqref{eq:equivariant-jouanolou-cd} is cartesian, we see that $\Omega_a|_{A^{\bG_m}} = \alpha^*\Omega_b$. Moreover, there is a distinguished triangle
$$
N_{\alpha}^\vee\to \alpha^*\Omega_b\to \Omega_{a^{\bG_m}}\xrightarrow{+1}\,.
$$
Taking its moving and fixed parts, we find that 
\begin{equation}
\label{eq:fixedmovingOmega-a}
\Omega_a|^m_{A^{\bG_m}}\cong N^\vee_{\alpha}\,,\qquad\textnormal{and}\qquad\Omega_a|^f_{A^{\bG_m}}\cong \Omega_{a^{\bG_m}}\,.
\end{equation}
Set $\overline{\bN}_{\iota} = \widetilde{\bE}_M^A|^m_{A^{\bG_m}}$, then taking the moving parts of \eqref{sp:eq:two-step-localization-pvp} restricted to $A^\bGm$ produces 
\begin{equation}\label{eq:localization-pvp-normal-bundle-N-alpha}
  \begin{tikzcd}
    \overline{\bF}^\vee[2] \ar{d} \ar{r} & \bN_{\iota_A} \ar{d} \ar{r} & N_{\alpha}^\vee \ar[equals]{d} \ar{r}{+1} & {} \\
    \overline{\bN}_\iota \ar{r} & \overline{\bF} \ar{r} & N_{\alpha}^\vee \ar{r}{+1} & {.}
  \end{tikzcd}
\end{equation}
By construction, the K-theory class of $\overline{\bN}_\iota$ is
\begin{equation}\label{sp:eq:lifted-Ni-k-class}
  \big[\overline{\bN}_\iota\big]=\big[(a^{\bG_m})^*\bN_{\iota}[2]\big]
\end{equation}
where $\bN_\iota[2]=\left(\pi^*\bE\oplus \Omega_{\pi}\oplus T_{\pi}[2]\right)|_{M^{\bG_m}}^m$ as defined in the statement of Theorem~\ref{sp:thm:symm-pullback-localization}(b) above. Note that the localization formula only depends on the K-theory class of $\bN$, and the shift $[2]$ above hence doesn't affect the formula. Matching the convention for $\bN_\iota$, we split $\bN_{\iota_A}$ into $\bN_{\iota_A}^>$ and $\bN_{\iota_A}^<$, containing positive and negative $\bG_m$-weights, respectively. Splitting \eqref{eq:localization-pvp-normal-bundle-N-alpha} into positive and negative weights, we see that our convention corresponds to choosing the positive weights of both $N_\alpha[2]$ and $N_\alpha^\vee$ to be part of $\bN_{\iota_A}^>$. This orientation in turn induces an orientation on $\widehat{\bE}_{A}^f$ using \eqref{eq:EEonMGGmtrivialization}, allowing us to define the associated Oh--Thomas class $\big[\tilde{\mathcal{O}}_{A^{\bG_m}}^{\vir}\big]$.	We then set
\begin{equation*}
  \big[\tilde{\mathcal{O}}_{M^{\bG_m}}^{\vir}\big]\coloneqq (b^*)^{-1} \alpha_*\left(\frac{[\tilde{\mathcal{O}}_{A^{\bG_m}}^{\vir}]}{\fe_{\bT}(N_{\alpha}^>+(N_{\alpha}^<)^\vee)}\otimes (\det \Omega_{a^{\bG_m}})^{-\frac{1}{2}}\right)\in G^{\bT}_{0}(M^{\bG_m})_{\loc}.
\end{equation*}
Note here that $\fe_{\bT}(N_{\alpha}^>+(N_{\alpha}^<)^\vee)=(-1)^{n_\alpha^<}\fe_{\bT}(N_{\alpha})$ by properties of the Euler class, where $n_\alpha^<\coloneqq \rk(N_\alpha^<)$. We note here that
\begin{equation}\label{sp:eq:tilde-class-equation-M-fixed-pt-locus}
  (a^\bGm)^*[\tilde{\mathcal{O}}_{M^\bGm}^{\vir}] \otimes (\det \Omega_{a^\bGm})^{\frac{1}{2}} = (-1)^{n_\alpha^<}[\tilde{\mathcal{O}}_{A^\bGm}^{\vir}].
\end{equation}
To see this, one just needs to use that $a^\bGm = b\circ \alpha$ and $\alpha^*\alpha_* = \otimes \fe_{\bT}(N_{\alpha})$.
\subsubsection{}
Suppose now that the condition of Theorem~\ref{sp:thm:symm-pullback-localization}~(b) is satisfied, so that \eqref{eq:tildeEEMdiagram} is $a^*(-)$ of a PVP diagram on $M$ producing $\wh{\bE}_M$. Taking the fixed part of \eqref{sp:eq:two-step-localization-pvp}, we obtain 
$$
 \begin{tikzcd}[column sep=small]
    \tilde{\bF}_{a^{\bG_m}}^\vee[2] \ar{d} \ar{r} & \widehat{\bE}^f_{A} \ar{d} \ar{r} & \Omega_{a^{\bG_m}} \ar[equals]{d} \ar{r}{+1} & {}  \\
    (a^{\bG_m})^*\wh{\bE}^f_{M} \ar{r} & \tilde{\bF}_{a^{\bG_m}} \ar{r} & \Omega_{a^{\bG_m}} \ar{r}{+1} & {}   
  \end{tikzcd}
$$
where $\wh{\bE}^f_{M}:=\wh{\bE}_{M}|^f_{M^{\bG_m}}$ with the orientation induced by \eqref{eq:EEonMGGmtrivialization}. Using  $\wh{\bE}^f_{M}$, we now directly construct $[\tilde{\mathcal{O}}_{M^\bGm}^{\vir}] $.

From \ref{sp:sec:weight-splitting-orientation} we see that the orientation on $\widehat{\bE}^f_{A}$ induced by \eqref{sp:eq:pvp-orientation} differs from the one obtained from applying \eqref{eq:EEonMGGmtrivialization} to $\wh{\bE}_{A}$ precisely by the sign $(-1)^{n^<_{\alpha}}$. Thus we have proved \eqref{sp:eq:tilde-class-equation-M-fixed-pt-locus} using \cite[Theorem B.3]{Park21} which identifies the two constructions of $[\tilde{\mathcal{O}}_{M^\bGm}^{\vir}] $.
\subsubsection{}
\label{sec:tildeOindependence}
To prove that $\big[\tilde{\mathcal{O}}_{M^{\bG_m}}^{\vir}\big]$ is independent of the choice of $a:A\to M$, consider the Jouanolou devices  $a_1:A_1\to M$, $a_2:A_2\to M$, and $a_3:A_3\to M$ from \eqref{eq:A_1A_2A_3}. Taking further fiber products $B_3=B_1\times_{M^\bGm}B_2$ and $A_3^\bGm=A_1^\bGm\times_{M^\bGm}A_2^\bGm$, we obtain the following diagram 
\begin{equation*}
  \begin{tikzcd}
    A_3^\bGm\ar[r,hookrightarrow,"\alpha_3"]\ar[d,"c_i^\bGm"]\ar[ddr,bend right=85,looseness=1.3,swap,"a_3^\bGm"] & B_3\ar[r,hookrightarrow,"\beta_3"]\ar[d,"c_i^B"] & A_3\ar[d,"c_i"]\ar[dd,bend left=40,"a_3"] \\
    A_i^\bGm\ar[r,hookrightarrow,"\alpha_i"]\ar[dr,"a_i^\bGm"] & B_i\ar[r,hookrightarrow,"\beta_i"]\ar[d,"b_i"] & A_i\ar[d,"a_i"] \\
    & M^\bGm\ar[r,hookrightarrow,"\iota"] & M\ar[d,"\pi"] \\
    & & \fM
  \end{tikzcd}
\end{equation*}
where only the two squares on the right are Cartesian. Again, as in Section~\ref{sp:sec:ind-of-choices}, using \cite[Theorem A.18.ii)]{Bo25}, we find that the CY$4$ pullback of $\widehat{\bE}_{A_i}$ along $c_i$ gives us the same CY$4$ obstruction theory as $\widehat{\bE}_{A_3}$. Thus, restricting to $A_3^\bGm$ and taking fixed parts tells us that $\widehat{\bE}_{A_3}^f$ and the CY$4$ pullback of $\widehat{\bE}_{A_i}^f$ along $c_i^\bGm$ coincide as CY$4$ obstruction theories except for possibly the chosen orientations. To compute the difference, we have to compare the orientations of the corresponding moving parts. The orientation of $\bN_{\iota_{A_3}}$ used in defining $\big[\tilde{\mathcal{O}}_{A_3^{\bG_m}}^{\vir}\big]$ is the one obtained by splitting \eqref{eq:localization-pvp-normal-bundle-N-alpha} for $A_3$ and $\alpha_3$ into positive and negative weight parts. In particular, the pieces $N_{\alpha}$ are also split this way. 

On the other hand, if we consider the PVP diagram for the CY4 pullback of $\widehat{\bE}_{A_i}$ along $c_i$, the moving part of its restriction to $A^{\bG_m}_3$ is given by
\begin{equation}\label{sp:eq:normal-compatibility-pvp}
  \begin{tikzcd}
    \bF_{c_i}^\vee|_{A_3^\bGm}^m[2] \ar{d} \ar{r} & \bN_{\iota_{A_3}} \ar{d} \ar{r} & \Omega_{c_i}|_{A_3^\bGm}^m \ar[equals]{d} \ar{r}{+1} & {} \\
    (c_i^\bGm)^*\bN_{\iota_{A_i}} \ar{r} & \bF_{c_i}|_{A_3^\bGm}^m \ar{r} & \Omega_{c_i}|_{A_3^\bGm}^m \ar{r}{+1} & {.}
  \end{tikzcd}
\end{equation}
 Note that $\Omega_{c_i} = c^*_j\Omega_{a_j}$ where $j\neq i$ is the other index, so by \eqref{eq:fixedmovingOmega-a}, we conclude that
\begin{equation*}
  \Omega_{c_i}|_{A_3^\bGm}^m\cong (c_j^\bGm)^* \Omega_{a_j}|_{A_j^\bGm}^m\cong (c_j^\bGm)^*N_{\alpha_j}^\vee\,.
\end{equation*}
We consider the orientation of $\bN_{\iota_{A_3}}$ determined by \eqref{eq:localization-pvp-normal-bundle-N-alpha} and induce another orientation of $\bN_{\iota_{A_3}}$ by applying \eqref{sp:eq:pvp-orientation} to \eqref{sp:eq:normal-compatibility-pvp}. This makes all of $(c_j^\bGm)^*N_{\alpha_j}^\vee$ positive. Write $\bar{\bN}_{\iota,k}$ for the bottom-left corner of  \eqref{eq:localization-pvp-normal-bundle-N-alpha} constructed on $A_k$. Then $\bar{\bN}_{\iota,3} \cong c^*_i \bar{\bN}_{\iota,i}$, and their orientations are identified. Using that $N_{\alpha_3} = (c^{\bG_m}_1)^*N_{\alpha_1}\oplus (c^{\bG_m}_2)^*N_{\alpha_2}$, we see that the current orientation differs from the one for $\big[\tilde{\mathcal{O}}_{A_3^{\bG_m}}^{\vir}\big]$ by the sign $(-1)^{n_{\alpha_j}^<}$, using the same procedure as in Section~\ref{sp:sec:compatibility-normal-orientation-comparison}. By \cite[Theorem B.3]{Park21}, we hence get
\begin{equation*}
  \big[\tilde{\cO}^{\vir}_{A_3^\bGm}\big] = (-1)^{n_{\alpha_j}^<}(c_i^\bGm)^* \big[\tilde{\cO}^{\vir}_{A_i^\bGm}\big]\otimes \det(\Omega_{c_i^\bGm})^{\frac{1}{2}}\,.
\end{equation*}
Inserting this equation into \eqref{sp:eq:tilde-class-equation-M-fixed-pt-locus}, we obtain
\begin{align*}
  &(-1)^{n_{\alpha_3}^<}\big((a_3^\bGm)^*\big)^{-1}\left(\big[\tilde{\cO}^{\vir}_{A_3^\bGm}\big]\otimes \det\big(\Omega_{a_3^\bGm}\big)^{-\frac{1}{2}}\right)\\
  \cong& (-1)^{n_{\alpha_3}^<+n_{\alpha_j}^<} \big((a_3^\bGm)^*\big)^{-1}\left((c_i^\bGm)^* \big[\tilde{\cO}^{\vir}_{A_i^\bGm}\big]\otimes \det(\Omega_{c_i^\bGm})^{\frac{1}{2}}\otimes \det\big(\Omega_{a_3^\bGm}\big)^{-\frac{1}{2}}\right)\\
  \cong& (-1)^{n_{\alpha_i}^<} \big((a_i^\bGm)^*\big)^{-1}\big((c_i^\bGm)^*\big)^{-1}\left((c_i^\bGm)^* \big[\tilde{\cO}^{\vir}_{A_i^\bGm}\big]\otimes (c_i^\bGm)^*\det(\Omega_{a_i^\bGm})^{-\frac{1}{2}}\right)\\
  \cong& (-1)^{n_{\alpha_i}^<} \big((a_i^\bGm)^*\big)^{-1}\left(\big[\tilde{\cO}^{\vir}_{A_i^\bGm}\big]\otimes \det(\Omega_{a_i^\bGm})^{-\frac{1}{2}}\right)\,.
\end{align*}
Hence, $\big[\tilde{\mathcal{O}}_{M^{\bG_m}}^{\vir}\big]$ is independent of the chosen Jouanolou device by \eqref{sp:eq:tilde-class-equation-M-fixed-pt-locus}.

\subsubsection{}
To show that 
\begin{equation}\label{eq:affinelocalizationtext}
  [\widehat{\mathcal{O}}_M^{\vir}] = \iota_* \frac{[\tilde{\mathcal{O}}_{M^{\bG_m}}^{\vir}]}{\widehat{\fe}_{\bT}\big(\bN^{>}_\iota\big) },
\end{equation}
we first focus on the right-hand side and compute
\begin{align*}
    \frac{[\tilde{\mathcal{O}}_{M^{\bG_m}}^{\vir}]}{\widehat{\fe}_{\bT}\big(\bN^{>}_\iota\big) }& = (-1)^{n_\alpha^<}(b^{*})^{-1}\alpha_*\left(\frac{[\tilde{\mathcal{O}}_{A^{\bG_m}}^{\vir}]}{\widehat{\fe}_{\bT}\big((a^{\bG_m})^*\bN^{>}_{\iota}\big)\fe_{\bT}(N_{\alpha})}\otimes (\det \Omega_{a^{\bG_m}})^{-\frac{1}{2}}\right)\\
    & = (-1)^{n_\alpha^<}(b^{*})^{-1}\left(\alpha_*\left(\frac{[\tilde{\mathcal{O}}_{A^{\bG_m}}^{\vir}]}{\widehat{\fe}_{\bT}\big((a^{\bG_m})^*\bN^{>}_{\iota}\big)\widehat{\fe}_{\bT}(N_{\alpha})}\right)\otimes (\det \Omega_{b})^{-\frac{1}{2}}\right)\\
    & = (b^{*})^{-1}\left(\alpha_*\left(\frac{[\tilde{\mathcal{O}}_{A^{\bG_m}}^{\vir}]}{\widehat{\fe}_{\bT}\big(\bN^{>}_{\iota_A}\big)}\right)\otimes (\det \Omega_{b})^{-\frac{1}{2}}\right).
\end{align*}
The second equality follows from $\det(\Omega_{a^{\bG_m}})^{-\frac{1}{2}} = \det(\Omega_{b}|_{A^{\bG_m}})^{-\frac{1}{2}}\det(N_{\alpha})^{-\frac{1}{2}}$, and the last one uses the construction of $\bN_{\iota_A}$ in \eqref{eq:localization-pvp-normal-bundle-N-alpha} and \eqref{sp:eq:lifted-Ni-k-class} above, as well as $\widehat{\fe}_{\bT}(N_{\alpha}^>+(N_{\alpha}^<)^\vee)=(-1)^{n_\alpha^<}\widehat{\fe}_{\bT}(N_{\alpha})$. Now apply $a^*\iota_*$ to the above and use flat base-change $\iota_* (b^{*})^{-1}=(a^*)^{-1}\beta_*$ to get 
\begin{equation*}
  (\iota_A)_*\left(\frac{[\tilde{\mathcal{O}}_{A^{\bG_m}}^{\vir}]}{\widehat{\fe}_{\bT}\big(\bN^{>}_{\iota_A}\big)}\right)\otimes(\det \Omega_{a})^{-\frac{1}{2}}.
\end{equation*}
Applying $a^*$ to the left-hand side of \eqref{eq:affinelocalizationtext} yields $[\widehat{\mathcal{O}}_A^{\vir}]\otimes (\det \Omega_a)^{-\frac{1}{2}}$ by construction. Hence, the claim reduces to the localization formula
\begin{equation*}
  [\widehat{\mathcal{O}}_A^{\vir} ] = (\iota_A)_* \frac{[\tilde{\mathcal{O}}_{A^{\bG_m}}^{\vir}]}{\widehat{\fe}\big(\bN^{>}_{\iota_A}\big) }
\end{equation*}
from Theorem~\ref{thm:eqvirloc}.
\end{proof}

\subsection{Compatibility}
\subsubsection{}
In the proof of the wall-crossing formula, we will encounter situations where a master space $M$ with a $\bG_m$-action has components $Z$ of the fixed-point locus, so that both $M$ and $Z$ have natural smooth maps to potentially different stacks with compatible CY4 obstruction theories. In these cases, we want to have a mechanism to compare the induced virtual classes, including their orientations.

\begin{proposition}\label{prop:comparison}
	Let $\fM$ and $\fN$ be algebraic $\bT$-stacks with trivial $\bG_m$-action and $\psi:\fN\to \fM$ a $\bT$-equivariant morphism. For a smooth  $\bT$-equivariant map $\pi:M\to\fM$ from a quasi-compact, separated algebraic  space $M$ consider the commutative diagram
  \begin{equation*}
    \begin{tikzcd}
      Z \ar{d}{\nu} \ar[hookrightarrow]{r}{\iota} & M \ar[loop right]{r}{\bG_m} \ar{d}{\pi} \\
      \fN\ar{r}{\psi} & \fM,
    \end{tikzcd}
  \end{equation*}
  where $\iota$ is a closed embedding of connected components of a $\bG_m$-fixed point locus and $\nu$ is another $\bT$-equivariant smooth morphism. We again assume that $M$ has the $\bT$-equivariant resolution property. 
  
  Let $\bE_\fM$, $\bE_{\fN}$ be $\bT$-equivariant CY$4$ obstruction theories on $\fM$ and $\fN$, respectively. Suppose that the composition
  \begin{equation}\label{sp:comp-ass:cotangent-isom}
    \Omega_\pi|_Z^f \to \iota^*\Omega_\pi \to \Omega_\nu
  \end{equation}
  is an isomorphism. Further, we require the existence of a morphism $\delta:\psi^*(\bE_\fM)\to \bE_\fN$ making
  \begin{equation}\label{sp:comp-ass:bE-isom}
    \pi^*(\bE_\fM)|_Z^f \to \iota^*\pi^*(\bE_\fM) \xrightarrow{\sim} \nu^*\psi^*(\bE_\fM) \xrightarrow{\nu^*(\delta)} \nu^*(\bE_\fN)
  \end{equation}
  into an isomorphism of symmetric complexes\footnote{We usually deal with the case where $\bE_\fN$ is a summand of $\psi^*\bE_\fM$ and the pairing on $\bE_\fN$ is a restriction of the one on $\psi^*\bE_\fM$. In this case, we only need to check that \eqref{sp:comp-ass:bE-isom} is an isomorphism, since the symmetry is preserved automatically.} in the sense of Definition~\ref{sp:def:obstruction-theory} up to orientations. Under the isomorphism \eqref{sp:comp-ass:bE-isom}, the orientation on $\pi^*(\bE_\fM)|_Z^f$ induced by splitting $\pi^*\bE_\fM|_Z^m$ into positive and negative weight parts and using \eqref{eq:EEonMGGmtrivialization}, differs from the orientation of $\nu^*(\bE_{\fN})$ by a fixed sign $\varepsilon_Z$.
  
  In Theorem~\ref{sp:thm:symm-pullback-localization}, we constructed two virtual classes for $Z$:
  \begin{itemize}
  \item $\big[\widehat{\mathcal{O}}_Z^{\vir}\big]$ obtained by applying \eqref{sp:eq:jsp-vir-class} to $\nu$; and
  \item $\big[\tilde{\mathcal{O}}_{Z}^{\vir}\big]$ from Section~\ref{sp:sec:weight-splitting-orientation}.
  \end{itemize}
  We have the following equality of virtual classes:
	\begin{equation}\label{sp:eq:master-space-vir-class-comp-sign}
    \big[\widehat{\mathcal{O}}_Z^{\vir}\big] = \varepsilon_Z(-1)^{\omega_\pi^>}\big[\tilde{\mathcal{O}}_{Z}^{\vir}\big],
  \end{equation}
  where $\omega_\pi^>\coloneqq \rk(\Omega_\pi|_Z^>)$.
\end{proposition}
\begin{remark}
    Note that the extra sign $(-1)^{\omega^{>}_\pi}$ in \eqref{sp:eq:master-space-vir-class-comp-sign} is needed to recover the correct wall-crossing formula as in the presence of obstruction theories on master spaces, this sign coincides with the one computed in \cite[Lemma 8.4]{Bo25}. It follows from Theorem \ref{sp:thm:symm-pullback-localization} (b) that \eqref{sp:eq:master-space-vir-class-comp-sign} holds also in the presense of CY4 obstruction theories thus recovering the aforementioned result.
\end{remark}
\subsubsection{}

\begin{proof}
Let $a:A\to M$ be a $\bG_m$-equivariant Jouanolou device. As in \cite[\S 2.5.3]{KLT23} and \eqref{eq:equivariant-jouanolou-cd}, this gives us a commutative diagram
\begin{equation}\label{sp:eq:eq-jouanolou-diagram-Z}
	\begin{tikzcd}
		A_Z\ar[hookrightarrow]{r}{\alpha}\ar[bend left]{rr}{\iota_A}\ar{dr}{a_Z} & B_Z \ar{d}{b_Z}\ar[hookrightarrow]{r}{\beta} &A\ar{d}{a}\\
		\,&Z\ar[hookrightarrow]{r}{\iota}\ar{d}{\nu}& M\ar{d}{\pi} \\
		\,&\fN \ar{r}{\psi} & \fM
	\end{tikzcd}
\end{equation}
with a cartesian top-right square and $A_Z$ the $\bG_m$-fixed point locus of $B_Z$. As shown in Section~\ref{sp:sec:ind-of-choices}, the class $\big[\widehat{\mathcal{O}}_Z^{\vir}\big]$ is independent of the choice of a Jouanolou device, so we may assume it is constructed using $\nu_a\coloneqq \nu\circ a_Z$. Consider now the upper two rows of the PVP diagrams \eqref{eq:pvp-diagram} for $\pi_{a}\coloneqq \pi\circ a$ and $\nu_a$:
\begin{equation}\label{eq:localization-pvp-Z}
  \begin{tikzcd}[column sep=small]
    \bF_{\pi_a}^\vee[2] \ar{d} \ar{r} & \widehat{\bE}_{A} \ar{d} \ar{r} & \Omega_{\pi_a} \ar[equals]{d} \ar{r}{+1} & {} 
    & \bF_{\nu_a}^\vee[2] \ar{d} \ar{r} & \widehat{\bE}_{A_Z} \ar{d} \ar{r} & \Omega_{\nu_a} \ar[equals]{d} \ar{r}{+1} & {}
    \\
    \pi_a^*\bE_{\fM} \ar{r} & \bF_{\pi_a} \ar{r} & \Omega_{\pi_a} \ar{r}{+1} & {,} 
    & \nu_a^*\bE_{\fN} \ar{r} & \bF_{\nu_a} \ar{r} & \Omega_{\nu_a} \ar{r}{+1} & {.} 
  \end{tikzcd}
\end{equation}
The obstruction theory $\widehat{\bE}_{A_Z}$ comes with a natural orientation from \eqref{sp:eq:pvp-orientation} applied to the second diagram of \eqref{eq:localization-pvp-Z}, which makes it a CY$4$ obstruction theory and hence gives us the associated Oh--Thomas class $\big[\widehat{\mathcal{O}}_{A_Z}^{\vir}\big]$.	This is then used to define the class of interest
\begin{equation}\label{sp:eq:hat-class-equation}
  \big[\widehat{\mathcal{O}}_{Z}^{\vir}\big]:= \big(a_Z^*\big)^{-1} \left([\widehat{\mathcal{O}}_{A_Z}^{\vir}]\otimes \det (\Omega_{a_Z})^{-\frac{1}{2}}\right)\in G^{\bT}_{0}\big(Z, \bZ[2^{-1}]\big).
\end{equation}
By section~\ref{sp:sec:ind-of-choices}, this class $[\widehat{\mathcal{O}}_{Z}^{\vir}]$ is independent of the choice of $A$.

\subsubsection{}\label{sp:sec:compatibility-Ef-setup}
Recall our construction of $\widehat{\bE}_A$ in two steps in Section~\ref{sp:sec:localization-class}. Consider the second diagrams in \eqref{eq:tildeEEMdiagram} and \eqref{sp:eq:two-step-localization-pvp}, which combine to produce $\widehat{\bE}_A$. Restricting to $A_Z$, taking fixed parts, and using \eqref{sp:comp-ass:cotangent-isom}, \eqref{sp:comp-ass:bE-isom}, and \eqref{eq:fixedmovingOmega-a} we obtain
\begin{equation*}
 \begin{tikzcd}
    \arrow[d]a_Z^*\Omega_{\nu}[-1]\arrow[r]&\nu_a^*(\bE_\fN)\arrow[d]\\
    0\arrow[r]&a_Z^*\Omega_{\nu}^\vee[3] , 
  \end{tikzcd}\qquad\textnormal{and}\qquad 
  \begin{tikzcd}
    \arrow[d]\Omega_{a_Z}[-1]\arrow[r]&\widetilde{\bE}_{A_Z}^f\coloneqq \widetilde{\bE}_M^A|_{A_Z}^f\arrow[d]\\
    0\arrow[r]&\Omega_{a_Z}^\vee[3].
  \end{tikzcd}
\end{equation*}
 By functoriality \cite[Theorem A.18.ii)]{Bo25} and uniqueness from Lemma \ref{sp:lemma:affine-symm-pb}, this gives us an isomorphism
$$
\widehat{\bE}_{A}^f:=\widehat{\bE}_{A}|^f_{A_Z}\cong \widehat{\bE}_{A_Z}
$$
where $\widehat{\bE}_{A_Z}$ is obtained by applying Lemma \ref{sp:lemma:affine-symm-pb} to $\nu_a$. This is an isomorphism of CY4 obstruction theories up to orientations.
\subsubsection{}
Next, we will study the difference between the two orientations. For this, take the PVP orientations of $\widehat{\bE}_{A}$ and $\widehat{\bE}_{A_Z}$, obtained by applying \eqref{sp:eq:pvp-orientation} to the left-hand side and the right-hand side of \eqref{eq:localization-pvp-Z} respectively. Inserting them into \eqref{eq:EEonMGGmtrivialization} yields an orientation $o_{\mathrm{PVP}}(\bN_{\iota_A})$ on $\bN_{\iota_A}$.

Now recall from Section~\ref{sp:sec:weight-splitting-orientation}, that we considered another orientation on $\bN_{\iota_A}$ obtained by splitting \eqref{eq:localization-pvp-normal-bundle-N-alpha} into positive and negative $\bGm$-weight parts. We will call it $o_{\mathrm{w}}(\bN_{\iota_A})$. For the same orientation of $\widehat{\bE}_A$ as before, it induces another orientation of $\widehat{\bE}_{A_Z}$ using \eqref{sp:eq:pvp-orientation}.

The resulting orientation on $\widehat{\bE}_A^f$ allowed us to define the associated Oh--Thomas class $\big[\tilde{\mathcal{O}}_{A_Z}^{\vir}\big]$, and we set
\begin{equation*}
  \big[\tilde{\mathcal{O}}_{Z}^{\vir}\big]= (-1)^{n_\alpha^<}(b_Z^*)^{-1} \alpha_*\left(\frac{[\tilde{\mathcal{O}}_{A_Z}^{\vir}]}{\fe_{\bT}(N_{\alpha})}\otimes (\det \Omega_{a_Z})^{-\frac{1}{2}}\right)\in G^{\bT}_{0}(Z)_{\loc}.
\end{equation*}
Writing out \eqref{sp:eq:tilde-class-equation-M-fixed-pt-locus} for the component $Z$ we have
\begin{equation}\label{sp:eq:tilde-class-equation}
  a_Z^*[\tilde{\mathcal{O}}_{Z}^{\vir}] \otimes (\det \Omega_{a_Z})^{\frac{1}{2}} = (-1)^{n_\alpha^<}[\tilde{\mathcal{O}}_{A_Z}^{\vir}].
\end{equation} 
Comparing \eqref{sp:eq:tilde-class-equation} to \eqref{sp:eq:hat-class-equation}, it suffices to find the appropriate sign relating $ (-1)^{n_\alpha^<}[\tilde{\mathcal{O}}_{A_Z}^{\vir}]$ to $[\widehat{\mathcal{O}}_{A_Z}^{\vir}]$. This sign is determined by the difference of orientations on $\widehat{\bE}_A^f$, constructed using \eqref{eq:EEonMGGmtrivialization}, and hence the difference between $o_{\PVP}(\bN_{\iota_A})$ and $o_{\mathrm{w}}(\bN_{\iota_A})$.

\subsubsection{}
\label{sp:sec:compatibility-normal-orientation-comparison}
To understand $o_{\mathrm{w}}(\bN_{\iota_A})$ we need to split  \eqref{eq:localization-pvp-normal-bundle-N-alpha} into positive and negative weights. Unraveling the definition of $\overline{\bN}_\iota$ using \eqref{eq:tildeEEMdiagram}, we then obtain
\begin{equation}\label{sp:eq:bN-weights-decomposition}
  \begin{tikzcd}[row sep=0pt]
    \left(\overline{\bF}^\lessgtr\right)^\vee[2] \ar[r] & \bN_{\iota_A}^\gtrless \ar[r] & \left(N_\alpha^\lessgtr\right)^\vee\ar[r,"+1"] & {}\\
    N_\alpha^\gtrless[2] \ar[r] & \left(\overline{\bF}^\lessgtr\right)^\vee[2] \ar[r] & \left(\overline{\bN}_\iota^\lessgtr\right)^\vee[2]\ar[r,"+1"] & {}\\
    \left(a^*\Omega_\pi|_{A_Z}^\lessgtr\right)^\vee[2] \ar[r] & \left(\overline{\bN}_\iota^\lessgtr\right)^\vee[2] \ar[r] & \tilde{\bF}_{\pi_a}|_{A_Z}^\gtrless\ar[r,"+1"] & {}\\
    \pi_a^*\bE_\fM|_{A_Z}^\gtrless \ar[r] & \tilde{\bF}_{\pi_a}|_{A_Z}^\gtrless \ar[r] & a^*\Omega_\pi|_{A_Z}^\gtrless\ar[r,"+1"] & {.}
  \end{tikzcd}
\end{equation}
This is used to express $o_{\mathrm{w}}(\bN_{\iota_A})$ in explicit factors in the lower half of the diagram below. In the upper half, we use \eqref{eq:localization-pvp-normal-bundle-N-alpha} and \eqref{eq:tildeEEMdiagram} to describe $o_{\mathrm{PVP}}(\bN_{\iota_A})$ in explicit factors. This allows us to write out both orientations explicitly in the diagram below and compare signs by checking up to which sign the diagram commutes\footnote{A similar diagram for the specific set-up used in \S\ref{sec:horizontalflagWC}, when obstruction theories are present, appeared in \cite[(8.8)]{Bo25}}.
\begin{equation*}
  \begin{tikzcd}[column sep=huge]
    & \begin{array}{c}
      \det(N_\alpha)\det(a^*\Omega_\pi^\vee|_{A_Z}^m)\det(\pi_a^*\bE_\fM|_{A_Z}^m)  \\
     \colorbox{lightorange}{\(\displaystyle \det(a^*\Omega_\pi|_{A_Z}^m)\det(N_\alpha^\vee)\)}
    \end{array} \ar[ddr,bend left,"o_{\mathrm{PVP}}(\bN_{\iota_A})"{pos=0.3},"\eqref{eq:localization-pvp-normal-bundle-N-alpha}\& \eqref{eq:tildeEEMdiagram}"'{pos=0.3}] & \\
    & \begin{array}{c}
      \det(N_\alpha^>)\det(N_\alpha^<)\det((a^*\Omega_\pi|_{A_Z}^<)^\vee)\det((a^*\Omega_\pi|_{A_Z}^>)^\vee)\\
      \det(\pi_a^*\bE_\fM|_{A_Z}^>)\det(\pi_a^*\bE_\fM|_{A_Z}^<)\colorbox{lightorange}{\(\displaystyle \det(a^*\Omega_\pi|_{A_Z}^>)\)}\\
      \colorbox{lightorange}{\(\displaystyle\det(a^*\Omega_\pi|_{A_Z}^<)\det((N_\alpha^<)^\vee)\det((N_\alpha^>)^\vee)\)}
    \end{array}\arrow[dr, dash, draw=blue, line width=3pt, opacity=0.5]\ar[dr,line width=0.75pt,swap,"\Xi^\mathrm{PVP}"]\arrow[u, dash, draw=orange, line width=3pt, opacity=0.5]\ar[u,line width=0.75pt,"\eqref{eq:oNgeq}","\gamma_{\mathrm{PVP}}"'] & \\
    \cO \arrow[uur, bend left, dash, draw=orange, line width=3pt, opacity=0.5]\ar[uur, bend left, line width=0.75pt, "o_{\mathrm{PVP}}(\bN_{\iota_A})"{pos=0.7}]\ar[ur,dash,draw=orange,line width=2pt,shift left=1pt,opacity=0.6]\ar[ur,dash,draw=blue,line width=2pt,shift right=1pt,opacity=0.6]\ar[ur,line width=0.75pt,swap,"\eqref{eq:oNgeq}"]\arrow[dr, dash, draw=blue, line width=3pt, opacity=0.5]\ar[dr,line width=0.75pt,"\Xi^\mathrm{w}_1"]\ar[ddr, bend right, swap, "o_{\mathrm{w}}(\bN_{\iota_A})"{pos=0.4}] & & \det(\bN_{\iota_A})\\
    & \begin{array}{c}
      \det(N_\alpha^>)\det((a^*\Omega_\pi|_{A_Z}^<)^\vee)\det(\pi_a^*\bE_\fM|_{A_Z}^>)\\
      \det(a^*\Omega_\pi|_{A_Z}^>)\det((N_\alpha^<)^\vee)\colorbox{lightblue}{\(\displaystyle \det(N_\alpha^<)\det((a^*\Omega_\pi|_{A_Z}^>)^\vee)\)}\\
      \det(\pi_a^*\bE_\fM|_{A_Z}^<)\colorbox{lightblue}{\(\displaystyle \det(a^*\Omega_\pi|_{A_Z}^<)\det((N_\alpha^>)^\vee)\)}
    \end{array}\arrow[ur, dash, draw=blue, line width=3pt, opacity=0.5]\ar[ur,line width=0.75pt,"\Xi^\mathrm{w}_2"]& \\
    & \det(\bN_{\iota_A}^>)\det(\bN_{\iota_A}^<)\ar[u,"\eqref{sp:eq:bN-weights-decomposition}",,"\gamma_{\mathrm{w}}"']\ar[uur, bend right, "\eqref{eq:oNgeq}"{pos=0.7}, "o_{\mathrm{w}}(\bN_{\iota_A})"'{pos=0.65}] & 
  \end{tikzcd}
\end{equation*}
Here $\Xi^{\mathrm{PVP}}$ (resp. $\Xi_i^{\mathrm{w}}$) are defined from $o_{\mathrm{PVP}}(\bN_{\iota_A})$ (resp. $o_{\mathrm{w}}(\bN_{\iota_A})$) via the isomorphisms $\gamma_{\mathrm{PVP}}$ (resp. $\gamma_{\mathrm{w}}$). That also means that the diagram commutes by construction with the exception of the top-left {\color{orange} orange triangle}, which commutes only up to $\varepsilon_Z$ by pulling back the orientation assumption along $a_Z$, and the {\color{blue} blue square} in the middle which, by commutation properties of determinants, only commutes up to the sign 
\begin{equation*}
  (-1)^{n_\alpha^<(n_\alpha^<+\omega_\pi^>+e^>+\omega_\pi^<)+\omega_\pi^>(n_\alpha^<+\omega_\pi^>+e^>)+e^<(\omega_\pi^>+n_\alpha^<)+n_\alpha^<\omega_\pi^<}=(-1)^{n_\alpha^<+\omega_\pi^>},
\end{equation*}
where $n_\alpha^\gtrless=\rk(N_\alpha^\gtrless)$, $\omega_\pi^\gtrless=\rk\big(a^*\Omega_\pi|_{A_Z}^\gtrless\big)=\rk(\Omega_\pi|_{Z}^\gtrless)$, and $e^\gtrless=\rk(\pi_a^*\bE_\fM|_{A_Z}^\gtrless)$, and the equality uses evenness of $\bE_\fM$. Alternatively, since reordering the positive parts and then doing so for the negative parts does not introduce any signs, we may focus on the contents of the \colorbox{lightorange}{orange box} compared to \colorbox{lightblue}{blue box} independent of their order. This shows that the square diagram commutes up to the sign $(-1)^{n_\alpha^<+\omega_\pi^>}$.

We hence get $[\widehat{\mathcal{O}}_{A_Z}^{\vir}]=\varepsilon_Z(-1)^{n_\alpha^<+\omega_\pi^>}[\tilde{\mathcal{O}}_{A_Z}^{\vir}]$, which yields
\begin{equation*}
  [\widehat{\mathcal{O}}_{Z}^{\vir}]=\varepsilon_Z(-1)^{\omega_\pi^>}[\tilde{\mathcal{O}}_{Z}^{\vir}]
\end{equation*}
by \eqref{sp:eq:hat-class-equation} and \eqref{sp:eq:tilde-class-equation}.

\end{proof}

\subsection{Flag Pushforward Formula}
\subsubsection{}
In this section, we only work $\sT$-equivariantly. Thus, we for example omit the $\bG_m$-equivariance in the situation of Theorem \ref{sp:thm:symm-pullback-localization} when constructing $[\widehat{\mathcal{O}}_{M}^{\vir}]$ from \eqref{eq:OvirJouanolou}. Below, we will encounter situations where there is an additional flag-bundle $p: N\to M$, and one needs to compare the classes of $N$ and $M$ constructed using Theorem \ref{sp:thm:symm-pullback-localization}. The next proposition addresses a slightly more general situation by providing the expected formula. 
\begin{proposition}
\label{prop:smoothpushforward}
    Let $\fM$ be an algebraic $\sT$-stack and $M,N$ quasi-compact, separated algebraic $\sT$-spaces satisfying $\sT$-equivariant resolution property. Let
    $$
    \begin{tikzcd}[column sep=5pt]
N\arrow[dr,"\rho"']\arrow[rr,"p"]&& M\arrow[dl,"\pi"]\\
&\fM&
    \end{tikzcd}
    $$
  be a $\sT$-equivariant commutative diagram where all morphisms are smooth and and $p$ is projective.  Suppose that $\fM$ has a $\sT$-equivariant CY4 obstruction theory $\bE_{\fM}$, so that Theorem \ref{sp:thm:symm-pullback-localization} applied to $\pi$ and $\rho$ gives rise to the classes 
  $$  [\widehat{\mathcal{O}}_{M}^{\vir}]\in G^{\sT}_0\big(M,\bZ[2^{-1}]\big)\qquad\textnormal{and}\qquad [\widehat{\mathcal{O}}_{N}^{\vir}]\in G^{\sT}_0\big(M,\bZ[2^{-1}]\big)\,,
  $$
  respectively. They satisfy 
  $$
  p^*[\widehat{\mathcal{O}}_{M}^{\vir}] \otimes \det(\Omega_p)^{\frac{1}{2}}= [\widehat{\mathcal{O}}_{N}^{\vir}]\,.
  $$

\end{proposition}
\begin{remark}
    Note that if $M$ admits a CY4 obstruction theory $\wh{\bE}_M$ coming from a PVP diagram \eqref{eq:pvp-diagram} with $\phi_{\pi}:\bE_{\pi}\to \Omega_{\pi}$ being an isomorphism, then Theorem \ref{sp:thm:symm-pullback-localization} (b) implies that $[\widehat{\mathcal{O}}_{N}^{\vir}]$ is constructed using $\wh{\bE}_M$. This also applies in the setting of the corollary below and will be used to relate the invariants $\sz_\alpha^{\Fr}(\tau)$ defined in Theorem \ref{thm:sst-invariants} to the ones defined in \cite[Definition 5.13]{Bo25} in the presence of Joyce--Song obstruction theories. This is precisely the situation from Definition \cite[Definition~5.3.a)]{Bo25} summarized in Example~\ref{ex:CY4JScategories}.
\end{remark}
\subsubsection{}
As already mentioned, we are mainly interested in situations when $p$ is a flag bundle over $M$. In this case, one immediately concludes the following.
\begin{corollary}
\label{cor:flag-pushforward}
  In the situation of Proposition \ref{prop:smoothpushforward}, let $p$ be a flag bundle with fibers of constant Euler characteristic $d$. The classes $[\widehat{\mathcal{O}}_{M}^{\vir}]$ and $[\widehat{\mathcal{O}}_{N}^{\vir}]$ satisfy 
$$
p_*\left([\widehat{\mathcal{O}}_{N}^{\vir}]\otimes \widehat{\fe}_{\sT}(T_p)\right)=d\cdot[\widehat{\mathcal{O}}_{M}^{\vir}]
$$
for any $\cK\in K^0_{\sT}(M)$.  
\end{corollary}
\begin{proof}
    Using Proposition \ref{prop:smoothpushforward}, this follows from the projection formula and 
    $
    p_*\big(\fe(T_p)\big) = d\cdot [\cO_M].
    $
\end{proof}

\subsubsection{}
\begin{proof}[Proof of Proposition~\ref{prop:smoothpushforward}]
We begin by setting up our Jouanolou devices. Take an equivariant Jouanolou device $a:A\to M$, and pull it back along $p$ to get an affine bundle $p^*A\to N$ an induced smooth and proper map $p^*A\to A$. Since the total space of $p^*A$ is not affine, we additionally take a Jouanolou device $C\to p^*A$ above it to obtain
\begin{equation}
\label{eq:compatibleJouanolousmooth}
   \begin{tikzcd}[column sep=5pt]
      C\ar[d]\ar[dd,bend right=50,"c"']\ar[drr,"p_A"] & & \\
      p^*A\arrow[d]\arrow[rr]&&A\arrow[d,"a"]   \\
      N\arrow[dr,"\rho"']\arrow[rr,"p"]&& M\arrow[dl,"\pi"]\\
      &\fM\, ,&
    \end{tikzcd}  
\end{equation}
where, by construction, $a$ and $c$ are Jouanolou devices and the map $p_A$ is smooth and quasi-projective. Since the classes $[\widehat{\mathcal{O}}_{M}^{\vir}]$ and $[\widehat{\mathcal{O}}_{N}^{\vir}]$ are independent of the choice of Jouanolou device by Theorem~\ref{sp:thm:symm-pullback-localization}, we may use $A$ and $C$ respectively to construct them.

\subsubsection{}  
Recall that the obstruction theories $\widehat{\bE}_A$ and $\widehat{\bE}_C$ defining the classes $[\widehat{\mathcal{O}}_{M}^{\vir}]$ and $[\widehat{\mathcal{O}}_{N}^{\vir}]$ are constructed by applying  Lemma~\ref{sp:lemma:affine-symm-pb} along $\pi_a\coloneqq \pi\circ a$ and $\rho_c\coloneqq \rho\circ c$ in \eqref{eq:compatibleJouanolousmooth}, respectively. Since all the maps in question are smooth, we have \eqref{eq:startingpointPVPfunctoriality} by Lemma \ref{sp:lemma:affine-symm-pb}. This allows us to use functoriality from \cite[Theorem A.18.ii)]{Bo25} which combined with uniqueness in Lemma \ref{sp:lemma:affine-symm-pb} shows that $\widehat{\bE}_C$ is recovered from the self-dual diagram 
$$
 \begin{tikzcd}
    \arrow[d]\Omega_{p_A}[-1]\arrow[r]&p^*_A(\widehat{\bE}_A)\arrow[d]\\
    0\arrow[r]&\Omega_{p_A}^\vee[3]
  \end{tikzcd}
$$
 constructed by using that $C$ is affine. The resulting two orientations coincide by Lemma \ref{lem:functorialityorientations}. By \cite[Theorem B.3]{Park21}, we then obtain
$$
p_A^*[\widehat{\mathcal{O}}_{A}^{\vir}] \otimes \det(\Omega_{p_A})^{\frac{1}{2}} = [\widehat{\mathcal{O}}_{C}^{\vir}]\,.
$$
Plugging \eqref{sp:eq:jsp-vir-class} into the left-hand side of this equality, we get
\begin{align*}
  &p_A^*\Big(a^*[\widehat{\mathcal{O}}_{M}^{\vir}] \otimes \det(\Omega_a)^{\frac{1}{2}}\Big)\otimes \det(\Omega_{p_A})^{\frac{1}{2}}  \\
  \cong \,&(p\circ c)^*[\widehat{\mathcal{O}}_{A}^{\vir}] \otimes \det(\Omega_{p\circ c})^{\frac{1}{2}}\\
 \cong \,& c^*\Big(p^*[\widehat{\mathcal{O}}_{M}^{\vir}] \otimes \det(\Omega_p)^{\frac{1}{2}}\Big)\otimes \det(\Omega_{c})^{\frac{1}{2}}\,.
\end{align*}
Comparing with the right-hand side, which is equal to $c^* \big[\widehat{\cO}^{\vir}_{N}\big]\otimes \det(\Omega_{c})^{\frac{1}{2}}$ by \eqref{sp:eq:jsp-vir-class}, we conclude the claim of the proposition.
\end{proof}

\section{Background}
\label{sec:background}
\subsection{CY4 Category $\cat{A}$}
\label{sec:CY4-category}
\subsubsection{}
The precise notion of an abelian Calabi--Yau four category $\cat{A}$ which gives rise to CY4 obstruction theories is explained in \cite[Definition 3.19]{Bo25} summarizing \cite{BD1, BDII}. Briefly, it is a heart of a homotopy category of a right CY4 dg-category $\cat{D}_{\operatorname{lper}}$ which consists of locally perfect objects in a pre-triangulated left CY4 dg-category $\cat{D}$ in the sense of \cite[§2.3]{BD1}. The condition that $\cat{A}\subset H^0(\cat{D}_{\operatorname{lper}})$ and $\cat{D}_{\operatorname{lper}}$ is right CY4 is exactly what implies Serre duality
$$
\RHom_{\cat{A}}(E,F)\simeq\RHom_{\cat{A}}(F,E)^\vee[-4]
$$
for any  $E,F \in \cat{A}$. 

 Let  $\fM_{\cat{A}}$ and $\fM_{\cat{D}}$ be the moduli stacks of $\cat{A}$ and $\cat{D}$, respectively.\footnote{These were defined as higher stacks or derived stacks in \cite{TV}.} We will always additionally assume that that the natural embedding
$$
\begin{tikzcd}
	i_{\cat{A}}: \fM_{\cat{A}}\arrow[r,hookrightarrow]& \fM_{\cat{D}}
\end{tikzcd}
$$
is open. Then the left CY4 structure on $\cat{D}$ ensures by \cite[Main theorem]{BDII} that $\fM_{\cat{A}}$ can be enriched to a $-2$-shifted symplectic stack and thus carries a CY4 obstruction theory. 
\subsubsection{}
\label{sec:assab}
\begin{definition}
\label{def:categoryA}
Let $\cat{A}$ be a noetherian CY4 abelian category. Suppose that there is an action of a torus $\sT$ on $\cat{A}$ compatible with the CY4 structure. By definition, there is an $\sT$-equivariant CY4 obstruction theory $\bE$ on $\fM_{\cat{A}}$
Choose a quotient $K^0_e(\cat{A})\twoheadrightarrow  \bar{K}(\cat{A})$ such that the Euler pairing on $\cat{A}$ induces a morphisms $\chi: \bar{K}(\cat{A})\times \bar{K}(\cat{A})\to \bZ$. The image of $C_e(\cat{A})$ in $ \bar{K}(\cat{A})$ will be denoted by $ \bar{C}(\cat{A})$. The next few points collect some notation for structures of $\fM_{\cat{A}}$.
\begin{enumerate}[label =(\alph*)]
    \item There are maps
\begin{equation}
\label{eq:murho}
\begin{tikzcd}
\Phi:\fM_{\cat{A}}\times \fM_{\cat{A}}\arrow[r] &\fM_{\cat{A}}\end{tikzcd}\,,\qquad\begin{tikzcd} \Psi: B\bG_m\times\fM_{\cat{A}}\arrow[r]& \fM_{\cat{A}}\end{tikzcd}\end{equation}
    corresponding respectively to taking direct sums and rescaling automorphisms of objects. They are restrictions of the corresponding maps on $\fM_{\cat{D}}$. The first map is $\sT$-equivariant with respect to the diagonal action on the source. The action $\Psi$ commutes with the action of $\sT$.
    \item  We will denote the \textit{rigidification} (see \cite[App. A]{AOV} or \cite{Romagny}) of $\fM_{\cat{A}}$ with respect to the action $\Psi$ by
    $$
   \fM^{\pl}_{\cat{A}} := \fM_{\cat{A}}\fatslash B\bG_m\,.
    $$
    It admits the induced $\sT$-action and the natural $\sT$-equivariant $B\bG_m$-torsor $\Pi:\fM_{\cat{A}}\to \fM^{\rig}_{\cat{A}}$.
   \item Consider the diagonal action of $\sT$ on $\fM_{\cat{A}}\times \fM_{\cat{A}}$ and the $\sT$-equivariant universal objects $\cE_{\cat{A}}, \cF_{\cat{A}}$ of the first and second copy, respectively. Set 
   $$
\mExt_{\cat{A}}=\RHom_{\fM_{\cat{A}}\times \fM_{\cat{A}}}(\cE_{\cat{A}},\cF_{\cat{A}})
   $$
   defined for any $\cat{A}$ as explained, for example, in \cite[Example 3.21]{Bo25}. Vertex algebras are then constructed using the $\sT$-equivariant complex
  \begin{equation} 
  \label{eq:ThetaofA}
\Theta_{\cat{A}} := \mExt_{\cat{A}}^\vee[-1]\,.
   \end{equation}
    \item Fix a set $\scE(\cat{A})\subset  \bar{K}(\cat{A})$ of \textit{emergent classes}\footnote{There are called \textit{permissible classes} in \cite[Assumption 1.13(c)]{Joyce2021} and \cite[\S 5.2.1]{KLT25}.} that will satisfy assumptions specified in Section~\ref{ass:sec:invariants-assumptions}. For each $\alpha\in \scE(\cat{A})$, we will denote by $\fM_{\alpha}$ the corresponding open and closed substack of $\fM_{\cat{A}}$ consisting of objects of class $\alpha$. Given a perfect complex or a K-theory class on $\fM_{\cat{A}}$, we will denote its restriction to $\fM_{\alpha}$ by appending a subscript $(-)_{\alpha}$. More generally, this will apply to products of $\fM_{\cat{A}}$ where a restriction for each of the factors will be denoted by an additional subscript. 
 \end{enumerate}
\end{definition}
\subsubsection{}

The obstruction theory on $\fM_{\cat{A}}$ is
\begin{equation}
\label{eq:EE=Delta-Theta}
\bE = \Delta^*\Theta_{\cat{A}}\,.
\end{equation}
It induces the \textit{rigidified CY4 obstruction theory} $\bE^{\rig}$ on $\fM^{\rig}_{\cat{A}}$ by \cite[(127)]{BKP}\footnote{Compare this to \cite[Lemma 2.5.5]{KLT25} in the CY$3$ setting.} which describes a PVP diagram \eqref{eq:pvp-diagram} along $\Pi:\fM_{\cat{A}}\to \fM^{\rig}_{\cat{A}}$ with the upper two rows given by 
\begin{equation}
\label{eq:EE-rigidified}
 \begin{tikzcd}[column sep=small]
   \bF_{\Pi}^\vee[2] \ar{d} \ar{r} & \bE \ar{d} \ar{r} & \cO_{\fM_{\cat{A}}}[-1] \ar[equals]{d} \ar{r}{+1} & {}  \\
  \bE^{\rig} \ar{r} & \bF_{\Pi}\arrow[r] & \cO_{\fM_{\cat{A}}}[-1]\ar{r}{+1} & {}   
  \end{tikzcd}\,.
\end{equation}
\subsubsection{}
The CY$4$ pullback virtual class constructions in Section~\ref{sec:sym-pullback}, used in the proof of the wall-crossing formula, \`a priori depend on whether we pull back the obstruction theory on $\fM_{\cat{A}}$ or its rigidified obstruction theory on $\fM_{\cat{A}}^\pl$. The following result however tells us that the constructions using $\bE^{\rig}$ and $\bE$ coincide.\footnote{A more general statement appeared in \cite[Lemma A.19]{Bo25}.}
\begin{lemma}
\label{lem:Mrig-and-M-no-difference}
    Suppose that we are in the situation of Theorem \ref{sp:thm:symm-pullback-localization} where  $\fM = \fM_{\cat{A}}$ has the above obstruction theory. Then the classes $[\widehat{\mathcal{O}}_{M}^{\vir}]$ and $ [\tilde{\mathcal{O}}_{M^{\bG_m}}^{\vir}]$ constructed using Theorem \ref{sp:thm:symm-pullback-localization} (a) are equal to the ones we would obtain by applying the same construction to 
    $$
    \Pi\circ \pi: M\to \fM^{\rig}
    $$
    and the CY4 obstruction theory $\bE^{\rig}$.
\end{lemma}
\begin{proof}
    Consider the diagram 
    $$
     \begin{tikzcd}
     A\arrow[r,"\pi_a"] \arrow[rr,bend left = 50, "\Pi\circ \pi_a"]&\fM_{\cat{A}} \arrow[r, "\Pi"]&\fM^{\rig}_{\cat{A}}
 \end{tikzcd}\,,
    $$
where $a:A\to M$ is the Jouanolou device from \eqref{sec:OvirJouanoloudef}. Applying Lemma \ref{sp:lemma:affine-symm-pb} to $\pi_a$ and $\Pi\circ\pi_a$ we obtain the corresponding PVP diagrams with two different obstruction theories $\wh{\bE}_A$ and $\wh{\bE}^{\rig}_A$ on $A$, respectively. Because $\pi_a$ and $\Pi$ are smooth, we can apply Lemma \ref{lem:smoothcompositionPVPstart} to obtain \eqref{eq:startingpointPVPfunctoriality}. Therefore, the functoriality of PVP diagrams from \cite[Theorem A.18.ii)]{Bo25} can be used to identify $\wh{\bE}_A = \wh{\bE}^{\rig}_A$ including their orientations as in Lemma \ref{lem:functorialityorientations}. By the constructions in §\ref{sec:OvirJouanoloudef} and §\ref{sp:sec:weight-splitting-orientation}, this already implies the statement.
\end{proof}
\subsubsection{}
For future uses, we will also note down the definition of a group 2-cocycles. They are needed for the construction of vertex algebras and appear from comparing orientations on different components of $\fM_{\cat{A}}$.

\begin{definition}
\label{def:group-2-cocycle}
    We will say that a map 
    $$
    \varepsilon\colon K^0(\cat{A})\times K^0(\cat{A})\to \{-1,+1\}
    $$
    is a \textit{group 2-cocycle }if it satisfies 
        \begin{align*}
\label{eq:epsidentity}
        \varepsilon_{\alpha,\beta}& =(-1)^{\chi(\alpha,\beta) + \chi(\al,\al)\chi(\beta,\beta)}\varepsilon_{\beta,\alpha}\\
        \varepsilon_{\alpha,\beta}\varepsilon_{\alpha+\beta,\gamma}&=\varepsilon_{\beta,\gamma}\varepsilon_{\alpha,\beta+\gamma}  \,,\qquad  \varepsilon_{\alpha,0} = \varepsilon_{0,\alpha}=1
    \numberthis
\end{align*}
for all $\al,\beta,\gamma\in K^0(\cat{A})$.
\end{definition}

\subsection{Auxiliary Framed Stacks}
\label{sec:auxiliary-stacks}

\subsubsection{}

\begin{definition}\label{bg:def:framing-functor}
  A {\it framing functor} for the abelian category $\cat{A}$ is a
  $\bC$-linear $\sT$-equivariant exact functor
  \[ \Fr\colon \cat{A}^{\Fr} \to \cat{Rep}_\bC(\sT) \]
  on a full exact $\sT$-invariant subcategory $\cat{A}^{\Fr} \subset
  \cat{A}$, closed under isomorphisms in $\cat{A}$ and direct summands
  in $\cat{A}$ (i.e. if $E, F \in \cat{A}$ with $E \oplus F \in
  \cat{A}^{\Fr}$, then $E, F \in \cat{A}^{\Fr}$), such that:
  \begin{enumerate}[label=(\alph*)]
  \item the moduli substack $\fM^{\Fr}_\alpha \subset \fM_\alpha$, of
    objects in $\cat{A}^{\Fr}$ of class $\alpha$, is open, and $\Fr$
    induces morphisms of moduli stacks
    \begin{align*}
      \fM^{\Fr}_\alpha &\to \bigsqcup_{d \ge 0} [\pt/\GL(d)] \\
      \fM^{\Fr,\pl}_\alpha &\to \bigsqcup_{d \ge 0} [\pt/\PGL(d)];
    \end{align*}
  \item $\Hom(E, E) \to \Hom(\Fr(E), \Fr(E))$ is injective for all $E
    \in \cat{A}^{\Fr}$, so $\Fr(E) \neq 0$ for $E \neq 0$;
  \item The dimension of the underlying vector space $\fr(E) \coloneqq \dim \Fr(E)$ depends only on the class
    $\alpha$ of $E \in \cat{A}^{\Fr}$.
  \end{enumerate}
  We usually write $\fr(\alpha)$ instead of $\fr(E)$.
\end{definition}
\subsubsection{}

\begin{definition} \label{def:auxiliary-stack}
  Let $Q$ be an acyclic quiver with edges $Q_1$ and vertices $Q_0 =
  Q_0^f \sqcup Q_0^o$ split into {\it ordinary} vertices
  $\blackbullet\in Q_0^o$ and {\it framing} vertices $\blacksquare \in
  Q_0^f$, such that ordinary vertices have no outgoing arrows. Given
  an abelian category $\cat{A}$ and a tuple
  \[ \vec\Fr \coloneqq (\Fr_v)_{v \in Q_0^o} \]
  of framing functors for $\cat{A}$, let
  ${\cat{A}}^{Q(\vec\Fr)}$ be the exact category of triples $(E,
  \vec V, \vec\rho)$ where:
  \begin{itemize}
  \item $E \in \cat{A}^{\vec\Fr} \coloneqq \bigcap_{v \in Q_0^o}
    \cat{A}^{\Fr_v}$;
  \item $\vec V = (V_v)_{v\in Q_0^f}$ are finite-dimensional vector
    spaces; set $V_v\coloneqq \Fr_v(E)$ for $v\in Q_0^o$;
  \item $\vec \rho = (\rho_e)_{e\in Q_1}$ are morphisms $\rho_e\colon
    V_{t(e)}\to V_{h(e)}$, where $t$ and $h$ denote tail and head.
  \end{itemize}
  A morphism $f\colon (E, \vec V, \vec\rho) \to (E', \vec V',
  \vec\rho')$ in $\cat{A}^{Q(\vec\Fr)}$ is given by a morphism $E \to
  E'$ in $\cat{A}$, which induces linear maps $f_v\colon \Fr_v(E) \to
  \Fr_v(E')$ for $v \in Q_0^o$, along with linear maps $f_v\colon V_v
  \to V'_v$ for $v \in Q_0^f$, such that $f_{h(e)} \circ \rho_e =
\rho'_e \circ f_{t(e)}$ for all $e \in Q_1$. The group $\bC^\times$
  of scaling automorphisms acts diagonally on $(V_v)_{v \in Q_0}$.
\end{definition}
\subsubsection{}
\begin{example}
\label{ex:CY4JScategories}
A particular construction of a framing functor was discussed in \cite[Definition 5.3]{Bo25}. One starts with a triangulated CY4 category $\wt{\cat{B}}^{\vec\Fr}$ (see \cite[Definition 3.19.i)]{Bo25}), which contains \textit{spherical objects} $(1,0)_{v}$ for all $v\in Q^o_0$ and $\cat{A}^{\vec\Fr}$ as a full subcategory. The objects $(1,0)_{v}$ must additionally satisfy  
\begin{itemize}
    \item $\Ext^{i}_{\wt{\cat{B}}^{\vec\Fr}}\big((1,0)_v, E\big)=0$ for all $i\neq 1$ and $E\in \cat{A}^{\vec\Fr}$ of class $\al\in \scE(\cat{A})$; and
    \item $\Fr_v(E):=\Ext^{i}_{\wt{\cat{B}}^{\vec\Fr}}\big((1,0)_v, E\big)$ satisfies the condition of Definition \ref{bg:def:framing-functor} (b) and (c) for every  $E\in \cat{A}^{\vec\Fr}$ of class $\al\in \scE(\cat{A})$.
\end{itemize}

One then considers the full exact sub-categories 
$$
\cat{B}^{\Fr_v} = \big\langle (1,0)_v, \cat{A}^{\vec\Fr}\big\rangle_{\operatorname{ex}}
$$
generated by $(1,0)_v$ and $\cat{A}^{\vec\Fr}$. One additionally assumes that each object $P\in \cat{B}^{\Fr_v}$ fits into an essentially unique exact triple 
$$
 \begin{tikzcd}
          E\arrow[r]&P\arrow[r]&(1,0)_v\otimes V
      \end{tikzcd}
$$
and the morphisms $\Hom_{\cat{B}^{\Fr_v}}(P_1,P_2)$ are equivalent to morphisms of their triples.
\end{example}
\subsubsection{}\label{bg:sec:aux-stacks}
A triple $(E, \vec V, \vec\rho)$ has class $(\alpha, \vec d)$ where
$\alpha$ is the class of $E$ and $\vec d \coloneqq (\dim V_v)_{v \in
  Q_0^f}$. Let
\[ \fM^{Q(\vec\Fr)} = \bigsqcup_{\alpha,\vec d} \fM^{Q(\vec\Fr)}_{\alpha,\vec d} \]
be the moduli stack of objects in $\cat{A}^{Q(\vec\Fr)}$. Clearly the
$\sT$-action lifts from $\fM$ to $\fM^{Q(\vec\Fr)}$. Let
$(\cV_v)_{v \in Q_0}$ be the universal bundles for $(V_v)_{v \in
  Q_0}$, and in particular let $(\cFr_v(\cE))_{v \in Q_0^o}$ denote
the universal bundles for $(\Fr_v(E))_{v \in Q_0^o}$. Let
\[ \pi_{\fM^{\vec\Fr}_\alpha}\colon \fM^{Q(\vec\Fr)}_{\alpha,\vec d} \to \fM^{\vec\Fr}_\alpha \]
be the forgetful map to the moduli stack $\fM^{\vec\Fr}_\alpha$ of
objects in $\cat{A}^{\vec\Fr}$ with class $\alpha$ and $\pi_{\fM^{\vec\Fr,\rig}_\alpha}$ its rigidification. Then we have
\[ \fM^{Q(\vec\Fr)}_{\alpha,\vec d} = \tot\bigg(\bigoplus_{e \in Q_1} \cV_{t(e)}^\vee \otimes \cV_{h(e)}\bigg) \xrightarrow{f_Q} \fM^{\vec\Fr}_{\alpha} \times \prod_{v \in Q_0^f} [\pt/\GL(d_v)], \]
so $\pi_{\fM^{\vec\Fr}_\alpha}$ is smooth. Define the bilinear element\footnote{They were denoted by $\Theta^\vee_{\pi_{\fM^{\vec\Fr}_\alpha}}$ in \cite[Corollary 5.8]{Bo25}}.
\begin{equation} \label{eq:framed-stack-forgetful-map-cotangent}
  \bF^{Q(\vec\Fr)} \coloneqq \sum_{e \in Q_1} \cV^\vee_{t(e)} \boxtimes \cV_{h(e)} - \sum_{v \in Q_0^f} \cV^\vee_v \boxtimes \cV_v \in K_\sT^\circ\left(\fM^{Q(\vec\Fr)} \times \fM^{Q(\vec\Fr)}\right).
\end{equation}
Write $p_1$ and $p_2$ for the projections of $\fM^{\vec\Fr}_{\alpha} \times \prod_{v \in Q_0^f} [\pt/\GL(d_v)]$ to $\fM^{\vec\Fr}_{\alpha}$ and $\prod_{v \in Q_0^f} [\pt/\GL(d_v)]$ respectively. By writing the cotangent complex triangle for the composition $\pi_{\fM^{\vec\Fr}_\alpha}=p\circ f_Q$, and using $T_{f_Q}\cong f_Q^*\big(\bigoplus_{e \in Q_1} \cV_{t(e)}^\vee \otimes \cV_{h(e)}\big)$ and $f_Q^*T_{p_1}\cong \bigoplus_v \End(\cV_v)[1]$ using \cite[Ex. 9.10]{HalpernLeistner2020ModernModuli}, we see that the restriction of $\bF^{Q(\vec\Fr)}$ to the diagonal is the relative
tangent complex of $\pi_{\fM^{\vec\Fr}_\alpha}$. When working with specific underlying classes $\gamma$, $\delta$, and dimension vectors $\vec f$, $\vec g$, we also write $\bF^{Q}_{\gamma\delta}(\vec f,\vec g)$ to record this information.
\subsubsection{}
\begin{example}
\label{ex:JSobstheoryframing}
   We return to Example \ref{ex:CY4JScategories}. Set $\bar{\fN}^{\vec\Fr}:=\fM_{\bar{\cat{B}}^{\vec\Fr}}$ and $\fN^{\Fr_v}:=\fM_{\cat{B}^{\Fr_v}}$ for the corresponding moduli stacks. As in \cite[Definition 5.3.a)]{Bo25}, we then require that the natural embedding 
   $$
\fN^{\Fr_v}\hookrightarrow \bar{\fN}^{\vec\Fr}
$$
is open. Since $\wt{\fM}^{Q^{\JS}(\Fr_v)}\cong \fN^{\Fr_v}$ and $\bar{\cat{B}}^{\vec\Fr}$ is CY4 in the sense of \cite[Definition 3.19.i)]{Bo25}, there is a CY4 obstruction theory on $ \bar{\fN}^{\vec\Fr}$ restricting to a CY4 obstruction theory on $\wt{\fM}^{Q^{\JS}(\Fr_v)}$. However, in this situation, we will usually replace $\wt{\fM}^{Q^{\JS}(\Fr_v)}$ by $\fN^{\Fr_v}$ using their equivalence, and we will work with its induced CY4 obstruction theory $\bF_{\fN^{\Fr_v}}$. It is determined by the diagram (see \cite[5.17]{Bo25})
\begin{equation}
\label{eq:FFJSobstructiontheory}
    \begin{tikzcd}
  \arrow[d]C(\delta)[-1]\arrow[r]&\arrow[d]\big(\wt{\bF}_{\fN^{\Fr_v}}\big)^\vee[2]\arrow[r]&\arrow[d]\bF_{\fN^{\Fr_v}}\\
  \arrow[d,"{\delta}"']\bL_{\pi_{\fM^{\Fr}}}[-1]\arrow[r,"\psi"]&\arrow[d,"{\psi^\vee[2]}"]\bE\arrow[r]&\arrow[d]\wt{\bF}_{\fN^{\Fr_v}}\\
  \cV\otimes \cV^*\otimes \bH\arrow[r,"{\delta^\vee[2]}"]&{\bL_{\pi_{\fM^{\Fr}}}^\vee[3]}\arrow[r]&{C\big({\delta}\big)^\vee[3]}
    \end{tikzcd}\,,
\end{equation}
   where each column and row is a distinguished triangle and $\bH = \Ext^{\bullet}\big((1,0)_v,(1,0)_v\big)$.
\end{example}
\subsubsection{}
\label{sec:puresheavesJSframing}
As in Example \ref{ex:sheavesandreps0}, choose $\cat{A}= \Coh(X)$ with $X$ satisfying \eqref{eq:H3ZZ_2vanishing}. Choose $\bar{K}(\cat{A})$ as explained there and $\scE(\cat{A})$ to be the set of $\alpha$ represented by pure sheaves of fixed rank. 

In this situation, it is common to construct $\Fr$ following \cite[Example 5.4.i)]{Bo25} for a given ample divisor $D_v$ on $X$. The framing is naturally formulated in the situation of Example \ref{ex:CY4JScategories} after setting 
$$\wt{\cat{B}}^{\vec\Fr} = D^b(X)\qquad\text{and}\qquad (1,0)_v = \cO_X(-D_v)[1]\,.$$
 The category $\cat{A}^{\Fr_v}$ will consist of all sheaves $E$ in $\Coh(X)$ which satisfy the condition 
$$
H^i\big(E(D_v)\big) = 0\qquad\text{for } i\neq 0\,.
$$
So the framing functor is given by 
$$
\Fr_v(E) = H^0\big(E(D_v)\big)\,.
$$
By Lemma \cite[Lemma~6.1]{Bo25}, this data satisfies the necessary conditions from Example \ref{ex:CY4JScategories} and \ref{ex:JSobstheoryframing}. In particular, the invariants from Theorem \ref{thm:sst-invariants} can be defined using the obstruction theory \eqref{eq:FFJSobstructiontheory} without relying on the Jouanolou devices in Theorem \ref{sp:thm:symm-pullback-localization}. 
\subsubsection{}
Consider now our auxiliary framed stacks from Section~\ref{sec:auxiliary-stacks}. Under the conditions below, these have canonical "de-rigidification" morphisms and hence satisfy the conditions in Lemma~\ref{lem:k-homology-pl}.

\begin{definition}[``De-rigidification''] \label{def:de-rigidification}
  Suppose $\vec d$ is a dimension vector with $d_i = 1$ for some $i$.
  Then the rigidification map
  \[ \Pi_{\alpha,\vec d}^\pl\colon \fM_{\alpha,\vec d}^{Q(\vec\Fr)} \to \fM_{\alpha,\vec d}^{Q(\vec\Fr),\pl} \]
  is a {\it trivial} $\bC^\times$-gerbe because the group of scaling
  automorphisms may be identified with $\Aut(V_i) = \bC^\times$. In
  other words, rigidification has a non-canonical description as the
  choice of an isomorphism $\bC \xrightarrow{\sim} V_i$, and it has a
  section
  \[ I_{\alpha,\vec d}\colon \fM^{Q(\vec\Fr),\pl}_{\alpha,\vec d} \to \fM^{Q(\vec\Fr)}_{\alpha,\vec d} \]
  given by forgetting this isomorphism. We often write omit the
  subscript on $I_{\alpha,\vec d}$ and just write $I$ when there is no
  ambiguity.
\end{definition}

\subsection{Homology Theories and Lie Algebra Structure}
\label{sec:homology-theories-LA}
\subsubsection{}
The wall-crossing formula central to this paper is proved using master space localization and the technical tools developed in Section~\ref{sec:sym-pullback}. These techniques are robust and hence, the wall-crossing formula is fundamentally agnostic to the homology theory used, as long as certain operations are available. In the following sections, we outline two theories, proposed for this purpose by the third-named and first-named author, respectively. Before we do so, for the reader interested primarily in the wall-crossing formula for invariants, we give an overview of the operations used in the formulation and the proof of wall-crossing. This is by no means a complete list of axioms, but should allow the reader to understand the following sections on semistable invariants and wall-crossing without diving into the technical details of homology theories.

\subsubsection{}
\label{sec:equivariant-hom-notation}
We work in a homology (or K-homology) theory for algebraic stacks, which, in this section and the wall-crossing proofs, we denote by $H(-)$ for simplicity. Our proofs work in both $H(-)$ and an appropriate equivariant localized version, which we denote by $H^\sT(-)_\loc$. For generality, we express the wall-crossing formulae and prove them using $H^\sT(-)_\loc$. The resulting Lie algebras, recalled in \S\ref{sec:LA-quotient} below, will be labeled by $L(-)_\loc$. 

In the sections below, we present explicit candidates: see $\bfK_\circ^\sT(-)_\loc$ of Definition~\ref{bg:def:k-homology-concrete-theory}, $H_*^\sT(-)_\loc$ of Section~\ref{sec:Khan-bivariant}, and $K_0^\sT(-)_\loc$ of Section~\ref{sec:homology-to-khomology}. Note that the latter two theories are originally defined before localizing. 

\subsubsection{}\label{bg:sec:universal-invariants}
The first thing we need in this homology theory is a way to define invariants. Specifically, let $M$ be a quasi-compact and separated algebraic space with a $\sT$-action, such that $M^\sT$ is a proper algebraic space with the resolution property. If $M$ is equipped with a virtual class $[\hat\cO_M^\vir]$, we want to be able to use this to define a {\it universal enumerative invariant}
\begin{equation*}
    \sZ_M \in H^\sT(M)_\loc,
\end{equation*}
which describes integration over the virtual class. This is realized in Definition~\ref{bg:def:univ-enum-inv} and Definition~\ref{def:bivariant-hom-invariants} for the respective theories.

Sometimes, we need to consider $\bT:=\bG_m\times \sT$-equivariant theories for a fixed $\sT$. In this case, we will write $H^{\bT}(-)$ for the generic $\bT$-equivariant homology theory.

\subsubsection{}
In our case, the algebraic spaces $M$ above will be semistable loci with no strictly semistable objects. Thus there will exist natural open embeddings $j:M\hookrightarrow \fM^\pl_{\cat{A}}$. To compare invariants of different spaces $M$, we need to be able to push them forward to a common homology group. For this, we require \emph{arbitrary} pushforwards in our homology theory $H^\sT(-)_\loc$, so that we can define
\begin{equation}
\label{eq:universal-inv-general}
  j_*\sZ_M\in H^\sT(\fM^\pl)_\loc.
\end{equation} 

\subsubsection{}
\label{sec:LA-quotient}
The wall-crossing formulae will live in Lie algebras $L(\fM)_\loc$, and they will be expressed in terms of the corresponding Lie brackets. We give a rough overview of these Lie algebras. For the abelian categories we consider, the homology $H^\sT(\fM)_\loc$ of their moduli stacks comes equipped with a vertex algebra-like structure consisting among other things of a translation operator $D(z):H^\sT(\fM)_\loc\to H^\sT(\fM)_\loc\pseries{1-z}$, and a state-field correspondence $Y(-,z)(-)$. These structures are made explicit for the homology theories concretely considered in this paper in Theorem~\ref{thm:mVOA-monoidal-stack} and Section~\ref{bg:sec:va-and-t-deformations}, respectively.

A specific quotient of $H^\sT(\fM)_\loc$ described in Definition \ref{def:pl-groups} and §\ref{sec:Lie-algebras}, respectively, is the associated Lie algebra $L(\fM)_\loc$, which we primarily work in. Note that the latter construction first defines a non-localized Lie algebra and then localizes in §\ref{sec:localized-LA}. In order to compare the semistable invariants defined in $L(\fM)_\loc$ in Section~\ref{sec:sst-inv} with the invariants $j_*Z_M$ already defined in $H^\sT(\fM^\pl)_\loc$, we need a morphism
\begin{equation*}
  L(\fM)_\loc\to H^\sT(\fM^\pl)_\loc.
\end{equation*}
In operational K-homology, this map is defined in Lemma~\ref{lem:k-homology-pl}. For $H^{\sT}_*(-)$ and $K^{\sT}_0(-)$ this map is an isomorphism as recalled in Proposition~\ref{prop:quotient-by-T}, so the lifts along it are canonical. We will abuse notation and denote them by $Z_M\in L(\fM)_\loc$.

\subsubsection{}
\label{sec:LA-formula+residue}
We describe the Lie algebra here for the two stacks, which we consider in the proof of the wall-crossing formula. The first is the moduli stack $\fM_{\cat{A}}=\sqcup_\alpha \fM_\alpha$ of our underlying CY$4$ category, where $\alpha$ ranges over the emergent classes. Taking $\Theta_{\alpha\beta}$ to be the restrictions of $\Theta_{\cat{A}}$ from \eqref{eq:ThetaofA} to the components $\fM_\alpha\times\fM_\beta$, the Lie bracket is
\begin{equation*}
  [-,-] = \varepsilon_{\alpha\beta}\cdot\rho_z \left\{ \Phi_* (D(z)\times\id) \left(\frac{(-\boxtimes-)}{\fe_{\bT}(z\Theta_{\alpha\beta})}\right)\right\}
\end{equation*}
for entries in $L(\fM_\alpha)_\loc$ and $L(\fM_\beta)_\loc$ respectively. Here, $\varepsilon_{\alpha\beta}$ are signs satisfying \eqref{eq:epsidentity} and appearing from comparing orientations in Section~\ref{sec:epsilonorientations}, $\rho_z$ is a residue map in $z$ defined in \cite[Def. 3.1.6]{Liu2022} and \eqref{eq:LiefromVA}, respectively, $\Phi$ is the direct sum morphism from \eqref{eq:murho}, and we write $\frac{(-\boxtimes-)}{\fe_{\bT}(z\Theta_{\alpha\beta})}$ for $(-\boxtimes-)\cap\frac{1}{\fe_{\bT}(z\Theta_{\alpha\beta})}$, where $\boxtimes$ and $\cap$ are the equivariant Künneth and cap product operations expected from our homology theory, and $\fe_{\bT}$ is the appropriate $\bT$-equivariant Euler class. In this case $z$ is the weight of an additional torus $\bG_m$ acting trivially. For formulas in K-theory $\fe_{\bT}(-)$ becomes  $\hat{\fe}_{\bT}(-)$ defined in \S\ref{sec:localization-formula}. 

In case we take the Lie bracket of two universal enumerative invariants $\sZ_M$ and $\sZ_N$ as above, this expression should be roughly read as follows. It describes all virtual integrals over $M\times N$ (via $\boxtimes$), tensored by $\frac{1}{\fe_\bT(z\Theta_{\alpha\beta})}$ (via cap product), and by any class pulled back via the direct sum map $\Phi$ composed with $\Psi\times \id$ (via $\Phi_* (D(z)\times\id)$). At the end we take residues and multiply by signs.

\subsubsection{}
\begin{remark}
\begin{enumerate}[label=\roman*)]
    \item  In the case of the homology theories constructed in §\ref{sec:equivariant-homology-bivariant}, the above operations are additionally continuous with respect to the (graded) linear topology on $H^{\sT}_*(-)$ and $K^{\sT}_0(-)$. This is why we use completed tensor products $(-\wh{\otimes}-)$ there to denote this additional property. Still, there always exist natural maps $(-\otimes-)\to (-\wh{\otimes}-)$, so a reader not interested in the precise formulation may ignore this topological refinement. 
\item When $\sT = \{1\}$, $H^{\sT}_*(\fM_{\cat{A}})$ is precisely the homology used in \cite[§2.2]{GJT} due to the non-equivariant limit property in §\ref{sec:Khan-bivariant}. In particular, we get the vertex algebra which was used in \cite[§4.4]{GJT} to formulate the CY4 wall-crossing conjecture.
\end{enumerate}
\end{remark}
\subsubsection{}
\label{sec:flag-VA+LA}
The other situation considered during the proof of wall-crossing are the auxiliary moduli stacks $\fM^{\bar{Q}(r)}_{\alpha,\vec d}$, where $\alpha$ runs over all emergent classes, and $\vec d$ over all possible dimension vectors of the quiver $\bar{Q}(r)$ defined in \eqref{wc:eq:flag-quiver}. Following \cite[Definition 8.5]{Bo25}, one introduces the class
\begin{equation}
\label{eq:Theta-flag}
  \Theta_{(\beta,\vec e)}^{(\gamma,\vec f)}\coloneqq \pi^*\Theta_{\beta\gamma}+(12)^*\bF^{\bar{Q}(r)}_{\gamma\beta}(\vec f,\vec e)+\bF_{\beta\gamma}^{\bar{Q}(r)}(\vec e,\vec f)^\vee,
\end{equation}
where $\pi$ is the forgetful map and $\bF_{\beta\gamma}^{\bar{Q}(r)}(\vec e,\vec f)$ is the class from Section~\ref{bg:sec:aux-stacks}. Replacing the signs $\varepsilon_{\al,\be}$ by $  \varepsilon_{(\beta,\vec e)}^{(\gamma,\vec f)}$ from Definition \ref{def:framed-epsilons} the Lie bracket is given by
\begin{equation}\label{bg:eq:aux-lie-bracket}
  [-,-] = \varepsilon_{(\beta,\vec e)}^{(\gamma,\vec f)}\cdot \rho_z\left\{ \Phi_* \circ (D(z)\times\id) \frac{(-\boxtimes-)}{\fe_\bT\left(z\Theta_{(\beta,\vec e)}^{(\gamma,\vec f)}\right)}\right\}
\end{equation}
for entries in $L(\fM^{\bar{Q}(r)}_{\beta,\vec e})_\loc$ and $L(\fM^{\bar{Q}(r)}_{\gamma,\vec f})_\loc$, respectively. For K-theory, one should again use $\hat{\fe}_{\bT}(-)$ instead.

\subsection{Operational K-Homology}
\subsubsection{}

\begin{definition}[{\cite[\S 2.2]{Liu2022}}] \label{def:operational-k-homology}
  Recall that we write $K_\sT^\circ(\fX) \coloneqq
  K_0(D_{\cat{Perf},\sT}(\fX))$. Let $\bK^\sT_\circ(\fX)$ be the
  $\bk_{\sT}$-module which consists of all collections
  \[ \phi\coloneqq \{K^\circ_\sT(\fX \times S) \xrightarrow{\phi_S} K^\circ_\sT(S)\}_S \]
  of homomorphisms of $K^\circ_\sT(S)$-modules, for all $S \in
  \cat{Art}_\sT$ which we call the {\it base}, which obey the
  following axioms.
  \begin{enumerate}[label = (\alph*)]
  \item (Naturality) For any morphism $h\colon S \to S'$ in
    $\cat{Art}_\sT$, there is a commutative square
    \[ \begin{tikzcd}
        K^\circ_{\sT}(\fX \times S') \ar{r}{(\id \times h)^*} \ar{d}{\phi_{S'}} & K^\circ_{\sT}(\fX \times S) \ar{d}{\phi_S} \\
        K^\circ_{\sT}(S') \ar{r}{h^*} & K^\circ_{\sT}(S).
      \end{tikzcd} \]
  \item (Equivariant localization) There exists a proper algebraic
    space $\fF_\phi$ with the resolution property,
    equipped with the trivial $\sT$-action, and a map $\fix_\phi\colon
    \fF_\phi \to \fX$ in $\cat{Art}_\sT$, such that $\phi$ factors as
    \[ \begin{tikzcd}
        K_{\sT}^\circ(\fX \times S) \ar{rr}{\phi_S} \ar{dr}[swap]{(\fix_\phi \times \id)^*} && K_{\sT}^\circ(S) \\
        & K_{\sT}^\circ(\fF_\phi \times S) \ar{ur}[swap]{\phi_S^{\sT}}
      \end{tikzcd} \]
    for homomorphisms $\{\phi_S^{\sT}\}_S$ of
    $K_{\sT}^\circ(S)$-modules which themselves satisfy all other
    axioms, i.e. forming an element $\phi^\sT \in \bK_\circ^\sT(\fF)$.
  \item (Finiteness) for any proper algebraic space $S$ with the
    resolution property, equipped with the trivial $\sT$-action,
    \begin{equation} \label{eq:operational-k-homology-finiteness-condition}
      \phi_S^\sT\left(I^\circ(\fF \times S)^{\otimes N}\right) = 0, \qquad \forall N \gg 0,
    \end{equation}
    where $I^\circ(\fF_\phi \times S) \subset K^\circ(\fF_\phi \times
    S) \subset K^\circ_\sT(\fF_\phi \times S)$ is the {\it
      non-equivariant} augmentation ideal of rank-$0$ elements.
  \end{enumerate}
  The sum $\phi + \psi$ of two elements $\phi, \psi \in
  \bK_\circ^\sT(\fX)$ still satisfies the equivariant localization and
  finiteness axioms by setting $\fF_{\phi + \psi} \coloneqq \fF_\phi
  \sqcup \fF_\psi$.
  
  Similarly, define the {\it localized} groups
  $\bK_\circ^\sT(\fX)_{\loc}$ by replacing all groups
  $K_{\sT}^\circ(-)$ with the localized groups
  $K_{\sT}^\circ(-)_{\loc}$.
\end{definition}

\subsubsection{}
\label{sec:k-homology-shorthand}

To prevent notational clutter, and for clarity, we generally write
formulas involving elements $\phi$ in a K-homology group
$\bK_\circ^\sT(\fX)$ in terms of just the functional $\phi_S$ for $S =
\pt$. It will always be clear how to extend the formula to $\phi_S$ for
general $S$. For instance, in \eqref{eq:k-homology-pushforward} below,
the definition $(f_*\phi)(\cE) \coloneqq \phi(f^*\cE)$ means that
\[ f_*\phi \coloneqq \{(f_*\phi)_S\colon K^\circ_\sT(\fY \times S) \to K^\circ_\sT(S)\} \in \bK^\sT_\circ(\fY) \]
is defined by $(f_*\phi)_S(\cE) \coloneqq \phi_S((f \times \id)^*\cE)$
for every base $S$.

\subsubsection{}
\label{sec:k-homology-properties}

We list some properties of $\bK_\circ^\sT(-)$, mostly inherited from
$K_\sT^\circ(-)$.
\begin{itemize}
\item Tensor product on $K^\circ_\sT(\fX)$ induces a cap product
  \begin{equation} \label{eq:k-homology-cap}
    \cap\colon \bK^\sT_\circ(\fX) \otimes K^\circ_\sT(\fX) \to \bK^\sT_\circ(\fX), \qquad (\phi \cap \cF)(\cE) \coloneqq \phi(\cE \otimes \cF).
  \end{equation}
  This is well-defined since $I^\circ(\fF_\phi)$ is an ideal.

\item A $\sT$-equivariant morphism $f\colon \fX \to \fY$ induces a
  functorial pushforward
  \begin{equation} \label{eq:k-homology-pushforward}
    f_*\colon \bK^\sT_\circ(\fX) \to \bK^\sT_\circ(\fY), \qquad (f_*\phi)(\cE) \coloneqq \phi(f^*\cE)
  \end{equation}
  which satisfies the {\it push-pull} formula
  \[ f_*(\phi) \cap E = f_*(\phi \cap f^*E), \qquad \phi \in \bK_\circ^\sT(\fX), \; E \in K^\circ_\sT(\fY). \]

\item A proper and representable $\sT$-equivariant morphism $f\colon
  \fX \to \fY$ of finite Tor-amplitude induces a functorial pullback
  \begin{equation} \label{eq:k-homology-pullback}
    f^*\colon \bK^\sT_\circ(\fY) \to \bK^\sT_\circ(\fX), \qquad (f^*\phi)(\cE) \coloneqq \phi(f_*\cE),
  \end{equation}
  which respects cap product, i.e.
  \[ f^*\phi \cap f^*E = f^*(\phi \cap E), \qquad \phi \in \bK_\circ^\sT(\fY), \; E \in K^\circ_\sT(\fY). \]

\item There is an external tensor product
  \begin{equation} \label{eq:k-homology-boxtimes}
    \boxtimes\colon \bK_\circ^\sT(\fX) \otimes \bK_\circ^\sT(\fY) \to \bK_\circ^\sT(\fX \times \fY), \qquad (\phi \boxtimes \psi)_S \coloneqq \psi_S \circ \phi_{\fY \times S}.
  \end{equation}
\end{itemize}
All these operations are homomorphisms of $\bk_\sT$-modules.

\subsubsection{}

\begin{definition} \label{def:supported-on}
  An element $\phi \in \bK_\circ^\sT(\fX)_\loc$ is {\it supported on}
  a $\sT$-invariant substack $\iota\colon \fX' \hookrightarrow \fX$ if
  there exists $\phi' \in \bK_\circ^\sT(\fX')$ such that
  \[ \iota_* \phi' = \phi. \]
  As for usual homology classes or for sheaves, we often omit
  $\iota_*$ and simply write $\phi' = \phi$.
\end{definition}

\subsubsection{}

While Definition~\ref{def:operational-k-homology} provides a
reasonably minimal set of axioms that an operational K-homology theory
should satisfy, in practice it is convenient to choose an explicit
construction of operational K-homology groups $\bfK_\circ^\sT(-)$ that
satisfy these axioms. For this paper, we moreover require an
operational K-homology theory $\bfK_\circ^\sT(-)$ which is {\it
  commutative}, meaning that for all $\fX, \fX'$,
\[ \phi \boxtimes \psi = \sigma_*(\psi \boxtimes \phi), \quad \forall \phi \in \bfK_\circ^\sT(\fX), \, \psi \in \bfK_\circ^\sT(\fX'), \]
where $\sigma$ swaps the two factors. This commutativity is not a
consequence of the axioms in
Definition~\ref{def:operational-k-homology}. Below we provide a
construction of such a commutative operational theory. As explained in
\cite[\S 2.2]{KLT25}, this construction may be adapted in a
straightforward way to other cohomology theories, such as Chow, by
replacing $G_0^\sT(-)$ and $K^\circ_\sT(-)$ appropriately.

\subsubsection{}

\begin{definition}\label{bg:def:k-homology-concrete-theory}
  Consider triples $\phi = (Z_\phi, \fix_\phi, \cF_\phi)$ where:
  \begin{enumerate}[label = (\roman*)]
  \item $Z_\phi$ is a proper algebraic space with the resolution
    property;
  \item $\fix_\phi\colon Z_\phi \to \fX$ is a $\sT$-equivariant
    morphism for the trivial $\sT$-action on $Z_\phi$;
  \item $\cF_\phi \in G_0^\sT(Z_\phi)_{\loc}$ is a G-theory element.
  \end{enumerate}
  Given a base $S \in \cat{Art}_\sT$, a triple $\phi = (Z_\phi,
  \fix_\phi, \cF_\phi)$ defines homomorphisms
  \begin{align}
    \phi_S\colon K^\circ_\sT(\fX \times S)_{\loc} &\to K^\circ_\sT(S)_{\loc} \nonumber \\
    \cE &\mapsto (\pi_S)_*(\pi_Z^*\cF_\phi \otimes (\fix_\phi \times \id)^*(\cE)) \label{eq:actual-k-homology-elements}
  \end{align}
  where $\pi_Z$ and $\pi_S$ are the projections from $Z_\phi \times S$
  to the $Z_\phi$ and $S$ factors respectively. Using (i), since both
  $\pi_Z^*\cF_\phi$ and $(\fix_\phi \times \id)^*\cE$ are flat with
  respect to $\pi_S$, and $\pi_S$ is proper, the output of $\phi_S$
  indeed defines an element of $K^\circ_\sT(S)_{\loc}$. Some
  straightforward base change formulas show that $\{\phi_S\}_S$
  defines an element of $\bK_\circ^\sT(\fX)_{\loc}$; in particular,
  the finiteness axiom follows from \cite[Lemma 2.1.6]{KLT25}.

  Hence define the $\bk_{\sT,\loc}$-submodule $\bfK_\circ^\sT(-)_\loc
  \subset \bK_\circ^\sT(-)_\loc$ generated by all elements $\phi =
  \{\phi_S\}_S$ associated via \eqref{eq:actual-k-homology-elements}
  to any triple $(Z_\phi, \fix_\phi, \cF_\phi)$. It is straightforward
  to check that $\bfK_\circ^\sT(-)_\loc$ is closed under the cap
  product \eqref{eq:k-homology-cap}, the pullback
  \eqref{eq:k-homology-pullback}, and the external tensor product
  \eqref{eq:k-homology-boxtimes}, and is additionally a commutative
  K-homology theory \cite[Proposition 2.2.7]{KLT25}.
\end{definition}

\subsubsection{}

It is convenient to continue using the shorthand in
\S\ref{sec:k-homology-shorthand} for elements of
$\bfK_\circ^\sT(\fX)_{\loc}$. Namely, we view an element as an operator
\[ \phi = \chi\left(Z_\phi, \cF_\phi \otimes \fix_\phi^*(-)\right)\colon K_\sT^\circ(\fX)_{\loc} \to \bk_{\sT,\loc} \]
which behaves well with respect to base change. For instance, given
another element $\psi = \chi(Z_\psi, \cF_\psi \otimes \fix_\psi^*(-))
\in \bfK_\circ^\sT(\fY)_{\loc}$,
\begin{equation} \label{eq:universal-invariants-tensor-product}
  \phi \boxtimes \psi = \chi\left(Z_\phi \times Z_\psi, \left(\cF_\phi \boxtimes \cF_\psi\right) \otimes \left(\fix_\phi \times \fix_\psi\right)^*(-)\right).
\end{equation}
It is easy to see that $\bfK_\circ^\sT(\pt) = \bk_{\sT,\loc}$, the
trivial $\bk_{\sT,\loc}$-module with generator $\id = \chi(\pt, -)$.

\subsubsection{}

\begin{example} \label{ex:k-homology-BGm}
  Let $\fX = [\pt/\bC^\times]$. Then for every integer $k > 0$, there
  are natural projection morphisms
  \[ \fix_k\colon \bP^k = (\bC^{k+1} \setminus \{0\}) / \bC^\times \to [\pt/\bC^\times], \]
  which may be used to define elements
  \[ \xi^k \coloneqq (-1)^k \cdot \chi\left(\bP^k, \cO_{\bP^k}(-k) \otimes \fix_k^*(-)\right) \in \bfK_\circ^\sT([\pt/\bC^\times])_{\loc}. \]
  (The sign $(-1)^k$ is a convenient normalization.) Write
  $K^\circ([\pt/\bC^\times]) = \bZ[s^{\pm}]$. Then $\fix_k^*(s^n) =
  \cO_{\bP^k}(n)$, so
  \[ \xi^k(s^n) = (-1)^k \binom{n}{k}. \]
  Clearly these formulas can be made $\sT$-equivariant for the trivial
  $\sT$-action on $[\pt/\bC^\times]$. Thus
  \[ \bfK_\circ^\sT([\pt/\bC^\times])_{\loc} \supset \bk_{\sT,\loc}[\xi]. \]
\end{example}

\subsubsection{}
\label{sec:universal-invariants-shorthand}

\begin{definition}\label{bg:def:univ-enum-inv}
  Let $M$ be a quasi-compact and separated algebraic space with an
  $\sT$-action and a virtual class $[\hat\cO_M^\vir]\in
  K_{\sT}^\circ(M,\bZ[2^{-1}])$, constructed using either a CY$4$
  obstruction theory or the techniques in
  Section~\ref{sec:sym-pullback}. Furthermore, suppose that $M^\sT$ is
  a proper algebraic space with the resolution property. Define the
  {\it universal (operational) invariant}
  \begin{equation} \label{eq:univ-enum-inv}
    \begin{aligned}
      \sZ_M
      &\coloneqq \chi\left(M, \hat\cO_M^\vir \otimes -\right) \\
      &\coloneqq \chi\left(M^{\sT}, \frac{\tilde{\cO}^{\vir}_{M^{\sT}}}{\widehat{\fe}_{\sT}(\bN^{>}_{\iota})} \otimes \iota^*(-)\right) \in \bfK_\circ^{\sT}(M)_\loc
    \end{aligned}
  \end{equation}
  where the first line is an abbreviation for the second using \eqref{sp:eq:pullback-localization-formula} in Theorem~\ref{sp:thm:symm-pullback-localization}(c), and we are using the shorthand notation above.

  The typical setting is that $j\colon M \hookrightarrow
  \fM_{\cat{A}}^{\rig}$ is a $\sT$-invariant open locus and we use
  $j^*\bE^{\rig}$ as the CY4 obstruction theory on $M$. In this case,
  we will be interested in $j_*\sZ_M \in
  \bfK_\circ^{\sT}(\fM^{\pl}_{\cat{A}})_{\loc}$.
\end{definition}

\subsection{Multiplicative Vertex Algebra Structure on Operational K-Homology}
\subsubsection{}
\begin{definition} \label{def:graded-monoidal-stack}
  A {\it graded monoidal $\sT$-stack with symmetric bilinear element} is the
  data of:
  \begin{enumerate}[label = (\roman*)]
  \item an Artin stack $\fM = \bigsqcup_\alpha \fM_\alpha$, where
    $\alpha$ ranges over an additive monoid $\sA$ and the torus $\sT$ acts
    on each $\fM_\alpha$;
  \item for every $\alpha$ and $\beta$, $\sT$-equivariant morphisms
    \begin{align*}
      &\Phi_{\alpha,\beta}\colon \fM_\alpha \times \fM_\beta \to \fM_{\alpha+\beta}, \\
      &\Psi_\alpha \colon [\pt/\bC^\times] \times \fM_\alpha \to \fM_\alpha,
    \end{align*}
    and elements $\Theta_{\alpha,\beta} \in K_\sT^\circ(\fM_\alpha \times
    \fM_\beta)$.
  \item A bilinear integer-valued form $\chi(-,-)$ on $\sA$, and signs
  \begin{equation*}
    \varepsilon:\sA\times\sA\to \pm 1,
  \end{equation*} 
  forming a group 2-cocycle in the sense of Definition~\ref{def:group-2-cocycle} for $\chi(-,-)$.
  \end{enumerate}
  This data must satisfy the following axioms, for all $\alpha$ and
  $\beta$.
  \begin{enumerate}[label = (\alph*)]
  \item (Identity) $\fM_0 = *$ consists of a single point.
  \item ($\Phi$ is monoidal) $\Phi_{0,\alpha} = \id = \Phi_{\alpha,0}$
    and $\Phi_{\alpha,\beta+\gamma} \circ (\id \times
    \Phi_{\beta,\gamma} \times \id) = \Phi_{\alpha+\beta,\gamma} \circ
    (\Phi_{\alpha,\beta} \times \id)$.
  \item ($\Psi$ is $[\pt/\bC^\times]$-action) Define $\Omega\colon
    [\pt/\bC^\times]^2 \to [\pt/\bC^\times]$ using multiplication on
    on $\bC^\times$. Then
    \begin{align*}
      \Psi_{\alpha+\beta} \circ (\id \times \Phi_{\alpha,\beta}) &= \Phi_{\alpha,\beta} \circ (\Psi_\alpha^{(12)}, \Psi_\alpha^{(13)}) \\
      \Psi_\alpha \circ (\id \times \Psi_\alpha) &= \Psi_\alpha \circ (\Omega \times \id)
    \end{align*}
    where superscripts $(ij)$ mean to act on the $i$-th and $j$-th
    factors.
  \item ($\Theta_{\alpha,\beta}$ is bilinear) Let $\cL_{B\bG_m}
    \in G_0^\sT(B\bGm)$ be the weight-$1$ representation.
    Then
    \begin{equation} \label{eq:mVOA-bilinear-complex-conditions}
      \begin{alignedat}{2}
        (\Phi_{\alpha,\beta} \times \id)^*(\Theta_{\alpha+\beta,\gamma}) &= \Theta_{\alpha,\gamma} \oplus \Theta_{\beta,\gamma}, \qquad & (\Psi_\alpha \times \id)^*(\Theta_{\alpha,\beta}) &= \cL^\vee \boxtimes \Theta_{\alpha,\beta}, \\
        (\id \times \Phi_{\beta,\gamma})^*(\Theta_{\alpha,\beta+\gamma}) &= \Theta_{\alpha,\beta} \oplus \Theta_{\alpha,\gamma}, \qquad & (\id \times \Psi_\beta)^*(\Theta_{\alpha,\beta}) &= \cL \boxtimes \Theta_{\alpha,\beta}.
      \end{alignedat}
    \end{equation}
    Some obvious pullbacks along projections to various factors have
    been omitted.
  \item ($\Theta_{\alpha,\beta}$ is symmetric) We say the bilinear elements $\Theta$ are \textit{symmetric} if, for $(12)\colon
    \fM_\beta \times \fM_\alpha \to \fM_\alpha \times \fM_\beta$ swapping
    the two factors, we have
    \[ \Theta_{\beta,\alpha} = (12)^*\Theta_{\alpha,\beta}^\vee. \]
  \end{enumerate}
  A {\it morphism} $f\colon (\fM, \Phi, \Psi, \Theta) \to (\fM', \Phi',
  \Psi', \Theta')$ of two graded monoidal stacks with (symmetric)
  bilinear elements is a collection of morphisms $f_\alpha\colon
  \fM_\alpha \to \fM'_\alpha$ in $\cat{Art}_\sT$, for every $\alpha$,
  such that
  \[ \Phi'_{\alpha,\beta} \circ (f_\alpha \times f_\beta) = f_{\alpha+\beta} \Phi_{\alpha,\beta}, \qquad \Psi'_\alpha \circ f_\alpha = f_\alpha \circ \Psi_\alpha, \qquad (f_\alpha \times f_\beta)^*\Theta_{\alpha,\beta} = \Theta'_{\alpha,\beta}. \]
\end{definition}

\subsubsection{}

\begin{definition} \label{def:degree-map}
  Let the rigidification $\Pi\colon \fX \to \fX^\pl$ be a $\sT$-equivariant
  $\bC^\times$-gerbe. Then there is a natural $\sT$-equivariant map
  \[ \Psi\colon [\pt/\bC^\times] \times \fX \to \fX \]
  given by $(\pt, x) \mapsto x$ on $\bC$-points and $(\lambda, f)
  \mapsto (\lambda \cdot \id_x) \circ f$ on the associated stabilizer
  groups. It induces a {\it degree operator}
  \[ z^{\deg}\colon K_\sT^\circ(\fX) \xrightarrow{\Psi^*} K_\sT^\circ([\pt/\bC^\times] \times \fX) \cong K_\sT^\circ(\fX)[z^\pm] \]
  where we identify $K_\sT^\circ(B\bGm) =
  G_0^\sT(B\bGm) = \bk_\sT[z^\pm]$. If $\fX = \prod_i \fX_i$
  is a product of such $\bGm$-gerbes, we write $z^{\deg_i}$ for
  the operator defined by acting on the $i$-th factor by $\Psi^*$.
  Note that, by construction,
  \begin{equation} \label{eq:degree-map-pl-stack}
    z^{\deg} \Pi^*\cE = \Pi^*\cE, \qquad \cE \in K_\sT^\circ(\fX^\pl).
  \end{equation}
\end{definition}

\subsubsection{}

\begin{theorem}[{\cite[Theorem 3.3.5]{Liu2022}}] \label{thm:mVOA-monoidal-stack}
  Let $\fM = \bigsqcup_\alpha \fM_\alpha$ be a graded monoidal
  $\sT$-stack with symmetric bilinear elements
  $\Theta_{\alpha,\beta}$. Then
  \[ \bfK^{\sT}_\circ(\fM)_\loc \coloneqq \bigoplus_\alpha \bfK^{\sT}_\circ(\fM_\alpha)_\loc \]
  has the structure of a \emph{$\sT$-equivariant multiplicative vertex
  algebra} \cite[\S 3.2.1]{Liu2022}.
  \begin{itemize}
  \item The vacuum $\vac \in \bfK_\circ^{\sT}(\fM_0)_\loc$ is given by
    the identity map $\bk_{\sT} \to \bk_{\sT}$.

  \item The translation operator is
    \[ D(z)\phi \coloneqq \sum_{k \ge 0} (1-z)^k \Psi_*\left(\xi^k \boxtimes \phi\right) \in \bfK_\circ^{\sT}(\fM)_\loc\pseries{1-z}, \]
    where $\xi^k \in \bfK_\circ^{\sT}(B\bGm)_{\loc}$ are the elements
    from Example~\ref{ex:k-homology-BGm}. Explicitly, $(D(z)\phi)(\cE)
    = \phi(z^{\deg} \cE)$ where $z^{\deg}$ is the degree operator
    associated to $\Psi$ (Definition~\ref{def:degree-map}).

  \item The vertex product is given on $\phi \in
    \bfK^{\sT}_\circ(\fM_\alpha)_\loc$ and $\psi \in
    \bfK^{\sT}_\circ(\fM_\beta)_\loc$ by\footnote{The resolution property of $Z_\phi$ (and $Z_\psi$) in Definition~\ref{bg:def:k-homology-concrete-theory}(i) ensures that $\hat{\fe}_\sT(z\Theta_{\alpha\beta})$, once pulled back to $Z_\phi\times Z_\psi$, is well-defined in this expression, using \cite[Def. 2.1.10]{Liu2022}.}
    \begin{equation} \label{eq:monoidal-stack-vertex-product}
      \begin{aligned}
        Y(\phi, z) \psi
        &\coloneqq \varepsilon_{\alpha\beta} \cdot (\Phi_{\alpha,\beta})_* (D(z) \times \id) \left((\phi \boxtimes \psi)\cap\frac{1}{\hat{\fe}_\sT(z\Theta_{\alpha\beta})}\right) \\
        K^\circ_\sT(\fM_{\alpha+\beta}) \ni \cE
        &\mapsto \varepsilon_{\alpha\beta} \cdot (\phi \boxtimes \psi)\left(\frac{1}{\hat{\fe}_\sT(z\Theta_{\alpha\beta})} \otimes z^{\deg_1} \Phi_{\alpha,\beta}^* \cE \right).
      \end{aligned}
    \end{equation}
  \end{itemize}
  Furthermore, this construction is functorial: if $f\colon \fM \to
  \fM'$ is a morphism of graded monoidal $\sT$-stacks with symmetric
  bilinear elements, then
  \[ f_*\colon \bfK_\circ^{\sT}(\fM) \to \bfK_\circ^{\sT}(\fM') \]
  is a morphism of $\sT$-equivariant multiplicative vertex
  algebras.
\end{theorem}

\subsubsection{}
We apply this mainly to two situations:
\begin{itemize}
  \item The first is the situation in \S\ref{sec:LA-formula+residue}, where we consider the moduli stack $\fM_{\cat{A}}$ of our underlying CY$4$ category, ranging over the monoid of emergent classes. In this case, the symmetric bilinear elements $\Theta_{\alpha\beta}$ are the restrictions of $\Theta_{\cat{A}}$ from \eqref{eq:ThetaofA} to the components $\fM_\alpha\times\fM_\beta$. The morphisms $\Phi_{\alpha,\beta}$ and $\Psi_\alpha$ are the ones of \eqref{eq:murho}. The signs $\varepsilon_{\alpha\beta}$ are the orientation comparison signs discussed in Section~\ref{sec:epsilonorientations}.
  \item The second is the situation in \S\ref{sec:flag-VA+LA}, where we consider the moduli stacks $\fM^{\bar{Q}(r)}_{\alpha,\vec d}$, where $\alpha$ runs over the monoid of emergent classes, and $\vec d$ over all possible dimension vectors of the quiver $\bar{Q}(r)$ defined in \eqref{wc:eq:flag-quiver}. The bilinear elements are given by \eqref{eq:Theta-flag}, and the signs are the $\varepsilon_{(\beta,\vec e)}^{(\gamma,\vec f)}$ of Definition~\ref{def:framed-epsilons}.
\end{itemize}

\subsection{Lie-Algebra Structure on Operational K-Homology}

\subsubsection{}

Associated to the above vertex algebra is a Lie algebra on the
following $\bk_{\sT}$-module.

\begin{definition} \label{def:pl-groups}
  Given a $\sT$-equivariant multiplicative vertex algebra $(V, \vac,
  D, Y)$, let $\im(1 - D(z)) \subset V$ denote the
  $\bk_\sT$-submodule generated by the coefficients of
  \[ (1 - D(z))a \in V\pseries*{1-z} \]
  for all $a \in V$, and define
  \[ V^\pl \coloneqq V / \im(1 - D(z)). \]
\end{definition}

\subsubsection{}

\begin{lemma}[{\cite[\S 2.3.10]{KLT25}}] \label{lem:pl-group-functoriality}
  Let $f\colon \fM \to \fM'$ be a morphism of graded monoidal
  $\sT$-stacks.
  \begin{enumerate}[label = (\roman*)]
  \item The degree operator $z^{\deg}$ commutes with $f^*\colon
    K^\circ_\sT(\fM') \to K^\circ_\sT(\fM)$.
  \item The translation operator $D(z)$ commutes with $f_*\colon
    \bfK_\circ^\sT(\fM) \to \bfK_\circ^\sT(\fM')$.
  \end{enumerate}
  Consequently $f_*$ induces a map $f_*\colon \bfK_\circ^\sT(\fM)^\pl
  \to \bfK_\circ^\sT(\fM')^\pl$ of $\bk_\sT$-modules.
\end{lemma}

\subsubsection{}

\begin{lemma}[{\cite[\S 2.3.9]{KLT25}}] \label{lem:k-homology-pl}
  The morphism $\Pi_\alpha^\pl\colon \fM_\alpha \to \fM_\alpha^\pl$
  induces
  \[ (\Pi_\alpha^\pl)_*\colon \bfK_\circ^{\sT}(\fM_\alpha)^\pl \to \bfK_\circ^{\sT}(\fM_\alpha^\pl). \]
  If $\Pi_\alpha^\pl$ admits a section $I_\alpha$, then this is an
  isomorphism of $\bk_\sT$-modules with inverse $(I_\alpha)_*$.
\end{lemma}

Note that $\Pi_\alpha^\pl$ admits a section if and only if it is a
trivial $[\pt/\bC^\times]$-bundle, i.e. $\fM_\alpha = \fM_\alpha^\pl
\times [\pt/\bC^\times]$ \cite[Lemma 3.21]{Laumon2000}.

This lemma will be how we relate abstract invariants in
$K_\circ^{\sT}(\fM_\alpha)^\pl_{\loc}$ with the enumerative invariants
in $K_\circ^{\sT}(\fM_\alpha^\pl)_{\loc}$.

\subsubsection{}

\begin{theorem}[{\cite[Theorem 3.2.13]{Liu2022}}] \label{thm:mVOA-monoidal-stack-lie-algebra}
  The multiplicative vertex algebra structure on
  $\bfK_\circ^{\sT}(\fM)$ induces a Lie algebra structure on
  $\bfK_\circ^{\sT}(\fM)^\pl$, with Lie bracket given by
  \begin{align} 
    [\phi, \psi](\cE)
    &\coloneqq \rho_{K,z} \left(Y(\tilde\phi, z)\tilde \psi\right)(\cE) \nonumber \\
    &= \varepsilon_{\alpha\beta} \cdot \rho_{K,z} \left[ (\tilde\phi \boxtimes \tilde\psi)\left(\frac{1}{\hat{\fe}_\sT(z\Theta_{\alpha\beta})}) \otimes z^{\deg_1} \Phi_{\alpha,\beta}^*\cE\right)\right] \label{eq:monoidal-stack-lie-bracket}
  \end{align}
  for $\phi \in \bfK_\circ^{\sT}(\fM_\alpha)^\pl$ and $\psi \in
  \bfK_\circ^{\sT}(\fM_\beta)^\pl$, and $\tilde\phi \in
  \bfK_\circ^{\sT}(\fM_\alpha)$ and $\tilde\psi \in
  \bfK_\circ^{\sT}(\fM_\beta)$ are any lifts of $\phi$ and $\psi$
  respectively. Here, $\rho_{K,z}$ is the K-theoretic residue map of
  \cite[Def. 3.1.6]{Liu2022}, given by
  \[ \rho_{K,z}(f) \coloneqq z^0 \text{ coefficient of } (f_- - f_+) \]
  where $f_+$ and $f_-$ are the formal series expansions of the
  rational function $f$ around $z=0$ and $z=\infty$ respectively.
\end{theorem}

\subsection{Equivariant Homology Theories from Bivariant Theories}
\label{sec:equivariant-homology-bivariant} 
We will first briefly recall the construction of equivariant homologies of stacks from bivariant theories as explained in \cite{BB1} and summarized in \cite[Appendix B]{Bo25}. We use the term homology loosely here implying that we are dealing with some generalized (complex-oriented) cohomology theories.
\subsubsection{}
 Bivariant theories used in the present form were introduced in \cite[§2.2]{FMP}. We recall their definition, except that we allow  all morphisms here instead of restricting to some confined class of them. The same applies for independent squares.

Let $\Art$ be the category of Artin stacks and Artin morphisms between them as defined, for example, in \cite{TV,  LurieDAG, KhanEH}. In this subsection, all morphisms and diagrams are assumed to be in $\Art$, and we fix a ring $R$.
\begin{definition}
\label{def:bivariant}
    A bivariant theory for $\Art$ is a way to assign to each morphism  $\fX\xrightarrow{f}\fY$ in $\Art$ a graded $R$-module
    $
    \bB_*\big(\fX\xrightarrow{f}\fY\big)
    $
   together with the following graded $R$-morphisms:
    \begin{itemize}
\item (\textit{Composition Product}) For any pair of maps $\fX\xrightarrow{f}\fY\xrightarrow{g}\fZ$, there is a morphism
$$
\begin{tikzcd}
m: \bB_*\big(\fX\xrightarrow{f}\fY\big)\otimes  \bB_*\big(\fY\xrightarrow{g}\fZ\big)\arrow[r]& \bB_*\big(\fX\xrightarrow{g\circ f}\fZ\big)\,.
\end{tikzcd}
$$
Sometimes, we will simply write $m(a,b) = a\cdot b$ for composable elements.
\item (\textit{Base-Change Pullback}) 
Given a fiber-product diagram \[\begin{tikzcd}
  \fX' \ar{r}{}\arrow[d,"{f'}"'] & \fX\ar{d}{f} \\
  \fY' \arrow[r,"g"'] & \fY
\end{tikzcd}\,,\]a natural map $\begin{tikzcd}
 g^*:\bB_*\big(\fX\xrightarrow{f}\fY\big)\arrow[r]& \bB_*\big(\fX'\xrightarrow{f'}\fY'\big)
 \end{tikzcd}$
is given.
\item (\textit{Relative Pushforward}) For a commutative diagram 
$$
\begin{tikzcd}[column sep=small]
\fX'\arrow[dr,"{f'}"']\arrow[rr,"g"]&&\fX\arrow[dl,"f"]\\
   & \fY&
\end{tikzcd}\,,
$$
one has a natural morphism
$
\begin{tikzcd}
  g_*: \bB_*\big(\fX'\xrightarrow{f'}\fY\big) \arrow[r]&\bB_*\big(\fX\xrightarrow{f}\fY\big)
\end{tikzcd}\,.
$
\end{itemize}
These need to satisfy expected compatibilities noted down in \cite[§2.2]{FMP} which we don't recall here.
\end{definition}
Just as in \cite[§2.3]{FMP}, one can conclude that 
$$
\sE^*(\fX) := \bB_{-*}\big(\fX\xrightarrow{\id_{\fX}}\fX\big)\,,\qquad \sE_*(\fX) := \bB_*\big(\fX\to\pt\big)
$$
give rise to compatible cohomology and homology theories. In \cite{BB1}, two new assumptions are introduced when $\bQ\subset R$. They are noted down in \cite[§B.1]{Bo25}, so we only recall their suggestive names here:
\begin{itemize}
\item (\textit{Rational Künneth Isomorphism})
\item (\textit{Rational Triviality for $B\bZ_n$-Torsors})
\end{itemize}
These assumptions are needed when constructing Lie algebras from deformed vertex algebras on the equivariant homology of rigidified stacks.
\subsubsection{Equivariant homology}
Here, we will recall the limit construction used in \cite[Appendix B]{Bo25}. For representation-theoretic purposes, a more conceptual formulation naturally leads to the use of pro-systems, as explained in \cite{BB1}.

Let $\sT$ be an algebraic torus acting on $\fX$ in $\Art$. As in \cite[§2.1]{EdGrRR}, consider a good ind-system $\{E_i\sT\}_{i\in I}$ of open subsets of $\sT$-representations with free $\sT$-action and transition maps being $\sT$-equivariant embeddings. We will write
$$
B_i\sT = E_i\sT /\sT \,,\qquad [\fX/\sT]_i:=\fX\times_\sT E_i\sT\,,
$$
which leads to a natural cartesian diagram
\begin{equation}
\label{eq:XiGfiberprod}
\begin{tikzcd}[column sep=small]
   \arrow[d][\fX/\sT]_i\arrow[r]&{[\fX/\sT]}_j\arrow[d]\\
   B_i\sT\arrow[r]&B_{j}\sT
\end{tikzcd}
\end{equation}
for every transition map $E_i\sT\to E_j\sT$.

\begin{definition}[{\cite{BB1},\cite[Definition B.2]{Bo25}}]
\label{def:eqhom}
 For a fixed bivariant theory $\bB_*$, define the \textit{$\sT$-equivariant homology} $\sE^{\sT}_*(\fX)$ of $\fX$ as the graded limit
$$\sE^\sT_*(\fX):= \varprojlim_{i\in I} \bB_*\big([\fX/\sT]_i\to B_i\sT\big)$$
where transition morphisms are base change pullbacks along \eqref{eq:XiGfiberprod}. The \textit{$\sT$-equivariant cohomology} is defined by
\begin{equation}
\label{eq:defcoh}
\sE^*_\sT(\fX):=\varprojlim_{i\in I} \sE^*\big([\fX/\sT]_i\big)\,.
\end{equation}
Both $\sE^\sT_*(\fX)$ and $\sE^*_\sT(\fX)$ are considered as \textit{graded topological $R$-modules} for the graded limit topology. They are also naturally topological  $\sE^*_\sT(\pt) = \varprojlim \sE^*(B_i\sT)$ modules.
\end{definition}

\subsubsection{Operations on equivariant homology}
\label{sec:equivariantOPs}
In \cite{BB1}, we define all the necessary operations level-wise on the pro-system by checking that they are compatible with transition maps. We will only recall their construction for each level. In what follows, $\fX$ and $\fY$ are $\sT$-stacks in $\Art$ and all morphisms are $\sT$-equivariant morphisms in $\Art$.  
\begin{itemize}
\item (\textit{Pullbacks and Pushforwards})
For a morphism $f:\fX\to \fY$ one defines the $\sT$-equivariant pullback $f^*:\sE_{\sT}^*(\fY)\to \sE_{\sT}^*(\fX)$ as the limit of pullbacks along $[f/\sT]_i:[\fX/\sT]_i\to [\fY/\sT]_i$. One also constructs the pushforward
$
f_*:\sE_*^{\sT}(\fX)\to \sE_*^{\sT}(\fY)
$
using
$$
\begin{tikzcd}
([f/\sT]_i)_*:\bB_*\big([\fX/\sT]_i\to B_i\sT)\arrow[r]& \bB_*\big([\fY/\sT]_i\to B_i\sT)\,.
\end{tikzcd}
$$
\item (\textit{Cap Product})
The cap product
\begin{equation}
\label{eq:capproduct}
\begin{tikzcd}
\sE^*_{\sT}(\fX)\widehat{\otimes}_{\sE^*_\sT(\pt)}\sE^\sT_*(\fX)\arrow[r,"\cap"]&\sE^\sT_*(\fX)
\end{tikzcd}
\end{equation}
is itself obtained as a limit. Levelwise, this map acts on $\alpha_i\in \sE^*\big([\fX/\sT]_i\big)$, $a_i\in \bB_*\big([\fX/\sT]_i\to B_i\sT\big)$ by
$
\begin{tikzcd}
\alpha\otimes a_i\arrow[r,mapsto,"\cap"]&\alpha_i\cdot a_i\in \bB_*\big([\fX/\sT]_i\to B_i\sT\big).
\end{tikzcd}
$
\item (\textit{Equivariant Künneth Morphism})
Consider the diagonal action on $\fX\times \fY$. One constructs a morphism
\begin{equation}
\label{eq:equivariant-Kunneth}
\begin{tikzcd}
\sE^\sT_*(\fX)\widehat{\otimes}_{\sE^*_\sT(\pt)} \sE^\sT_*(\fY)\arrow[r,"\boxtimes_{\sT}"]&\sE^{\sT}_*(\fX\times \fY)\,.
\end{tikzcd}
\end{equation}
using the fiber-product diagram
$$
\begin{tikzcd}[column sep=small]
{\big[(\fX\times \fY)/\sT\big]_i}\arrow[d,]\arrow[r,]&{[\fY/\sT]_i}\arrow[d, "y_i"]\\
{[\fX/\sT]_i}\arrow[r,"x_i"']&B_i\sT
\end{tikzcd}\,.
$$
Taking $a_i\in \bB_{|a|}\big([\fX/\sT]_i\to B_i\sT\big)$ and $b_i\in \bB_{|b|}\big([\fY/\sT]_i\to B_i\sT\big)$, the map acts level-wise by 
$
\begin{tikzcd}
a_i\otimes b_i\arrow[r,mapsto,"\boxtimes_{\sT}"]&y_i^*(a_i)\cdot b_i = (-1) ^{|a||b|}x_i^*(b_i)\cdot a_i.
\end{tikzcd}
$
\item (\textit{Reduction to Subgroups}) 
For a subtorus $\sS\subset \sT$, one can also construct the reduction morphism
\begin{equation}
\label{eq:reduction}
\begin{tikzcd}[column sep = small]
\sE^{\sT}_*(\fX)\arrow[r,"\red^{\sS\subset \sT}_{\fX}"]&[1cm]\sE^{\sS}_*(\fX)\,,
\end{tikzcd}
\end{equation}
which we will not spell out here.
\end{itemize}
\subsubsection{}
\label{sec:Khan-bivariant}
Here, we recall the definition of relative chains from \cite{KhanEH}. Taking homology gives rise to a bivariant theory as shown in \cite{BB1}.

For a (higher) Artin stack $\fX$ denote by $\Sh(\fX)$ the derived category of sheaves of $R$-modules on $\fX$. The corresponding 6-functor formalism for $\Art$ is discussed at length in \cite{KhanEH, KhanVFC, KhanLV}, and the constituent functors are recalled briefly in \cite{BB1, Bo25}. Here, we will only need the adjoint pair
 \begin{equation}
 \label{eq:shriekpushpulladj}
        \begin{tikzcd}
            \Sh(\fX)\arrow[r, shift left=1ex, "f_!"{name=G}] & \Sh(\fY)\arrow[l, shift left=.5ex, "f^!"{name=F}]
            \arrow[phantom, from=F, to=G, , "\scriptscriptstyle\boldsymbol{\bot}"]
        \end{tikzcd}
 \end{equation}
 which exists for any $f:\fX\to \fY$ in $\Art$.\footnote{We do not mention when a functor is derived because this is automatic here.} 
 \begin{definition}[\cite{KhanEH}]
 \label{def:relhom}
Let $f:\fX\to \fY$ be a morphism in $\Art$. The \textit{relative chains over $\fY$} with coefficients in $R$ are
        $$
        C^{/\fY}_{\bullet}(\fX):= f_! f^!(\un{R}_{\fY})
        $$
        where we view the expression on the right-hand side as a chain complex by reversing the grading. The bivariant theory is defined by 
        \begin{equation}
        \label{eq:HoverS}
        \bB_*(\fX\to\fY):=H_{*}\Big(C_{\bullet}^{/\fY}(\fX)\Big)\,.
        \end{equation}
\end{definition}
 \begin{proposition}[\cite{BB1}]
    Definition \ref{def:relhom} constructs a bivariant theory with rational Künneth isomorphisms and satisfying rational triviality for $B\bZ_n$-torsors. 
\end{proposition}
We will denote the resulting equivariant (co)homology theories by $H^*_{\sT}(-),H^{\sT}_*(-)$. They have the following additional properties.
\begin{enumerate}[label=\roman*)]
\item (\textit{Non-Equivariant Limit}) By \cite[Proposition 2.8]{KhanEH}, both $H^{\sT}_*(\fX)$ and $H_{\sT}^*(\fX)$ become the usual Betti homology and cohomology of stacks as in \cite[Definition 4.2]{gross}, \cite[Example 2.4.3 (c)]{Joycehall} when $\sT=\{1\}$. 
\item (\textit{Completeness of Cohomology}) By \cite[Lemma 3.3]{KhanEH}, there is a natural isomorphism
$$
H^*\big([\fX/\sT]\big)\cong H^*_{\sT}(\fX)\,.
$$
\item (\textit{Equivariant Chern Character})
Let $\PPerf_r$ be the stack of rank $r$ perfect complexes of $\bC$-vector spaces. By i), there is an isomorphism
$$
H^*(\PPerf_r) =  R[ \ch_1,\ch_2,\cdots ]\,,
$$
where $\ch_i$ is the $i$'th Chern character of the universal rank $r$ complex on $\PPerf_r$. Let $\cE$ be a $\sT$-equivariant rank $r$ perfect complex on $\fX$, then it descends to $[\cE/\sT]$ on $[\fX/\sT]$ with the corresponding natural morphism $p_{[\cE/\sT]}:[\fX/\sT]\to \PPerf_r$. One defines the $i$'th equivariant Chern character of $\cE$ to be\footnote{Clearly $\PPerf_r$ is not in $\Art$. However, we only need cohomology $H^*(-)$ for this construction which is defined by taking hypercohomologies of constant sheaves.} 
$$
\ch^{\sT}_i(\cE):=p_{[\cE/\sT]}^*(\ch_i)\in H_{\sT}^*(\fX)\,.
$$
This extends in an obvious way to classes in $K^0_{\sT}(\fX)$.
\item (\textit{Cycle-Class Map})
Let $X$ be a proper algebraic space with a $\sT$-action. As explained in \cite[§B.5.v)]{Bo25}, there is a natural cycle-class map
\begin{equation}
\label{eq:cycle-class-map}
\begin{tikzcd}A_{*}^{\sT}(X)\arrow[r,"\operatorname{cl}"]&H^{\sT}_{2 *}(X)\end{tikzcd}\,.
\end{equation}
\end{enumerate}
\subsubsection{}
\label{sec:homology-to-khomology}
In \cite{BB1}, a full-fledged equivariant K-homology will be worked out. However, it will be also explained there that any equivariant K-homology that is compatible with Joyce's bicharacter construction in the sense of \cite{Borcherds-quantum} is not much different from the corresponding equivariant homology theory. Thus, we define the completion functor $\Phi$ on $\bZ$-graded $R$-modules by
$$
\Phi(M_*) = \varprojlim_{a\to-\infty}\bigoplus_{i\geq a}M_{2i}\,.
$$
Our new bivariant theory on $\Art$ is now given by  
$
\bK_0(\fX\to \fY):=\Phi \big(\bB_{*}(\fX\to \fY)\big).\footnote{The topology on $\bK_0(\fX\to \fY)$ is taken to be the discrete one not the limit topology.}
$ It is considered to be ungraded or, equivalently, degree 0, and we always assume $\bQ\subset R$. Note that $\Phi$ is naturally lax-monoidal: there exists a natural transformation $\Phi(M_*)\otimes_R\Phi(N_*)\to\Phi(M_*\otimes N_*)$. This implies that this theory inherits composition products, base-change pullbacks, and relative pushforwards.
\begin{proposition}[{\cite{BB1}}]
\label{prop:K-theory-wellbehaved}
 The bivariant theory axioms from \cite[§2.2]{FMP} and rational triviality for $B\bZ_n$-torsors are satisfied by $\bK_0$. Applying definition \ref{def:eqhom}
 to $\bK_0$ we set $
 \hat{K}_{\sT}^0(\fX) :=E_{\sT}^0(\fX)
 $ and $K_0^{\sT}(\fX) := E^{\sT}_0(\fX)$. Then §\ref{sec:equivariantOPs} constructs all operations on $\hat{K}_{\sT}^0(\fX),K_0^{\sT}(\fX)$.

Let $\fX\in \Art$ and $X$ be a quasi-compact locally finite algebraic space. Then there are  natural isomorphisms 
    \begin{equation}
    \label{eq:completed-explicit-example}
    K^{\sT}_0(X)\cong \prod_{i\in \bZ} H^{\sT}_{2i}(X)\qquad \textnormal{and}\qquad \hat{K}_{\sT}^0(\fX)\cong \prod_{i\geq 0} H^{2i}\big([\fX/\sT]\big)
    \end{equation}
    of $ \hat{K}_{\sT}^0(\pt)$-modules. \end{proposition}
\begin{proof}
   Here we only construct the isomorphisms \eqref{eq:completed-explicit-example}. Because $X$ is quasi-compact and locally finite type, it is finite dimensional. Let $d$ be its highest dimension and $n_i$ the dimension of $B_i\sT$. Then 
   $$
   \bB_{2p}\big([X/\sT]_i\to B_i\sT\big)\cong H_{2(p+n_i)}\big([X/\sT]_i\big) 
   $$
   is zero for $p> d$. In particular, we see that 
   $$
    \bK_0\big([X/\sT]_i\to B_i\sT\big) = \prod_{p\leq d}   \bB_{2p}\big([X/\sT]_i\to B_i\sT\big)\,.
   $$
   Since taking limits commutes and each factor on the right converges to $H^{\sT}_{2p}(X)$, we obtain the isomorphisms. The second isomorphism follows by the same argument because cohomology is supported in non-negative degrees.
\end{proof}
Using the above, we will replace any $\Theta\in K^0_{\sT}(\fX)$ by its total equivariant Chern character $\ch(\Theta)\in \wh{K}^0_{\sT}(\fX)$. We will omit writing $\ch$ in this case.
\subsubsection{}
Here, we recall the main explicit example which include quiver moduli stacks and fixed point moduli stacks -- see \cite[Example B.8.2)]{Bo25} and \cite[Definition 4.6]{Bo25}, respectively.
\begin{example}[{\cite[Example B.8]{Bo25}}]
\label{ex:trivialThom}
Consider the situation when $\sT =\bC^*$ as it is easy to extrapolate to any finite-dimensional torus and $\fX$ with trivial $\sT$-action. Let $t=e^u$ be the K-theoretic equivariant weight generating $\wh{K}^0_{\sT}(\pt)\cong [[(1-t)]]$ and assume that $\bQ\subset R$. the first-named author showed that
$$
H^{\sT}_*\big(\fX\big)\cong \varprojlim_n  H_*(\fX)\otimes_R R[u]/u^{n+1}=:H_*(\fX)\llbracket u\rrbracket^{\gr}
$$
where the limit is a graded one. The same argument also implies that
$$K^{\sT}_0(\fX)\cong K_0(\fX)\llbracket (1-t)\rrbracket\,.$$
\end{example}
We also discuss what pushforwards of fixed points may look like along their $\sT$-equivariant maps to stacks.
\begin{example}[{\cite[Example B.9]{Bo25}}]
\label{ex:equivariantpushforw}
    Consider the stack $B\bG_m$ the homology of which is given by $H_*(B\bG_m) \cong H_*(\bP^{\infty})\cong R[p]$ where $p^n =[\bP^n]\in H_{2n}(\bP^{\infty})$.
    For the trivial $\sT=\bC^*$-action this gives
    $
    H^{\sT}_*(B\bG_m) = R[p]\llbracket u \rrbracket^{\gr}.
    $ Let $\phi: \pt\to B\bG_m$ be the $\sT$-equivariant morphism such that $\phi^*(\cL)= e^u$ for the universal line bundle $\cL$ on $B\bG_m$. Let $1\in R[u] = H^{\sT}_{*}(\pt)$ be the point class. Then the first-named author showed that
$$
\phi_*(1) = \sum_{n\geq 0}p^nu^n\in R[p]\llbracket u \rrbracket^{\gr}\,.
$$
\end{example}
\subsubsection{}
We finish by discussing virtual invariants in $H^{\sT}_*(\fX)_{\loc}$ and $K^{\sT}_0(\fX)_{\loc}$ where $R=\bQ$. Here localizing means inverting all homogeneous polynomials in $H^*_{\sT}(\pt)$.
\begin{definition}
\label{def:bivariant-hom-invariants}
    Consider the situation in Definition \ref{bg:def:univ-enum-inv} where $j:M\to \fM^{\pl}_{\cat{A}}$ is again an open embedding. Then $j^*\bE^{\rig}$ is a CY4 obstruction theory on $M$, so we have virtual classes
    $$
    [M]^{\vir}\in H^{\sT}_*(M)_{\loc}\qquad\textnormal{and}\qquad [\wh{\cO}^{\vir}_M]\in G_0^{\sT}(M, \bZ[2^{-1}])\,,
    $$
    where the first class was already explained in \cite[Example 4.9]{Bo25} and gives rise to 
    $$ \langle M\rangle:= j_* [M]^{\vir}\in H_*^{\sT}(\fM^{\pl}_{\cat{A}})_{\loc}\,.$$
Consider the composition of the natural maps
$$
\begin{tikzcd}
   { G^{\sT}_0(M,\bZ[2^{-1}])}\arrow[r,"\tau^{\sT}"]& \prod_i A^{\sT}_i(M,\bQ)\arrow[r]&\Big(\prod_{i} A^{\sT}_i(M,\bQ)\Big)_{\loc}\arrow[r,"\operatorname{cl}|_{M^{\sT}}"]& K^{\sT}_0(M)_{\loc}
\end{tikzcd}
$$
where $\tau^{\sT}$ was constructed in \cite{EdGrRR}, and $\operatorname{cl}|_{M^{\sT}}$ is induced by the commutative diagram
$$
\begin{tikzcd}
\Bigl(\prod_i A_i^{\sT}(M,\bQ)\Bigr)_{\loc}
  \arrow[r, "{\operatorname{cl}|_{M^{\sT}}}"]
&
\Bigl(\prod_i H_{2i}^{\sT}(M)\Bigr)_{\loc}
  \mathrlap{\;\cong K_0^{\sT}(M)_{\loc}}
\\
\Bigl(\prod_i A_i^{\sT}(M^{\sT},\bQ)\Bigr)_{\loc}
  \arrow[u, "\sim" {rotate=90, yshift=1ex, xshift= 1ex}]
  \arrow[r, "\operatorname{cl}"]
&
\Bigl(\prod_i H_{2i}^{\sT}(M^{\sT})\Bigr)_{\loc}
  \arrow[u, "\sim" {rotate=90, yshift=-1ex, xshift= 1ex}]
\end{tikzcd}
$$
 Applying this composition to $[\wh{\cO}^{\vir}_M]$, we get 
$$
Z_M \in K^{\sT}_0(M)_{\loc}\qquad \textnormal{and}\qquad 
j_*Z_M\in K^{\sT}_0(\fM^{\pl}_{\cat{A}})_{\loc}\,.
$$
\end{definition}
\subsection{Vertex Algebras and Their $\sT$-deformations}\label{bg:sec:va-and-t-deformations}
\subsubsection{}
We begin by recalling Joyce's vertex algebra in the CY4 setting as noted down in \cite[§3.2]{Joycehall} and \cite[§4.4]{GJT}. More precisely, we consider its $\sT$-deformation on $H^{\sT}_*(\fM_{\cat{A}})$ as constructed in \cite[§4.3]{Bo25} which recovers the non-equivariant version by the non-equivariant limit property in §\ref{sec:Khan-bivariant}. The axioms of $\sT$-deformations will appear and be checked in \cite{Bo26} (see also \cite[p. 33]{BoVA-conference-talk}). Unlike these references, we write everything multiplicatively here.

In what follows $\vdim = \rk(\bE): \fM_{\cat{A}}\to \bZ$ is a locally constant function determined by the obstruction theory \eqref{eq:EE=Delta-Theta}. It therefore allows us to shift homological degrees on each connected component of $\fM_{\cat{A}}$. For a local version of the following definition constructed on the fixed-point stack $\fM^{\sT}_{\cat{A}}$ see \cite[Definition 4.6, Definition 4.7]{Bo25}.
\begin{definition}[{\cite[Definition 4.8]{Bo25}, \cite[§3.2]{Joycehall}, \cite[§4.4]{GJT}}]
\label{Def:VAglobal}
    The \textit{$\sT$-deformation of vertex algebras} on 
    $$V_{*}:= H^{\sT}_{*+\vdim}(\fM_{\cat{A}})$$
    is determined by the data $\big(V_{*}, \ket{0}, D, Y\big)$ given as follows:
\begin{enumerate}
    \item Using the inclusion $0\colon*\to \fM_{\cat{A}}$ of the point corresponding to the zero object, set
    $$
    \ket{0}:=0_*(*)\in H^{\sT}_0(\fM_{\cat{A}})\,.
    $$
    \item Let $\Psi_*: H^{\sT}_*(B\bG_m)\hat{\otimes}_{H^*_{\sT}(\pt)} H^{\sT}_*(\fM_{\cat{A}})\to  H^{\sT}_*(\fM_{\cat{A}})$ be the push-forward along the second map in \eqref{eq:murho} after applying \eqref{eq:equivariant-Kunneth}. By Example \ref{ex:trivialThom}, we can identify the source of $\Psi_*$ with  $H_*(B\bG_m)\hat{\otimes}_R H^{\sT}_*(\fM_{\cat{A}})$. Using
    $\Hom_R\big(H_*(B\bG_m),R\big) =\wh{K}^0(B\bG_m) = R\llbracket (1-z)\rrbracket$, where $z$ is the class of the universal $\cL$ on $B\bG_m$, we get the map
    $$
    D(z) : H^{\sT}_*(\fM_{\cat{A}})\to  H^{\sT}_*(\fM_{\cat{A}})\llbracket (1-z)\rrbracket\,.
    $$
    By the same construction, the $B\bG_m$-action on the first factor of $\fM_{\cat{A}}\times \fM_{\cat{A}}$ produces the operator $D(z)\otimes \id$ acting on $H^{\sT}_*\big(\fM_{\cat{A}}\times \fM_{\cat{A}}\big)$.
    \item Let  $
    \varepsilon : K^0(\cat{A})\times K^0(\cat{A})\to \{-1,+1\}
    $ be a group 2-cocycle from Definition \ref{def:group-2-cocycle} and set $\bT:=\bG_m\times \sT$. Define an $H_{\sT}^*(\pt)$-bilinear $\sT$-deformations of state-field correspondences $Y: V_*\hat{\otimes}_{H_{\sT}^*(\pt)} V_*\to V_* \llbracket (1-z)^{\pm 1}\rrbracket$ by 
    \begin{align*}
    Y(v,z)w &= (-1)^{a\chi(\beta,\beta)}\varepsilon_{\alpha,\beta}\, \Phi_\ast\left((D(z)\otimes \textnormal{id})\frac{v\boxtimes_{\sT} w}{\fe_{\bT}(z\Theta_{\cat{A}})}\right)
    \end{align*}
   for any $v\in H^{\sT}_{a}(\fM_\alpha)$ and $w\in H^{\sT}_{*}(\fM_{\beta})$ where $\al,\beta\in K^0(\cat{A})$. Here $\fe_{\bT}(z\Theta_{\cat{A}})^{-1}$ is always expanded to include negative powers of $(1-z)$. 
\end{enumerate}
\end{definition}
If $z=e^s$, then $\fe_{\bT}(z\Theta_{\cat{A}}) = s^{\rk(\Theta_{\cat{A}})}c_{s^{-1}}(\Theta_{\cat{A}})$. To obtain the above expansion in $(1-z)$ one substitutes $s^{-1} = (1-z)^{-1}F(1-z)$ where $F$ is an explicit power-series with constant term $-1$, thus invertible. 
\subsubsection{}
That Definition \ref{Def:VAglobal} gives a graded vertex algebra when $\sT =\{1\}$ was proved in \cite[Theorem 3.14]{Joycehall}. Let 
$$V_{i,*}:= \bB_{*+\vdim}\big([\fM_{\cat{A}}/\sT]_i\to B_i\sT\big)$$
be one of the terms used in Definition \ref{def:eqhom} to define the equivariant homology $V_*$. Then one of the conditions on a $\sT$-deformation of vertex algebras introduced in \cite{Bo26, BB1}\footnote{See also \cite[p. 33]{BoVA-conference-talk}.} requires that each $V_{i,*}$ is a graded vertex algebra and the transition morphism between them are morphisms of graded vertex algebras. In the present case, this is satisfied essentially by \cite[Theorem 3.14]{Joycehall}. This axiom on its own is a generalization of the usual notion of deformations of vertex algebras as introduced in \cite[§5]{HaiLi}, however the actual definition of $\sT$-deformations requires an additional condition to be checked. 
\begin{theorem}[{\cite{Bo26, BB1}}]
\label{thm:homology-T-deformed-VA}
    The data $\big(V_{*}, \ket{0}, D, Y\big)$ from Definition \ref{Def:VAglobal} gives rise to a $\sT$-deformation of to algebras.
\end{theorem}
If $\sS=\{1\}$ and we apply $\red^{\sS\subset \sT}_{\fX}$ from \eqref{eq:reduction} to $V_*$ and its operations, we recover Joyce's vertex algebras as the non-equivariant limit. Thus our construction is truly a deformation of his.
\subsubsection{}
\label{sec:Khomology-Tdeformation}
We now discuss the K-theoretic version of the above construction and use equivariant K-homology from §\ref{sec:homology-to-khomology}. For degree reasons, we restrict to connected components of $\fM_{\cat{A}}$ corresponding to even classes $K^0_{e}(\cat{A})$ only.

Let $V^K_0:= K^{\sT}_0(\fM_{\cat{A}})$, then we construct the data $(V^K_0,\ket{0},D,Y^K)$ below:
\begin{enumerate}
    \item The vacuum $\ket{0}$ is constructed in the same way.
    \item Using Example \ref{ex:trivialThom}, there a natural inclusion $K_0(B\bG_m)\hookrightarrow K^{\sT}_0(B\bG_m)$, composing it with $\Psi_*$ from before, we get a map
    $
    K_0(B\bG_m)\otimes_R K^{\sT}_0(\fM_{\cat{A}})\to K^{\sT}_0(\fM_{\cat{A}})
    $
    inducing
      $$
    D(z) : K^{\sT}_0(\fM_{\cat{A}})\to  K^{\sT}_0(\fM_{\cat{A}})\llbracket (1-z)\rrbracket\,.
    $$
    \item The $\sT$-deformation of state-field correspondences is given in the same way except that we use $\hat{\fe}_{\bT}$ instead of the usual equivariant Euler-class:
     \begin{align*}
    Y^K(v,z)w &= \varepsilon_{\alpha,\beta}\, \Phi_\ast\left((D(z)\otimes \textnormal{id})\frac{v\boxtimes_{\sT} w}{\hat{\fe}_{\bT}(z\Theta_{\cat{A}})}\right)\,.
    \end{align*}
    We still expand $\hat{\fe}_{\bT}(z\Theta_{\cat{A}})^{-1}$ in negative powers of $(1-z)$ so that it is an element of $K^0_{\sT}(\fM_{\cat{A}}\times \fM_{\cat{A}})\llbracket(1-z)^{-1} ,(1-z)\rrbracket[z^{-\frac{1}{2}}, z^{\frac{1}{2}}]$.
\end{enumerate}
Explicitly, this expansion can be written in terms of the Connor-Floyd classes via
$$
\fe_{\bT}(z\Theta_{\cat{A}})^{-1} = (1-z^{-1})^{-\rk(\Theta_{\cat{A}})}\sum_{k\geq 0}\fc_k(-\Theta_{\cat{A}})(1-z^{-1})^{-k}
$$
where $(1-z^{-1}) = -z^{-1}(1-z)$ and $z^{-1}$ is expressed as an invertible power-series in $(1-z)$.
\begin{theorem}[{\cite{BB1}}]
The data $(V^K_0,\ket{0},D,Y^K)$ determines a $\sT$-deformation of vertex algebras. This in particular implies that $\bK_0\big([\fM_{\cat{A}}/\sT]_i\to B_i\sT\big)$ form a pro-system of vertex algebras.
\end{theorem}
Both this theorem and Theorem \ref{thm:homology-T-deformed-VA} apply to the setting in §\ref{sec:flag-VA+LA} giving $\sT$-deformations of vertex algebras on the corresponding equivariant homology of $\fM^{\bar{Q}(r)}_{\al,\vec d}$.
\subsubsection{}
\label{sec:Lie-algebras}
Let $z=e^u$ in the above. Define the operators $T$ on $V_*$ and $V^K_0$ by requiring that $e^{uT} = D(z)$. This is called a \textit{translation operator}. We introduce the quotients 
$$
L_*:= V_{*+2}/T V_*\qquad\textnormal{and}\qquad L^K_0 := V^K_0/TV^K_0
$$
For the presently considered equivariant homology theories, these quotients have a nice geometric interpretation. 
\begin{proposition}[{\cite[Proposition B.7]{Bo25},\cite{BB1}}]
\label{prop:quotient-by-T}
     Let $\alpha\in K^0(\cat{A})$ be such that for some $\beta\in K^0(\cat{A})$, one has
$
\chi(\beta,\alpha)\neq 0$ and $\fM^{\sT}_{\beta}\neq  \emptyset$.
Then, there are natural isomorphisms  $$
    H^{\sT}_*\big(\fM_{\alpha}\big)/T H^{\sT}_{*-2}\big(\fM_{\alpha}\big)\cong H^{\sT}_*\big(\fM^{\pl}_{\alpha}\big)\qquad \textnormal{and}\qquad K^{\sT}_0\big(\fM_{\alpha}\big)/TK^{\sT}_0\big(\fM_{\alpha})\cong H^{\sT}_*\big(\fM^{\pl}_{\alpha}\big)
    $$ whenever additionally $\bQ\subset R$. 
\end{proposition}
Moreover, both $L_*$ and $L^K_0$ carry a Lie algebras structure determined by
\begin{equation}
\label{eq:LiefromVA}
\big[\bar{v},\bar{w}\big] = -\bar{[(1-z)^{-1}]\Big\{z^{-1}Y(v,z)w\Big\}}\,,
\end{equation}
where $\bar{(-)}$ denotes the projection to the quotient and we use the above expansions in (possibly negative) powers of $(1-z)$. More generally, we define $\rho_z$ here by taking the coefficient of $(1-z)^{-1}$ in such an expansion. In terms of the variable $s$ in $z=e^s$ this would simply correspond to taking 
$$
\big[\bar{v},\bar{w}\big] = \bar{[s^{-1}]\Big\{Y(v,z)w\Big\}}\,,
$$

\subsubsection{}
\label{sec:localized-LA}
We can now localize all of our $H_{\sT}^*(\pt)$-modules. Let us write $V_{\loc,*}$ and $V^K_{\loc, 0}$ for the resulting $\sT$-deformations of vertex algebras on $ H^{\sT}_{*+\vdim}(\fM_{\cat{A}})_{\loc}$ and $K^{\sT}_0(\fM_{\cat{A}})_{\loc}$, respectively. The associated Lie algebras will be denoted by $L_{\loc,*}$ and $L^K_{\loc,0}$.

Suppose that in Definition \ref{def:bivariant-hom-invariants}, the map $j:M\to \fM^{\pl}_{\cat{A}}$ factors through $\fM^{\pl}_{\al}$ where $\al\in K^0(\cat{A})$ satisfies the condition of Proposition \ref{prop:quotient-by-T}. Then the proposition gives us elements 
$$
\langle M\rangle \in L_{\loc,*}\qquad\textnormal{and}\qquad j_*Z_M\in L^K_{\loc,0}\,.
$$

\section{Assumptions}
\label{sec:assumptions-all}
\subsection{Orientation Assumptions}
\subsubsection{}
\label{sec:epsilonorientations}
Recall that an orientation of $\fM_{\cat{A}}$ is an isomorphism $ \un{\bC}\xrightarrow{o} \det\big(\bE\big)$ satisfying  condition \eqref{eq:EEorient} with respect to the Serre duality induced by the CY4 structure. For each $\alpha\in \scE(\cat{A})$, we will denote the restriction of $o$ to $\fM_{\alpha}$ by $o_{\alpha}$. A compatibility of the orientations $o_{\alpha}$ under direct sums was observed in \cite{JTU, Joycehall}. However, the conventions used there would clash with those applied to equivariant localization in \cite{Oh2023} and those set in \cite[§3.3]{Bo25} (see \cite[Remark 4.3]{Bo25}), which is why we use the construction in \cite[§4.1]{Bo25} instead. It produces an isomorphism comparing orientations under taking direct sums.

Suppose that orientations $o_{\al},o_{\beta},$ and $o_{\al+\be}$ exist. The comparison determines a locally constant function $$\varepsilon_{\alpha,\beta}:\fM_{\al}\times \fM_{\beta}\to \{-1,+1\}$$
such that
\begin{equation}\label{as:eq:orientation-sum-discrepancy}
    o_{\al}\boxtimes o_{\beta} =\varepsilon_{\alpha,\beta}\mu^* o_{\al +\be }
\end{equation} 
holds. Whenever they are defined, these signs satisfy the properties of a group 2 cocycle from \eqref{eq:epsidentity} for classes $\al,\beta,\gamma\in \scE(\cat{A})$.

The first assumption on $\cat{A}$  is orientability with further restriction on $\varepsilon_{\al,\beta}$.
\begin{assumption}
\label{ass:orientation}
We assume that $\fM_{\alpha}$ are orientable for each $\alpha\in \scE(\cat{A})$.  Furthermore, there should be a uniform way of choosing orientations $o_{\alpha}$
    on each stack $\fM_{\alpha}$ independent of the connected components of $\fM_{\alpha}$. Such a choice of $\{o_{\alpha}\}_{\alpha\in \scE(\cat{A})}$ should lead to $\varepsilon_{\alpha,\beta}$ that are constant for any $\alpha,\beta\in \scE(\cat{A})$. 
    \end{assumption}
    \subsubsection{}
    \begin{example}
        \label{ex:sheavesandreps0}
    The main example we consider here is $\cat{A}=\Coh(X)$ for a Calabi--Yau fourfold $X$ -- see \cite[Example 4.5.1)]{Bo25}.
     \begin{enumerate}
         \item  If $X$ is not proper, then one needs to restrict the description to compactly supported sheaves by replacing $\Coh(X)$ with $\Coh_{\cs}(X)$. Its moduli stack will be denoted by $\fM_X$. In both cases, orientations have been proved in \cite{CGJ,bojko} by reducing to gauge-theoretic orientations. Recently, a correction of the proof of the existence of gauge-theoretic orientations has appeared in \cite{JU}. This puts the restriction \cite[($\ast$)]{JU} on $H^3(X,\bZ)$ in the compact case, which impacts the original statements in  \cite{CGJ, bojko}. The sufficient condition for orientations to exist as shown in \cite{bojko,karpov-thimm} using \cite{JU}, is expressed in terms of the compactly supported Steenrod square operation $\operatorname{Sq}^2_{\cs}: H^3_{\cs}(X,\bZ_2)\to  H^5_{\cs}(X,\bZ_2)$. It requires for any $\sv\in H^3_{\cs}(X,\bZ)$ and its image $\bar{\sv}\in H^3_{\cs}(X,\bZ_2)$ that 
         \begin{equation}
         \label{eq:H3ZZ_2vanishing}
\int_X \bar{\sv}\cup \operatorname{Sq}^2_{\cs}(\bar{\sv}) = 0\,.
         \end{equation}
         In particular, this is always satisfied when $H^3_{\cs}(X,\bZ_2) = 0$.
         
       Once orientations are chosen, the compatibility required in Assumption \ref{ass:orientation} follows from \cite[Theorem 1.15.(c)]{CGJ} and  \cite[Theorem 5.4]{bojko} as long as the natural map 
         $$
         \begin{tikzcd}
         K^0\big(\Coh_{\cs,e}(X)\big)\arrow[r]& K^0_{\cs,e}(X)
         \end{tikzcd}
         $$
         factors through the chosen quotient $\bar{K}(\Coh_{\cs}(X))$.\footnote{For the non-compact case, we additionally need $H^3(X,\bZ_2) = 0 = H^3(X,\bZ_2)$ as explained in \cite[Theorem 1.7]{bojko}} Here, the additional subscript $e$ picks out the classes $\al$ with $\chi(\al,\al)$ even.
     \end{enumerate} 
    \end{example}
\subsubsection{}
Recall the framed stacks from §\ref{sec:auxiliary-stacks} for a given quiver $Q$ as in Definition \ref{def:auxiliary-stack}. Following \cite[Lemma 8.4]{Bo25}, we introduce additional signs for $(\al,\vec d), (\beta,\vec e)\in \scE(\cat{A})\times \bZ^{Q^f_0}$ which are used for wall-crossing that takes place in $\fM^{Q(\vec\Fr)}$ (see §\ref{sec:horizontalflagWC}). 
\begin{definition}[{\cite[(8.9)]{Bo25}}]
\label{def:framed-epsilons}
    Let $\bF_{\al\beta}^{Q(\vec \Fr)}(\vec d,\vec e)$ be the restriction of \eqref{eq:framed-stack-forgetful-map-cotangent} to $\fM^{Q(\vec\Fr)}_{\alpha,\vec d}\times \fM^{Q(\vec\Fr)}_{\be,\vec e}$. Then using the signs from Assumption \ref{ass:orientation}, we define
    $$\varepsilon_{(\al,\vec d)}^{(\be, \vec e)}\coloneqq (-1)^{\rk\big(\bF_{\alpha\beta}^{Q(\vec \Fr)}(\vec d,\vec e)\big)}\varepsilon_{\al,\be}\,.$$
\end{definition}
 \subsection{Stability Assumptions}
    \subsubsection{}
    \label{sec:weak-stability}
In the current work, we choose to work in the largest possible generality when it comes to stability conditions. Thus, we use Joyce's \textit{weak stability conditions} from \cite{JoyceIII} that impose almost no restrictions. Recall that such a stability condition $\tau$ is determined by a map 
$$
\begin{tikzcd}
\tau: \bar{C}(\cat{A})\arrow[r]&S
\end{tikzcd}
$$
where $(S,\leq)$ is a totally ordered set. For it to be called a weak stability condition, it should satisfy 
$$
\tau(\al_1)\leq \tau(\al)\leq \tau(\al_2)\qquad\text{or}\qquad \tau(\al_1)\geq \tau(\al)\geq \tau(\al_2)
$$
whenever $\al=\al_1 + \al_2$ for $\al_1,\al_2\in \bar{C}(\cat{A})$. In this case, an object $E\in \cat{A}$ is said to be   $\tau$-semistable if for any short exact sequence 
$$  
0\to E_1\to E\to E_2\to 0
$$
in $\cat{A}$, one has $\tau(E_1)\leq \tau(E_2)$. If the strict inequality always holds, then $E$ is said to be $\tau$-stable. We will denote by
$$\fM_{\al}(\tau)\subset \fM_{\cat{A}}$$
the substacks consisting of $\tau$-semistable objects of class $\al\in \bar{K}(\cat{A})$. We fix a connected set $W$ of such weak stability conditions for a fixed $S$ with $W$ being a finite-dimensional manifold. The last condition is not necessary but makes formulating assumptions later cleaner. 
\subsubsection{}
\label{sec:W-space-of-stability-cons}
    For a fixed $\al\in \scE(\cat{A})$, we require that the map
    $$
    	W\ni\tau \mapsto \tau(\al)\in S
    $$
 is continuous with respect to the order topology on $S$. For any $\al,\beta\in \scE(\cat{A})$, this ensures that 
    \begin{itemize}
    	\item the set $W_{\al<\be} = \{\tau \in W \colon \tau(\al)<\tau(\be) \}$ is open,
    	\item the set $W_{\al=\be} = \{\tau \in W \colon \tau(\al)=\tau(\be) \}$ is closed.
    \end{itemize}
    In the cases one usually considers, the sets $W_{\al=\beta}$ are finite unions of real codimension 1 loci in $W$ or the entirety of $W$ (e.g. when $\beta =n\al$).  We will state more general assumptions below. 
    
   Let $\vec{\alpha} = (\alpha_1,\ldots,\alpha_n)$ for $\alpha_i\in \bar{C}(\cat{A})$ satisfy $\sum_{i=1}^n \al_i= \al$. Then we will say that it is a partition of $\alpha\in \bar{C}(\cat{A})$ which we will denote by $$\vec{\alpha} \vdash_{\cat{A}} \alpha\,.$$
   To write down enumerative wall-crossing formulae, one needs to make sure that there are finitely many non-zero coefficients  $\tilde{U}(\vec{\al};\tau,\tau')$ in the sum in Theorem \ref{sec:general-wall-crossing-intro}. For a fixed $\al\in \scE(\cat{A})$, this requires the set of \footnote{Here $U(-,\tau,\tau')$ is the other related type of coefficients discussed in \cite[§3.2]{Joyce2021}. This finiteness condition is required for the general wall-crossing formula. However, as noted in \cite[\S 1.3.9, \S 1.3.9]{KLT25}, the existence of a rank function as below guarantees finiteness of the sum in the dominant wall-crossing formula.}
   \begin{equation}\label{eq:wc-index-set-finite}
    \left\{\vec\beta\vdash_{\cat{A}}\alpha\ \Bigg|\ \begin{subarray}{c}\exists\vec \alpha_j\vdash_{\cat{A}}\beta_j, \tau,\tau'\in W\text{, such that}\\\fM_{\alpha_{jk}}(\tau)\neq \emptyset,\, U(\vec \alpha_j;\tau,\tau')\neq 0\,\forall\, j,k\\\text{and }\tau'(\beta_1)=\cdots=\tau'(\beta_p)
    \end{subarray}
    \right\}
   \end{equation}
    to be finite. Note that this implies that the set of $(\al_1,\al_2)\vdash_{\cat{A}}\al$ such that for some $\tau\in W$ one has $\tau(\al_1)=\tau(\al) = \tau(\al_2)$ and $\fM_{\al_{1}}(\tau)\neq \emptyset\neq \fM_{\al_{2}}(\tau)$ is finite due to $U(\al_i;\tau,\tau) = 1$. Thus there are finitely many sets $W_{\al_1=\al_2}$ where objects in class $\al$ can be destabilized.
    \subsubsection{}
    \begin{definition}
    \label{def:framedstability}
      We fix $\tau\in W$ and a value $\phi\in S$. We then define the full exact subcategories
       \begin{equation}
    \label{eq:tau-semi-phi-fixed}   \cat{A}^{Q(\vec\Fr)}_{\tau,\phi}\subset \cat{A}^{Q(\vec\Fr)}\end{equation}
       consisting of objects $(E,
  \vec V, \vec\rho)$ such that $E$ is $\tau$-semistable and $\tau(E) =\phi$. For the purpose of writing stability conditions on $\cat{A}^{Q(\vec\Fr)}_{\tau,\phi}$, one replaces $\bar{K}(\cat{A})$  by 
  \begin{equation}
  \label{eq:barK-of-auxiliary}
\bar{K}\big(\cat{A}^{Q(\vec\Fr)}_{\tau,\phi}\big):= \bar{K}(\cat{A})\times \bZ^{Q^{f}_0}\,.\end{equation}
Let $\lambda: \bar{K}(\cat{A})\to \bR$ be a group homomorphism, $\rk_{\tau}:\scE(\cat{A})\to \bN$\footnote{For us, $\bN$ will always mean strictly positive integers.} a map as in Assumption \ref{ass:stab} (e), and $\vec{\delta}\in \bR^{Q^{f}_0}$ a vector. The stability condition $\tau^{\lambda}_{\vec\delta}$ with values in the extended real numbers $\bar{\bR}$ is determined by\footnote{Taking the subcategory of semistable objects with fixed slope, addresses an issue with \cite[(5.21)]{Joyce2021}, and makes the above definition a weak stability condition. See \cite[(89)]{KLT25} for a slightly different approach to the same problem.}
\begin{equation}
\label{eq:barphi}
\tau^{\lambda}_{\vec\delta}: \bar{K}\Big(\cat{A}^{Q(\vec\Fr)}_{\tau,\phi}\Big)\to\bar{\bR}\,,\qquad
\tau^{\lambda}_{\vec\delta}(\al,\vec{d}) =\begin{cases}\frac{\lambda(\al) + \vec{d}\cdot \vec{\delta}}{\rk_{\tau}(\al)}&\text{if }\al\neq 0\\
\infty &\text{if }\alpha=0,\ \vec d\cdot \vec \delta>0\\
-\infty &\text{if }\alpha=0,\ \vec d\cdot \vec \delta\leq 0\end{cases}
\,.
\end{equation}
    \end{definition}
    Fix $(\alpha,\vec{d})$ with $\al\in \scE(\cat{A})$ and $\tau(\al) = \phi$, the $\tau^{\lam}_{\vec\delta}$-semistable objects of this class form an open substack $\fM^{Q(\vec\Fr)}_{\al,\vec{d}}(\tau^{\lam}_{\vec{\delta}})\subset \fM^{Q(\vec\Fr)}_{\al,\vec{d}}.$ When there are no strictly $\tau^{\lambda}_{\vec\delta}$-semistable objects of class $(\alpha,\vec{d})$, we will write \begin{equation}
\label{eq:framedstablemodulispace}
M^{Q(\vec\Fr)}_{\al,\vec{d}}(\tau^{\lambda}_{\vec\delta})\subset\fM^{Q(\vec\Fr),\rig}_{\al,\vec{d}}
    \end{equation}
    for the resulting moduli spaces. These moduli spaces inherit the $\sT$-action from the moduli stack $\fM_{\cat{A}}$ with an additional $\bG_m$-action considered in the case of quivers \eqref{eq:Vshapedquiver} and \eqref{wc:fig:wc-ms-quiver} where it acts by rescaling an edge (represented by a purple arrow). 
    
\subsubsection{}\label{ass:sec:invariants-assumptions}
    The assumptions below guarantee the existence of the invariants $z_{\al}(\tau)\in L(\fM_{\cat{A}})_{\loc}$ from Theorem \ref{thm:sst-invariants} for all $\al\in \scE(\cat{A})$ and $\tau$. They are also used in the proof of wall-crossing in §\ref{sec:wall-crossing} for these invariants.
    
    \begin{assumption}
    	\label{ass:stab}
    	Let $W$ be the manifold of weak stability conditions fixed above. 
    	\begin{enumerate}[label =(\alph*)]
    		\item If $(\al_1,\al_2) \vdash_{\cat{A}}\al\in \scE(\cat{A})$ and $\tau(\al_1)=\tau(\al_2)$ for some $\tau\in W$, then $\al_1,\al_2\in \scE(\cat{A})$.
    		\item For each pair $\tau,\tau'\in W$, there is a continuous path $\gamma_{(-)}:[0,1]\to W$ between them, such that the open subset 
    		$$
    		\gamma_{\al<\beta} = \big\{t\in [0,1]\colon \gamma_t(\al)<\gamma_t(\beta)\big\}\subset \gamma_{[0,1]}
    		$$
    		is a finite union of connected components for each $\al,\beta\in \scE(\cat{A})$. 
    		
    		We will say that $\gamma_{(-)}$ satisfies (P), if for each set of data \eqref{eq:wc-index-set-finite} such that $\gamma_t(\beta_j)\neq \gamma_t(\al)$ for some $j\in\{1,\ldots,p\}$ and $t\in [0,1]$ the set 
\begin{equation}
\label{eq:betajequalsal}
\gamma_{\be_j=\al} = \big\{s\in [0,1]\colon \tau_s(\be_j)=\tau_s(\al)\text{ for  }\gamma_s \text{ and all }j=1,\ldots,p\big\}
\end{equation}
is finite.
    		\item The set of \eqref{eq:wc-index-set-finite} is finite for a fixed $\al\in \scE(\cat{A})$.
    		\item 
    		Suppose that (P) holds for a fixed path $\gamma_{(-)}$. Fix $\al\in \scE(\cat{A})$ and $t\in [0,1]$. Define the subset $B_{\al,t}\subset \scE(\cat{A})$ consisting of all $\beta$ such that 
    		\begin{equation}
    			\label{eq:Balt}
    			(\beta,\al-\be)\vdash_{\cat{A}}\al\,,\qquad \tau_{t}(\be) = \tau_{t}(\al-\beta)\,,\quad\text{and}\quad \cM^{\gamma_t}_{\be}\neq \emptyset \neq \cM^{\gamma_t}_{\al-\be}\,.
    		\end{equation}
    		If $t'\neq t$ and $t'\notin \gamma_{\be_j=\al}$ for all data \eqref{eq:wc-index-set-finite}, there exists a group homomorphism $\lambda^{t,t'}_{\al}: \bar{K}(\cat{A})\to \bR$ such that $\lambda^{t,t'}_{\al}(\al)=0$ and such that for all $\beta\in B_{\al,t}$ we have
    		\begin{equation}
    			\label{eq:lambdaiffphi}
    			\lambda^{t,t'}_{\al}(\beta)<\lambda^{t,t'}_{\al}(\al-\beta)  \qquad \iff\qquad \tau_{t'}(\beta)<\tau_{t'}(\al-\beta).
    		\end{equation}

    		If $\ga_{(-)}$ does not satisfy $(P)$, $\lambda^{t,t'}_\al$ should exist for all $t'\neq t$.
    		\item For each $\tau\in W$, there exists a \textit{rank function} $\rk_{\tau}: \scE(\cat{A})\to \bN$ satisfying $\rk_{\tau}(\al) = \rk_{\tau}(\al_1) + \rk_{\tau}(\al_2)$ whenever $(\al_1,\al_2)\vdash_{\cat{A}}\al \in \scE(\cat{A})$ and $\tau(\al_1) = \tau(\al_2)$.
    		\item Any $\tau\in W$ satisfies the Harder-Narasimhan property on $\cat{A}$.
    		\item For each $\alpha\in \scE(\cat{A})$ and $\tau\in W$, the substacks $\fM_{\al}(\tau)\subset \fM_{\cat{A}}$ are open and finite type. If additionally there are no strictly $\tau$-semistable objects of class $\alpha$, the rigidification $$M_{\al}(\tau) := \big(\fM_{\al}(\tau)\big)^{\rig}$$ is a quasi-compact separated algebraic space with proper fixed-point locus $\big(M_{\al}(\tau)\big)^{\sT}$ and the $\sT$-equivariant resolution property. 
            \item For each $\tau\in W$ and $\al\in \scE(\cat{A})$ there should exist a framing functor $\Fr$ as in Definition \ref{bg:def:framing-functor} such that $\fM_{\al}(\tau)\subset \fM^{\Fr}_{\al}$.
            \item Consider a path $\gamma$ given in (b) and fix $t\in [0,1]$. In the situation of Definition \ref{def:framedstability} set $\tau = \gamma_t$ and $\lambda = s\lambda^{t,t'}$ from (d) for some $s,t'\in [0,1]$. Let $\Fr$ be chosen as in (h). For any $(\al,\vec{d})\in \scE(\cat{A})\times\bZ^{Q^{f}_0}$ and $\vec{\delta}\in \bR^{Q^{f}_0}$ so that there are no strictly $\tau^{\lambda}_{\vec{\delta}}$-semistable objects of this class, we require that \eqref{eq:framedstablemodulispace} is a quasi-compact, separated algebraic space with $\bT:=\bG_m\times \sT$-equivariant resolution property. 

Let $\si:\sT\to \bT$ be an inclusion such that $\pi_1\circ \si = \id_{\sT} $ for the projection $\pi_1:  \bG_m\times \sT\to \sT$. Then the fixed point loci
$\Big(M^{Q(\vec{\Fr})}_{\al,\vec{d}}(\tau^{\lam}_{\vec{\delta}})\Big)^{\si(\sT)}$ are required to be proper.
    	\end{enumerate}
    \end{assumption}

    \subsection{Modification for Fixed Determinant Pairs}
    \subsubsection{}
   Here we summarize the modification of Assumption \ref{ass:stab} following \cite[§6.4]{Bo25}. Let $X$ be a projective Calabi-Yau fourfold and write $\cat{A} = \Coh(X)$. we fix here another heart $\cat{B}$ in $D^b(X)$ with the corresponding data of Definition \ref{def:categoryA} satisfying Assumption \ref{ass:orientation}.  Note that this implies that we have fixed some $\bar{K}(\cat{A}) = \bar{K}(B)$. We choose a space $W$ of stability conditions so that $\scE(\cat{B})$ satisfies Assumption \ref{ass:stab} (a). Here, we reformulate the rest of Assumption \ref{ass:stab} so that it can be applied to wall-crossing in $\cat{B}$ for stable pairs $O\xrightarrow{s}F$ where $O\in \Coh(X)$ and $F$ has positive codimension so lies in $\Coh_{<4}(X)$. Hence, we will use the fixed determinant moduli stack $\fM_{\cat{B},\det(O)}$ and its substacks which will always be labeled by attaching subscript $\det(O)$.   
\begin{assumption}
\label{ass:pairWC}
    Suppose that there is an object $O\in \Coh(X)$ and consider the abelian category $\cat{B}_O$ of triples $(V,F,s)$ where $V\in \cat{Vect}$, $F\in\Coh_{<4}(X)$ and $s: V\otimes O\to F$ is a morphism of sheaves.  Let $\fN_O$ be the moduli stack of $\cat{B}_O$. Suppose that there exists a space of weak stability conditions $W^p$ on $\cat{B}_O$ with a homeomorphism $(-)^p: W\to W^p$. Denote by $\fN_{d,\al}(\tau^p)\subset \fN_O$ the moduli substack of $\tau^p$-semistable $(V,F,s)$ with $\dim(V) = d$ and $ \al\in \bar{C}(\cat{A})$ the class of $F$. This data must satisfy the following:
    \begin{enumerate}[label=(\alph*)]
        \item For all $\beta\in \scE(\cat{B}),\tau\in W$, there exists $d \in\{0,1\}$ and $\al\in\bar{C}\big(\cat{A}\big)$ such that there are isomorphisms of stacks
        \begin{equation}
        \label{eq:pairtocomplexiso}
       \big(\fN_{d,\al}(\tau ^p)\big)^{\pl}\cong \begin{cases}
         \fM_{\be}(\tau)_{\det(O)}  &\text{if }d=1\,,\\
         \\
    \big(\fM_{\be}(\tau)\big)^{\pl}  &\text{if }d=0 \,,
        \end{cases}
        \end{equation}
        induced by the functor $\varsigma: \cat{B}_O\to \cat{B}$ mapping each $V\otimes O\xrightarrow{s}F$ to the corresponding complex in degrees $[-1,0]$. We denote the set of such $(d,\al)$ by $\scE(\cat{B}_O)$. 
        \item For a fixed $\tau\in W$ the restriction of $\varsigma$ to:
\begin{equation}
\label{eq:equivalence-cat-pairs}
\left\{\begin{array}{c}
   V\otimes O\to F: \tau^p\textnormal{-semistable}    \\
      \textnormal{of class }(d,\al)\in \scE(\cat{B}_O)
\end{array}\right\}\xlongrightarrow{\varsigma}\left\{\begin{array}{c}
I: \tau\textnormal{-semistable}    \\
      \textnormal{of class }\beta\in \scE(\cat{B})\\
       \textnormal{with determinant }\det(O) \textnormal{ if }d=1  
\end{array}\right\}
\end{equation}
is an equivalence of categories identifying the classes of exact triples contained in them.
        \item The space of stability conditions $W$ further satisfies Assumption \ref{ass:stab} (b) - (g) where the stacks $\fM_{\beta}(\tau)_{\det(O)},\fM_{\cat{B},\det(O)}$ are used instead whenever $d=1$ in \eqref{eq:pairtocomplexiso}. Moreover, if there are no strictly semistables in the algebraic space $M_{\beta}(\tau)_{\det(O)}=\fM_{\beta}(\tau)_{\det(O)}$ for $d=1$, we do not rigidify it. 
        \item For every $\beta\in\scE(\cat{B})$ assume that there is a framing functor $\Fr$ on  $\cat{B}_O$ so that $\fN_{d,\al}(\tau^p)\subset \fN^{\Fr}_{(d,\al)}$ for the corresponding $(d,\al)\in \scE(\cat{B}_O)$. 
\item Fix a quiver $Q$ as in Definition \ref{def:auxiliary-stack} and let $\vec\Fr$ be chosen as in (d) for a given $\beta\in \scE(\cat{B})$. For $\tau\in W$, the category \eqref{eq:tau-semi-phi-fixed} is now replaced by the full subcategory $  \cat{B}^{Q(\vec\Fr)}_{\tau,\phi}$ of objects $I$ with fixed phase $\tau(I)=\phi$ in either side of \eqref{eq:equivalence-cat-pairs}  together with the framing data determined by $Q(\vec\Fr)$. Choose a path $\gamma$ in $W$ and $t,s,t'\in [0,1]$ as in Assumption \ref{ass:stab} (i). Using $\lambda^{t,t'}$ selected for $\tau:=\gamma_t$ and $\gamma_{t'}$ in $W$, set $\lambda:=s\lambda^{t,t'}$ and use the stability \eqref{eq:barphi} on   $\cat{B}^{Q(\vec\Fr)}_{\tau,\phi}$ to construct \eqref{eq:framedstablemodulispace} for $(\beta,\vec{d})\in \scE(\cat{B})\times \bZ^{Q^{f}_0}$ whenever there are no strictly $\tau^{\lambda}_{\vec\delta}$-semistables. These algebraic spaces are required to be proper and to have the $\bG_m$-equivariant resolution property. 
    \end{enumerate}
\end{assumption}
\subsubsection{}
\begin{example}[{\cite[Example 6.11]{Bo25}}]
\label{ex:JS-wall-crossing-setup}
  Fix a $\tau\in W$ from Assumption \ref{ass:stab} for $\cat{A} = \Coh(X)$. Consider the categories from Example \ref{ex:CY4JScategories} where $(1,0) = \cO_X(-D)[1]$ as in §\ref{sec:puresheavesJSframing}.  This time, we only consider pairs $V\otimes \cO_X(-D)\to F$ where $F$ is $\tau$-semistable and $\tau(F)=\phi$ for a fixed value $\phi\in S$. We will also impose either $\dim(F)<4$ or that $X$ is a strict CY fourfold. This smaller category will be denoted by $\fB^{\phi}_{D}$ with its moduli stack $\fN^{\phi}_{D}$. We assume here that $\fM^{\sig}_{\al}$ satisfies Assumption \ref{ass:stab} (h) for this framing.
   
   Note that $H^i\big(F(D)\big)=0$ already holds for all $(V\otimes \cO_X(-D)\to F)\in \fB^{\phi}_{D}$. Thus we only need to choose another divisor $D_+$ such that $H^i\big(\cO_X(D_+ - D)\big) = 0$ for $i>0$. Then we define $\Fr$ from Assumption \ref{ass:pairWC}.d) by 
   \begin{equation}
   \label{eq:framing-JS-pairs}
   \Fr(V\otimes \cO_X(-D)\to F) = V\otimes H^0\big(\cO_X(D_+ - D)\big) \oplus H^0\big(F(D)\big).
   \end{equation}
   Then $\cat{B}^{Q(\Fr)}_{\tau,\phi}$ from Assumption \ref{ass:pairWC} (e) is defined by setting $O=\cO_X(-D)$   and starting from $\fB^{\phi}_{D}$ instead of \eqref{eq:tau-semi-phi-fixed}. Consider the family of stability conditions $\{\theta_t\}_{t\in [-1,1]}$ on $\fB^{\phi}_{D}$ with phases
    $$
  \psi_{t}(d,\al) = \begin{cases}
     t&\text{if}\quad d\neq 0\,,\\
     0&\text{if}\quad d=0\,.
    \end{cases}
    $$
    Note that the $\theta_{t}$ semistable objects in $\fB^{\phi}_{D}$ are given as follows:
\begin{itemize}
    \item for $t>0$ and the class $(1,\al)$, they are precisely the Joyce--Song stable pairs,
    \item for $t=0$, all objects of $\fB^{\phi}_{D}$ are semistable,
    \item for $t<0$, the only semistable objects have class $(d,0)$ or $(0,\al)$.
\end{itemize}
 Due to the arguments in \cite[§12.6, §12.7]{JoyceSong}, there is a CY4 obstruction theory on $\big(\fN^{\phi}_D\big)^{\rig}$ which corresponds to fixing determinants when $\dim(F)<4$. Thus we may treat this example as a special case of Assumption \ref{ass:pairWC}. See \cite[Appendix A]{Bo21} where the necessary conditions from Assumption \ref{ass:stab} are checked.
\end{example}

\section{Semistable Invariants}\label{sec:sst-inv}
\subsection{Review of the Setup}
  
\subsubsection{}
We begin by reviewing the method for defining invariants counting semistable objects. The idea to use framing to define invariants originates from \cite[\S 7.2.2]{mochizuki}. In the present form, the construction was introduced in \cite[\S9.1]{Joyce2021} in homology, which was used to formulate the equivariant CY4 version in \cite[\S5.3]{Bo25} when CY4 obstruction theories exist. The refinement of \cite[\S 9.1]{Joyce2021} to operational K-homology was presented in \cite[§4]{Liu2022} and \cite[\S 3.1]{KLT25}, the latter addressing CY3 theories\footnote{It also included some improvements. For example, \cite{Liu2022} dealt only with
  stability conditions for simplicity, whereas \cite{KLT25} allows
  $\tau$ to be a weak stability condition.}. Using the tools developed in \S \ref{sec:sym-pullback}, we now extend all of the above to produce the invariants in Theorem~\ref{thm:sst-invariants}.

Suppose that we are given $\cat{A}$ with its space of stability condition $W$ and $\scE(\cat{A})$ satisfying Assumption \ref{ass:orientation} and Assumption \ref{ass:stab}. Fix a weak stability $\tau\in W$, and $\alpha \in \scE(\cat{A})$. Using Assumption \ref{ass:stab} (h), choose a framing functor $\Fr$ as in Definition \ref{bg:def:framing-functor} satisfying $\fM_{\al}(\tau)\subset \fM^{\Fr}_{\al}$ and recall the rank function $\rk_\tau$ from
Assumption~\ref{ass:stab} (e). We will construct the semistable invariant
$\sz_\alpha(\tau)$ assuming inductively that the semistable invariants
$\sz_\beta(\tau)$ have already been constructed for $\beta$ such that
$\rk_{\tau}(\beta) < \rk_{\tau}(\alpha)$. The philosophy behind this approach was explained in \cite[\S5.3]{Bo25}.
\subsubsection{}

\begin{definition} \label{def:pair-invariant}
  Consider the quiver 
  \[ \begin{tikzcd}
   \QJS:=\quad \overset{V}{\blacksquare} \ar{r}{\rho} & \overset{\Fr(E)}{\blackbullet}
  \end{tikzcd}. \]
  Given a framing functor $\Fr$ as in Definition \ref{bg:def:framing-functor}, consider the auxiliary exact
  category $\cat{A}^{{\QJS}(\Fr)}$ parameterizing triples $(E, V,
  \rho)$ as labeled above. Such a triple has class $(\beta, \dim V)$
  where $\beta$ is the class of $E$. Let
  $\fM_{\beta,d}^{{\QJS}(\Fr)}$ denote the moduli stack, with
  canonical maps and commutative square
  \begin{equation} \label{eq:pairs-forgetful-rigidification-maps}
    \begin{tikzcd}
      \fM_{\beta,0}^{{\QJS}(\Fr)} & \fM_{\beta,1}^{{\QJS}(\Fr)} \ar[shift left]{r}{\Pi_{\beta,1}^\pl} \ar{d}{\pi_{\fM_\beta^{\Fr}}} & \fM_{\beta,1}^{{\QJS}(\Fr),\pl} \ar{d}{\pi_{\fM_\beta^{\Fr,\pl}}} \ar[shift left]{l}{I_{\beta,1}} \\
      \fM_\beta^{\Fr} \ar{u}{\iota^{\QJS}_\beta} & \fM_\beta^{\Fr} \ar{r}{\Pi_\beta^\pl} & \fM_\beta^{\Fr,\pl}.
    \end{tikzcd}
  \end{equation}
  Here $\iota^{\QJS}_\beta$ are the natural isomorphisms given by $E \mapsto
  (E, 0, 0)$. Both
  $\iota^{\QJS}$ and the forgetful maps $\pi_{\fM^{\Fr}}$ induce morphisms of
  graded monoidal $\sT$-stacks. The $I_{\beta,1}$ are the canonical
  de-rigidification maps from Definition \ref{def:de-rigidification}. 
\end{definition}

\subsubsection{}
\begin{remark}
    Note that $\bL_{\pi_{\fM^{\vec\Fr}_{\al}}}$ from \ref{bg:def:framing-functor} is in general not a vector bundle as its tor-amplitude is supported in $[0,1]$. This is why we should technically always take $\hat\fc_{\rk}\big(\bL_{\pi_{\fM^{\vec\Fr}_{\al}}})$ when acting with Chern classes. However, we will only need to evaluate its cap product on algebraic spaces $M$ where $$\Omega_{\pi_{\fM^{\vec\Fr}_{\al}}}:=\bL_{\pi_{\fM^{\vec\Fr}_{\al}}}|_M$$
    is a vector bundle. Thus, we will regularly abuse notation and write 
    $$
  \cap\, \hat{\fe}\big(\Omega_{\pi_{\fM^{\vec\Fr}_{\al}}}\big):= \cap\,\hat\fc_{\rk}\big(\bL_{\pi_{\fM^{\vec\Fr}_{\al}}})\,.
    $$
    This notation will also be used when restricting $\bL_{\pi_{\fM^{\vec\Fr}_{\al}}}$  itself to $M$.
\end{remark}

\subsubsection{}
Throughout the inductive construction of $\sz_\alpha(\tau)$, we will
only need to determine the (semi)stability of objects $(E, V, \rho)$
where $E$ is a $\tau$-semistable object whose class $\beta$  belongs to the finite set \cite[Lem. 9.1]{Joyce2021}
\[ R_\alpha \coloneqq \{\alpha\} \cup \{\beta \in \scE(\cat{A}) : \alpha-\beta\in \bar{C}(\cat{A}), \, \tau(\beta) = \tau(\alpha-\beta), \, \fM_{\beta}(\tau), \fM_{\alpha-\beta}(\tau)\neq \emptyset\}. \]
Recall the definition of the stability condition in \eqref{eq:barphi} applied to the quiver $Q^{\JS}$ and $\phi=\phi(\al)$ in the case when $\lambda = 0$ and $\mu > 0$. We will denote the resulting stability condition by $\tau^{\JS}$.


\begin{lemma}[{\cite[Example 5.7]{Bo25}, \cite[Definition 3.1.2, Lemma 3.1.3]{KLT25},\cite[Example 5.6]{Joyce2021}}] \label{lem:pairs-stability}
 For any $\beta
  \in R_\alpha$ and $\dim V = 1$, the following are equivalent:
  \begin{itemize}
  \item $(E, V, \rho)$ is $\tau^{\JS}$-stable;
  \item $(E, V, \rho)$ is $\tau^{\JS}$-semistable;
  \item $\rho \neq 0$ and there does not exist $0 \neq E' \subsetneq
    E$ with $\tau(E') = \tau(E/E')$, $E'$ and $E/E'$ being $\tau$-semistable, and $\rho(V) \subseteq \Fr(E')
    \subsetneq \Fr(E)$.\footnote{Note that due to the different set-up used in Definition \ref{def:framedstability}, this last point additionally required $E'$ and $E/E'$ to be semistable when compared to \cite[Lemma 3.1.3]{KLT25}. However, in all examples considered there and those we have in mind, the condition used in \cite{KLT25} already implies the semistability of these factors.}
  \end{itemize}
  In particular, if $E$ is $\tau$-stable then $(E, V, \rho)$ is
  $\tau^{\JS}$-(semi)stable if and only if $\rho \neq 0$.
\end{lemma}
\subsubsection{}

Using Lemma~\ref{lem:pairs-stability},  Assumption~\ref{ass:stab} (i) gives us a quasi-compact, separated algebraic spaces 
$$
M^{\Fr}_{\al,1} (\tau^{\JS}):=M^{Q^{\JS}(\Fr)}_{\al,1} (\tau^{\JS})
$$
with the $\sT$-equivariant resolution property. We can therefore apply the construction in Theorem \ref{sp:thm:symm-pullback-localization} (a) to the smooth projection $\pi_{\fM_\alpha^{\Fr,\pl}}$ from \eqref{eq:pairs-forgetful-rigidification-maps} and the CY4 obstruction theory $\bE^{\pl}$ from \eqref{eq:EE-rigidified} to produce 
\begin{equation}
\label{eq:FRJS-class}
\big[\widehat{\mathcal{O}}_{M^{\Fr}_{\al,1} (\tau^{\JS})}^{\vir}\big]\in K^{\sT}_0\big(M^{\Fr}_{\al,1} (\tau^{\JS}),\bZ[2^{-1}]\big)\,.
\end{equation}
Recall from §\ref{sec:equivariant-hom-notation} that we use $H^\sT(-)_\loc$ to denote either of the equivariant theories $\bfK_\circ^\sT(-)$ of Definition~\ref{bg:def:k-homology-concrete-theory}, $H_*^\sT(-)$ of Section~\ref{sec:Khan-bivariant}, and $K_0^\sT(-)$ of Section~\ref{sec:homology-to-khomology}. Because the $\sT$-fixed point locus of $M^{\Fr}_{\al,1} (\tau^{\JS})$ is additionally assumed to be proper, Definition~\ref{bg:def:univ-enum-inv} or Definition \ref{def:bivariant-hom-invariants} give us a well-defined invariant $$\sZ_{\alpha,1}^{\Fr}(\tau^{\JS})\in
H^\sT\Big(M_{\alpha,1}^{{\QJS}(\Fr)}\Big)_{\loc}.$$ It is used to define\footnote{Note that Joyce's $\bar{\Upsilon}^k_{\al}(\tau)$ from \cite[Definition 9.2]{Joyce2021} and the $I_*\tilde{\sZ}^{\Fr}_{\alpha,1}(\tau^Q)$ in \cite[\S3.1.4]{KLT25} differ from our $\breve{\sZ}_{\alpha}^{\Fr}(\tau)$ because we divided out $\fr(\al)$. The actual semistable invariants $\sz_\alpha^{\Fr}(\tau)$ are unchanged, since \eqref{eq:sstable-def} contains the extra factors. The convention used here makes \eqref{eq:OmegaWC} match with the combinatorial procedure explained in \cite[\S8.1]{Bo25}.}
\begin{equation} \label{eq:pairs-invariant}
  \begin{aligned}
\fr(\al)\breve{\sZ}_{\alpha}^{\Fr}(\tau)
    &\coloneqq (\pi_{\fM_\alpha^{\Fr}})_*\left((I_{\alpha,1})_*\sZ_{\alpha,1}^{\Fr}(\tau^{\JS}) \cap\hat{\fe}(\Omega_{\pi_{\fM_\alpha^{\Fr}}}^\vee)\right)  \in H^\sT(\fM_\alpha)_{\loc}.
  \end{aligned}
\end{equation}
We denote its projection to the quotient $L(\fM_\alpha)_{\loc}$ from §\ref{sec:LA-quotient} in the same way.
\subsubsection{}
\begin{remark}
\label{rem:expecteddefinitionsame}
   Consider the situation in Example \ref{ex:JSobstheoryframing} which means that the obstruction theory \eqref{eq:FFJSobstructiontheory} is given. In 
\cite[Lemma 5.6]{Bo25}, it is shown that the projection $\pi_{\fM_\beta^{\Fr,\pl}}$ can be endowed with a PVP diagram \eqref{eq:pvp-diagram} where $\bE_{\pi_{\fM_\beta^{\Fr,\pl}}}\cong \bL_{\pi_{\fM_\beta^{\Fr,\pl}}}$. In particular, we are in the situation of Theorem \ref{sp:thm:symm-pullback-localization} (b) and the class $\sZ_{\alpha,1}^{\Fr}(\tau^{\QJS})$ is identified with the usual one as constructed in \cite[Definition 5.13.2)]{Bo25}. Since the defining relation of $\sz_\alpha^{\Fr}(\tau)$ in Theorem \ref{thm:sst-invariants-construction} below is the same as the one in \cite[(5.34)]{Bo25}, the more general invariants are also the same. Thus in this case, the definition is the expected one as proposed in \cite{Joyce2021}.
\end{remark}

\subsubsection{}

\begin{theorem}[Semistable invariants] \label{thm:sst-invariants-construction}
  Let $\tau\in W$ and $\alpha \in \scE(\cat A)$.  Choosing a
  framing functor $\Fr$ such that $\fM_{\alpha}(\tau)
  \subset \fM_{\alpha}^{\Fr,\pl}$, the equations
  \begin{equation}\label{eq:sstable-def}
   \fr(\al) \breve{\sZ}_{\alpha}^{\Fr}(\tau) = \sum_{\substack{  \vec{\alpha}\,\vdash_{\cat{A}} \alpha\,,\\ \forall i: \,\tau(\alpha_i) = \tau(\alpha)\\ \;\;\fM_{\alpha_i}(\tau) \neq \emptyset}} \frac{\fr(\alpha_1)}{n!} \left[\sz^{\Fr}_{\alpha_n}(\tau), \left[\cdots,\left[\sz^{\Fr}_{\alpha_2}(\tau), \sz^{\Fr}_{\alpha_1}(\tau)\right]\cdots\right]\right],
  \end{equation}
  in the Lie algebra
  $L(\fM_{\cat{A}})_{\loc}$ uniquely define elements
  \[ \sz_\alpha^{\Fr}(\tau) \in L(\fM_{\cat{A}})_{\loc} \,.\]
 The classes $\sz_\alpha^{\Fr}(\tau)$ are independent of $\Fr$ and satisfy all the properties listed
  in Theorem~\ref{thm:sst-invariants}.
\end{theorem}

The goal of this section is to prove
this theorem.

\subsection{Independence via Quantum Leftschetz}
\label{sec:independencequantumLefschetz}
\subsubsection{}
Here we recall the argument presented in \cite[\S1.4 and \S6.5]{Bo25} which deals with torsion-free sheaves on a projective Calabi-Yau fourfolds $X$. In this version, we for now assume that $X$ is strict, avoiding fixed determinant obstruction theories for higher rank sheaves. We set $\cat{A}= \Coh(X)$ with $X$ satisfying \cite[($\ast$)]{JU}. Choose the appropriate data as in §\ref{sec:puresheavesJSframing} such that $\scE(\cat{A})$ contains all $\alpha$ of positive rank. Recall, that this implies that the framing functor is given by 
$$
\Fr_v(E) = H^0\big(E(D_v)\big)
$$
for an ample divisor $D_v$, where the subscript $v$ labels different choices of such divisors.

For another ample divisor $H$ on $X$, we choose $\tau$ to be the Gieseker or slope stability with respect to it. To make the $D_v$ dependence explicit, we write
$$N^{\JS}_{D_v,\alpha}:=M^{\Fr_v}_{\al,1} (\tau^{\JS})\,.$$
The moduli space admits a virtual fundamental class $\big[\widehat{\cO}^{\vir}_{N^{\JS}_{D_v,\alpha}}\big]$ induced by the obstruction theory $\bF^{\JS}_{D_v,\alpha}$ which is obtained by rigidifying the obstruction theory \eqref{eq:FFJSobstructiontheory} as in \cite[(127)]{BKP} and restricting it to $N^{\JS}_{D_v,\alpha}$. For simplicity, we will write 
$$\sZ^{\JS}_{D_v,\alpha}:=\sZ_{\alpha,1}^{\Fr_v}(\tau^{\JS})\qquad \textnormal{and}\qquad \sz^{D_v}_{\alpha}(\tau):=\sz^{\Fr_v}_{\alpha}(\tau)$$ 
for the resulting invariants appearing in \eqref{eq:pairs-invariant} and \eqref{eq:sstable-def}.

\subsubsection{}
We now want to compare $N^{\JS}_{D_v,\alpha}$ and their classes $\big[\hat{\cO}^{\vir}_{N^{\JS}_{D_v,\alpha}}\big]$ for two different ample divisors $D_1$ and $D_2$ assuming that $\fM_{\al}(\tau)$ lies in $\fM^{\Fr_v}_{\al}$ for $v=1,2$. As in \cite[\S6.5]{Bo25}, we assume without loss of generality that $D :=D_2 - D_1$ is very ample so we can fix a section $s\in H^0\big(\cO_X(D)\big)$ such that $s^{-1}(0)= D$ is smooth. 

Let $\cF$ be the universal sheaf of class $\al$ on $X\times N^{\JS}_{D_2,\alpha}$ and $p:X\times N^{\JS}_{D_2,\al}\to N^{\JS}_{D_2,\al}$ the projection. In \cite[Lemma 6.13]{Bo25}, the first-named author proved that the map
 $$\big(\cO_X(-D_1)\xrightarrow{f} F\big)\mapsto \big(\cO_X(-D_2)\xrightarrow{f\circ s} F\big)$$
 induces a closed embedding
\begin{equation}
\label{eq:iotaNDi}
N^{\JS}_{D_1, \al}\xlongrightarrow{\iota} N^{\JS}_{D_2, \al}\,.
\end{equation}
Moreover, it is shown there that $N^{\JS}_{D_1, \al}$ is the vanishing locus of a natural section $\fv$ of the vector bundle
$$
\bV = Rp_*\Big(\cF(D_2)|_{D}\Big)\,.
$$
In \cite[Proposition 6.14]{Bo25}, a PVP diagram along $\iota$ is constructed with the upper two rows given by
$$
  \begin{tikzcd}[column sep=small]
    \wt{\bF}_{\iota}^\vee[2] \ar{d} \ar{r} & \bF^{\JS}_{D_1,\alpha}\ar{d} \ar{r} &  \bV^\vee[1] \ar[equals]{d} \ar{r}{+1} & {}  \\
   \iota^*\bF^{\JS}_{D_2,\alpha} \ar{r} & \wt{\bF}_{\iota} \ar{r} & \bV^\vee[1]\ar{r}{+1} & {}   
  \end{tikzcd}\,.
  $$
  Thus, it was shown there that 
  \begin{equation}
\label{eq:quantumLefschetzforJS}
\iota_*\Big(\big[\hat{\cO}^{\vir}_{N^{\JS}_{D_1, \al}}\big]\Big) = \big[\wh{\cO}^\vir_{N^{\JS}_{D_2, \al}}\big]\otimes \hat{\fe}(\bV) \,.
  \end{equation}
\subsubsection{}
Here we complete the argument which was proposed in \cite[\S6.5]{Bo25}. It is independent of the one in §\ref{sec:generalindependence} as it only uses the lift of the defining formula \eqref{eq:sstable-def} as explained in \cite[Lemma 5.12 (i)]{Bo25} and the proof of \cite[Theorem 6.12]{Bo25}. In Example \ref{ex:JS-wall-crossing-setup}, it is shown that this lift, called \textit{Joyce--Song pair wall-crossing}, is a particularly simple application of the general $\Fr_v$ dependent wall-crossing formula from \eqref{eq:Fr-dependent-wall-crossing}, where $\Fr_v$ is determined by $\cO_X(-D_v)$ on sheaves and by \eqref{eq:framing-JS-pairs} on pairs. It takes the form
\begin{equation}
\label{eq:JSwall-crossingformula}
\sZ^{\JS}_{D_v,\alpha}  = \sum_{\begin{subarray}{c}
        \vec{\alpha}\,\vdash_{\cat{A}} \alpha\,, \\
          \tau(\alpha_i) =\tau(\alpha) \end{subarray}}\frac{1}{n!}\Big[\sz^{D_v}_{\alpha_n}(\tau), \cdots \Big[\sz^{D_v}_{\alpha_1}(\tau),  e^{(1,0)}_{D_v}\Big]\cdots\Big]
\end{equation}
in the Lie algebra of $L\big(\fM^{\QJS(\Fr_v)}\big)$ from §\ref{sec:flag-VA+LA} where $e^{(0,1)}_{D_v}$ represents the point class of $\cO_X(-D_v)[1]$.  By \cite[Remark 1.3]{Bo25}, the Joyce--Song wall-crossing formula makes sense before it is proved that the invariants are independent of framing, so that this proof is not circular.

We will not repeat the detailed argument from the proof of \cite[Theorem 6.12]{Bo25} choosing to sketch it instead. One can extend the map $\iota$ from \eqref{eq:iotaNDi} to a map of stacks  $\fM_{\al,1}^{{\QJS}(\Fr_1),\pl}\to\fM_{\al,1}^{{\QJS}(\Fr_2),\pl}$. One can also define $\bV$ on the entire latter stack for. Then the following holds due to \eqref{eq:quantumLefschetzforJS}:
$$
(\iota)_*\Big(\textnormal{RHS of }\eqref{eq:JSwall-crossingformula}\textnormal{ for }D_1\Big) =\Big(\textnormal{RHS of }\eqref{eq:JSwall-crossingformula}\textnormal{ for }D_2\Big)\cap \hat{\fe}\big(\bV\big)\,.
$$
Due to \cite[§2.5]{GJT}, the right-hand side can be, up to minor technicalities explained in \cite[p. 84]{Bo25}, identified with the left-hand side except that all $\sz^{D_1}_{\alpha_i}$ are now replaced by $\sz^{D_2}_{\alpha_i}$. Thus by injectivity of the brackets $\Big[-,  e^{(1,0)}_{D_v}\Big]$ (see for example \cite[Lemma 5.12.ii)]{Bo25}), this implies that 
$$
\sz^{D_1}_{\alpha}(\tau) = \sz^{D_2}_{\alpha}(\tau)
$$
by induction.

\subsection{Independence for General Invariants}
\label{sec:generalindependence}
\subsubsection{}
\begin{definition}\label{sst:def:wedge-quiver}
  Let ${Q} \wedge {Q}$ denote the quiver
 \begin{equation}
 \label{eq:Vshapedquiver}
\includegraphics{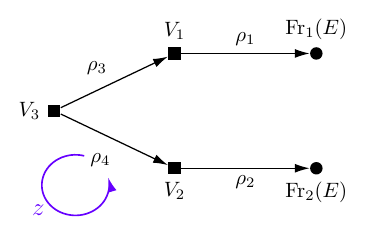}
\end{equation}

 Let $\Fr_1, \Fr_2$ be framing functors from Definition \ref{bg:def:framing-functor} such that
  \[ \fM_\alpha(\tau) \subset \fM_\alpha^{\Fr_1,\pl} \cap \fM_\alpha^{\Fr_2,\pl}\,. \]
 We onsider the auxiliary exact category $\cat{A}^{({Q} \wedge
    {Q})(\vec\Fr)}$ from Definition~\ref{def:auxiliary-stack} parameterizing triples $(E, \vec
  V, \vec \rho)$ as labeled above. For clarity, we will write
  ${Q}(\Fr_1) \wedge {Q}(\Fr_2)$ instead of $({Q} \wedge {Q})(\vec\Fr)$, and
  $\cFr_i(\cE_j)$ for the universal bundle of $\Fr_i(E_j)$. We also use the notation  $$\scL_{ij}
\coloneqq \cV_i^\vee \otimes \cV_j\,.$$
\end{definition}

\subsubsection{}

Let $\fM_{\alpha,\vec d}^{{Q}(\Fr_1) \wedge {Q}(\Fr_2)}$ be the
moduli stack of objects in $\cat{A}^{({Q} \wedge {Q})(\vec\Fr)}$.
It is equipped with a natural action of $\bG_m$ which
scales $\rho_4$ with its weight denoted by $z$ as represented in \eqref{eq:Vshapedquiver}. This clearly commutes with
the $\sT$-action so we have an induced action of $\bT:=\bG_m\times \sT$. Moreover, the $\bG_m$-action descends to an action
on the rigidified stack, and all forgetful maps are $\bG_m$-equivariant
for the trivial $\bG_m$-action on their targets.

For a fixed $\tau\in W$ we construct the stability condition from Definition~\ref{def:framedstability} for $\phi=\phi(\al)$, $\lambda = 0$, and $(\delta_1,\delta_2,\delta_3) = (\epsilon,\epsilon,1)$ where $0<\epsilon < 1/\rk_{\tau}(\al)$. We will denote it by $\tau^{\wedge}$. As shown in \cite[Definition~9.4]{Joyce2021}, there are no strictly $\tau^{\wedge}$-semistable objects of class
$
(\al,1,1,1)\,,
$
so we denote the resulting moduli spaces by 
  \[  \bbLambda_\alpha \coloneqq \fM^{{Q}(\Fr_1)\wedge {Q}(\Fr_2)}_{\alpha,(1,1,1)}(\tau^\wedge)\,. \]
  The $\bT$-action restricts to $ \bbLambda_{\al}$.
\subsubsection{}
 A similar statement used for proving wall-crossing is noted down in Proposition \ref{prop:master-space-fixed-loci} below. There we also represent the fixed-point loci and their $\bG_m$ weights diagrammatically. 
The virtual normal bundles $\bN_{\iota}$ from Theorem~\ref{sp:thm:symm-pullback-localization} (c) can be deduced in the same way as in \cite[Proposition~8.6]{Bo25}.
\begin{proposition}[{cf. \cite[Props. 9.5, 9.6]{Joyce2021},\cite[Proposition~7.17]{mochizuki}}] \label{prop:pairs-master-space}
 The $\bG_m$-fixed point locus of $\bbLambda_{\al}$ is a disjoint union of   the following pieces.
  \begin{enumerate}[label = (\roman*)]
  \item The embedding $\iota_{\rho_4=0}\colon \bbLambda_{\rho_4=0} \coloneqq
    \{\rho_4=0\} \hookrightarrow \bbLambda_\alpha$ has the virtual normal bundle  $$\bN_{\iota_{\rho_4=0}} = z\scL_{32}+z^{-1}\scL_{23}\,.$$
    By $\tau^{\wedge}$-stability, $\rho_1,
    \rho_2, \rho_3 \neq 0$. The forgetful map
    \[ \pi_{\rho_4=0}\colon \bbLambda_{\rho_4=0} \to M^{\Fr_1}_{\al,1} (\tau^{\JS}), \]
    which remembers only $(E,V_1,\rho_1)$, is a
    $\bP^{\fr_2(\alpha)-1}$-bundle.

  \item The embedding $\iota_{\rho_3=0}\colon \bbLambda_{\rho_3=0} \coloneqq
    \{\rho_3=0\} \hookrightarrow \bbLambda_\alpha$ has the virtual normal bundle  $$\bN_{\iota_{\rho_3=0}} = z^{-1} \scL_{31} \oplus z \scL_{13} \,.$$
    By $\tau^{\wedge}$-stability, $\rho_1,
    \rho_2, \rho_4 \neq 0$. The forgetful map
    \[ \pi_{\rho_3=0}\colon \bbLambda_{\rho_3=0} \to M^{\Fr_2}_{\al,1} (\tau^{\JS}), \]
    which remembers only $(E,V_2,\rho_2)$, is a
    $\bP^{\fr_1(\alpha)-1}$-bundle.

  \item \label{sst:it:complicated-locus} For each splitting $\alpha =
    \alpha_1 + \alpha_2$ with $\al_i\in R_{\al}$ for $i=1,2$, let
    \[ \iota_{\alpha_1,\alpha_2}\colon \bbLambda_{\alpha_1,\alpha_2} \coloneqq \{E = E_1 \oplus E_2, \; \rho_i\colon V_i \to \Fr_i(E_i) \subset \Fr_i(E) \text{ for } i = 1, 2\} \hookrightarrow \bbLambda_\alpha, \]
    where $\bG_m$ scales $V_3$, $V_1$, and $E_1$ with weight $z$.
    The positive weight part (see §\ref{sp:sec:loc-setup}) of the virtual normal bundle is
    \begin{equation} \label{eq:pairs-master-space-interaction-Nvir}
        \bN^{>}_{\iota_{\alpha_1,\alpha_2}} = z\left(  \cV_2^\vee \otimes \cFr_2(\cE_1) +\cV_1 \otimes \cFr_1(\cE_2)^\vee + \big(\pi_{\fM_{\al_1}^{\Fr_1}}\times \pi_{\fM_{\al_2}^{\Fr_2}}\big)^*\Theta_{\al_1\al_2} \right)
    \end{equation}
    where $\Theta$ was defined in \eqref{eq:ThetaofA}. By
    $\tau^{\wedge}$-stability, $\rho_1, \ldots, \rho_4 \neq 0$.
    The forgetful map
    \begin{equation} \label{eq:pairs-master-space-complicated-locus}
     \pi_{\al_1,\al_2}: \bbLambda_{\alpha_1,\alpha_2} \xrightarrow{\sim} M^{\Fr_1}_{\al_1,1} (\tau^{\JS}) \times M^{\Fr_2}_{\al_2,1} (\tau^{\JS}),
    \end{equation}
    whose $i$-th factor is given by remembering only $(E_i,V_i,\rho_i)$, is an isomorphism.
  \end{enumerate}
\end{proposition}
Our convention is to write $\bT$-equivariant sheaves on
$\bG_m$-fixed loci as a product of an explicit $\bG_m$-weight and a
$\sT$-equivariant sheaf. Obvious pullbacks are omitted.
\subsubsection{}
By Assumption~\ref{ass:stab} (i), $\bbLambda_\alpha$ is a quasi-compact, separated algebraic space with the $\bT$-equivariant resolution property.  We can therefore apply the construction in Theorem \ref{sp:thm:symm-pullback-localization} (a) to the smooth projection $\pi_{\fM_\alpha^{\vec\Fr,\pl}}$ from §\ref{eq:pairs-forgetful-rigidification-maps} and the CY4 obstruction theory $\bE^{\rig}$ to produce 
$$
\big[\widehat{\mathcal{O}}_{\bbLambda_{\al}}^{\vir}\big]\in G^{\bT}_0\big(\bbLambda_{\al},\bZ[2^{-1}]\big)\,.
$$
Let $\sZ_{\bbLambda_{\al}}\in H^{\bT}(\bbLambda_{\al})_{\loc}$ be the universal invariant as in §\ref{bg:sec:universal-invariants}.
For the insertion
\[ G \coloneqq \Omega_{\pi_{\fM_\alpha^{\vec\Fr}}}^\vee - \scL_{31} \otimes \scL_{32} \in K_{\bT}^\circ(\bbLambda_\alpha^\pl)\,, \]
we define the class 
\begin{equation} \label{eq:pairs-master-lhs}
(\pi_{\fM^{\Fr}_{\al}})_*\left( I_*\sZ_{\bbLambda_{\al}}\cap \hat{\fc}_{\rk}(G) \otimes (\scL_{31}^\vee \otimes \scL_{32})^{\frac{1}{2}}\right) \in H^\sT(\fM_\alpha)_{\loc}\,.
\end{equation}
Note that $\bG_m$ acts
trivially on $\fM_\alpha$, so we may apply the residue map $\rho_z$ from §\ref{sec:LA-formula+residue} to the above invariant.\footnote{In the case of the equivariant homology theories from §\ref{sec:equivariant-homology-bivariant}, this is possible due to Example \ref{ex:trivialThom}.} 
It follows from the properness in Assumption \ref{ass:stab} (i) that this total residue will vanish (see \cite[§8.4]{Bo25}, \cite[Lemma 2.6.4]{KLT25}). Ultimately, we will need to view \eqref{eq:pairs-master-lhs} as an element of $L(\fM_{\cat{A}})_{\loc}$ for which we use the same notation.
\subsubsection{}

On the fixed point locus $\bbLambda_{\rho_4=0}$, the line bundle $\scL_{31}$ carries
the non-zero section $\rho_3$, of trivial $\bG_m$-weight, and is
therefore $\bG_m$-equivariantly trivial. Similarly, $\scL_{32}$ carries
the section $\rho_4$ of $\bG_m$-weight $z$, so $\scL_{32}$ becomes
$z\cL$ for a $\sT$-equivariant line bundle $\cL$ of trivial
$\bG_m$-weight. Hence the first term in $G$ splits as
\[ \Omega_{\pi_{\fM_\alpha^{\vec\Fr}}}^\vee\Big|_{\bbLambda_{\rho_4=0}} = z\cL + \Omega_{\pi_{\rho_4=0}}^\vee + \pi_{\rho_4=0}^* \Omega_{\pi_{\fM_\alpha^{\Fr_1}}}^\vee, \]
and the second term in $G$ cancels the $z\cL$ term. Since
$\Omega_{\pi_{\rho_4=0}}$ is a vector bundle,
\[ \hat{\fc}_{\rk}(G)\Big|_{\bbLambda_{\rho_4=0}} = \hat{\fe}(\Omega_{\pi_{\rho_4=0}}) \otimes \pi_{\rho_4=0}^*\hat{\fe}(\Omega_{\pi_{\fM_\alpha^{\Fr_1}}}^\vee)\,. \]
 Thus, by
Theorem~\ref{sp:thm:symm-pullback-localization} and
Proposition~\ref{prop:pairs-master-space}, the contribution from
$\bbLambda_{\rho_4=0}$ to \eqref{eq:pairs-master-lhs} is
\begin{equation} \label{eq:pairs-master-term-1}
  (\pi_{\fM^{\Fr}_{\al}}\circ I \circ \iota_{\rho_4=0})_*\left(\frac{\wt{\sZ}_{\bbLambda_{\rho_4=0}}}{\hat{\fe}(z\cL)}\cap (z\cL)^{\frac{1}{2}} \otimes \hat{\fe}(\Omega_{\pi_{\rho_4=0}}) \otimes \pi_{\rho_4=0}^*\hat{\fe}(\Omega_{\pi_{\fM_\alpha^{\Fr_1}}}^\vee) \right)\,,
\end{equation}
where $\wt{\sZ}_{\bbLambda_{\rho_4=0}}$ is the universal invariant defined using $\big[\tilde{\cO}^{\vir}_{\bbLambda_{\al}}\big]$. Applying $\rho_z$,
the only contribution comes from
\[ \rho_z \frac{(z\cL)^{\frac{1}{2}}}{\hat{\fe}(z\cL)} = 1 \]
because all other terms have trivial $\bG_m$-weight. Thus
\eqref{eq:pairs-master-term-1} becomes
\[  (\pi_{\fM^{\Fr}_{\al}}\circ I \circ \iota_{\rho_4=0})_*\left(\wt{\sZ}_{\bbLambda_{\rho_4=0}}\cap \hat{\fe}(\Omega_{\pi_{\rho_4=0}}) \otimes \pi_{\rho_4=0}^*\hat{\fe}(\Omega_{\pi_{\fM_\alpha^{\Fr_1}}}^\vee) \right).\]
Now, the induced virtual cycle $\big[\tilde{\cO}^\vir_{\bbLambda_{\rho_4=0}}\big]$
 must be compared to the virtual
cycle $\big[\hat{\cO}^\vir_{\bbLambda_{\rho_4=0}}\big]$ on the same space which
arises by CY4 pullback along
$\pi_{\fM_\alpha^{\vec\Fr, {\pl}}} \circ \pi_{\rho_4=0}$. We apply
Proposition~\ref{prop:comparison} to $M = \bbLambda_\alpha$, $Z =
\bbLambda_{\rho_4=0}$, and $\fN = \fM = \fM_\alpha^{\pl}$ with the obvious
maps between them. The hypotheses are obviously satisfied, and clearly
$\Omega_{\pi_{\fM_\alpha^{\vec\Fr}}}^{>}|_{\bbLambda_{\rho_4=0}} = 0$. Thus
we conclude 
\[ \big[\tilde{\cO}^\vir_{\bbLambda_{\rho_4=0}}\big] = \big[\hat{\cO}^\vir_{\bbLambda_{\rho_4=0}}\big]\,. \]\

Finally, by Corollary~\ref{cor:flag-pushforward}
applied to the projective bundle $\pi_{\rho_4=0}$, along with the
projection formula, the result is
\[\fr_2(\alpha) (\pi_{\fM^{\Fr}_{\al}})_*I_*\left(\sZ_{\alpha,1}^{\Fr_1}(\tau^{\JS}) \cap \hat{\fe}\big(\Omega_{\pi_{\fM_\alpha^{\Fr_1}}}^\vee\big)\right), \]
which, by the defining equation \eqref{eq:pairs-invariant}, is exactly
$\fr_2(\alpha)\fr_1(\al) \breve \sZ_\alpha^{\Fr_1}(\tau)$.

\subsubsection{}

On the fixed point locus $\bbLambda_{\rho_3=0}$, the calculation is analogous.
Namely, $\scL_{31}$ restricts to become $z^{-1}\cL$ for some
$\sT$-equivariant line bundle $\cL$, while $\scL_{32}$ restricts to
the $\bG_m$-equivariantly trivial line bundle, and the contribution to
the residue comes from
\[ \rho_z^K \frac{(z\cL^\vee)^{\frac{1}{2}}}{\hat{\fe}(z\cL^\vee)} = 1. \]
In contrast, when applying Proposition~\ref{prop:comparison} to
compare virtual cycles, we obtain 
\[ \big[\tilde{\cO}^\vir_{\bbLambda_{\rho_3=0}}\big] = -\big[\hat{\cO}^\vir_{\bbLambda_{\rho_3=0}}\big] \]
because now $\Omega_{\pi_{\fM_\alpha^{\vec\Fr}}}^{>}|_{\bbLambda_{\rho_3=0}} = z\scL^\vee$ has rank
$1$. Hence the contribution from $\bbLambda_{\rho_3=0}$ to
\eqref{eq:pairs-master-lhs} is $-\fr_1(\alpha)\fr_2(\alpha) \breve
\sZ_\alpha^{\Fr_2}(\tau)$.

\subsubsection{}
\label{sec:pairs-master-3}

Finally, consider $\bbLambda_{\alpha_1,\alpha_2}$. We implicitly use the
isomorphism $\pi_{\alpha_1,\alpha_2}$ where appropriate. The first
term in $G$ restricts to
\begin{equation} \label{eq:pairs-master-3-relative-Tvir}
  \begin{aligned}
  \Omega_{\pi_{\fM_\alpha^{\vec\Fr}}}^\vee\Big|_{\bbLambda_{\al_1,\al_2}} = &  \Big(\scL_{31} + \scL_{32}
    + \cV_1^\vee \otimes \cFr_1(\cE_1) + z^{-1} \cV_1^\vee \otimes \cFr_1(\cE_2) \\
    &+ z\cV_2^\vee \otimes \cFr_2(\cE_1) + \cV_2^\vee \otimes \cFr_2(\cE_2)\Big) - \sum_{i=1}^3 \scL_{ii}\,.
  \end{aligned}
\end{equation}
Observing that $\scL_{31},\scL_{32}$ are isomorphic to $\cO$ and re-distributing $\scL_{ii}$ into $\Omega^\vee_{\pi_{\fM_{\alpha_i}^{\Fr_i}}}$ for $i=1,2$, we get
\[ \hat{\fc}_{\rk}(G)\Big|_{\bbLambda_{\alpha_1,\alpha_2}} = \hat{\fe}\left(z^{-1} \cV_1^\vee \otimes \cFr_1(\cE_2) \oplus z\cV_2^\vee \otimes \cFr_2(\cE_1)\right) \otimes \left(\hat{\fe}(\Omega_{\pi_{\fM_{\alpha_1}^{\Fr_1}}}^\vee) \boxtimes \hat{\fe}(\Omega_{\pi_{\fM_{\alpha_2}^{\Fr_2}}}^\vee)\right). \]
Comparing the first factor of the right hand side with $\hat{\fe}\big(\bN^{>}_{\iota_{\al_1,\al_2}}\big)$ from 
\eqref{eq:pairs-master-space-interaction-Nvir}, we see that the contribution
from $\bbLambda_{\alpha_1,\alpha_2}$ to \eqref{eq:pairs-master-lhs} is $\rho_z$ of
\begin{equation} \label{eq:pairs-master-term-3}
 (\pi_{\fM_\alpha^{\Fr}}\circ I\circ \iota_{\alpha_1,\alpha_2})_*\left( \frac{(-1)^{\fr_1(\alpha_2)} \tilde{\sZ}_{\bbLambda_{\alpha_1,\alpha_2}}}{\hat{\fe}_{\bT}(z\Theta_{\alpha_1\alpha_2})} \otimes \left(\hat{\fe}(\Omega_{\pi_{\fM_{\alpha_1}^{\Fr_1}}}^\vee) \boxtimes \hat{\fe}(\Omega_{\pi_{\fM_{\alpha_2}^{\Fr_2}}}^\vee)\right) \right)\,,
\end{equation}
where $\tilde{\sZ}_{\bbLambda_{\alpha_1,\alpha_2}}$ is the universal invariant defined using $\tilde{\cO}^\vir_{\bbLambda_{\alpha_1,\alpha_2}}$.
\subsubsection{}
We now use the isomorphism \eqref{eq:pairs-master-space-complicated-locus} and compare 
$\big[\wt{\cO}^\vir_{\bbLambda_{\al_1,\al_2}}\big]$ to 
$$
\Big[\wh{\cO}^\vir_{M^{\Fr_1}_{\al_1,1} (\tau^{\JS})}\big]\boxtimes\Big[\wh{\cO}^\vir_{M^{\Fr_2}_{\al_2,1} (\tau^{\JS})}\Big]\,.
$$
Both factors in this product are the classes from \eqref{eq:FRJS-class}. Even though, we originally applied Theorem \ref{sp:thm:symm-pullback-localization} to $\bE^{\rig}$ to construct all of these classes, we may instead work with $\bE$ on $\fM_{\cat{A}}$ due to Lemma \ref{lem:Mrig-and-M-no-difference}.  The classes $\Big[\wh{\cO}^\vir_{M^{\Fr_i}_{\al_i,1} (\tau^{\JS})}\big]$ can therefore be constructed as CY4 pullbacks along $\pi_{\fM^{\Fr}_{\al_i}}\circ I_{\al_i,1}$, for example.

 We can now apply Proposition~\ref{prop:comparison} to the commutative $\bT$-equivariant diagram
\[ \begin{tikzcd}[column sep=huge, row sep= large]
M^{\Fr_1}_{\al_1,1} (\tau^{\JS}) \times M^{\Fr_2}_{\al_2,1} (\tau^{\JS})\ar[hookrightarrow]{r}{\iota_{\alpha_1,\alpha_2}} \ar{d}[swap]{(\pi^{\Fr_1}_{\fM_{\alpha_1}} \times \pi^{\Fr_2}_{\fM_{\alpha_2}}) \circ (I_{\alpha_1,1} \times I_{\alpha_2,1})} & \bbLambda_\alpha \ar{d}{\pi_{\fM^{\vec\Fr}_\alpha} \circ I_{\alpha,1}} \\
    \fM_{\alpha_1} \times \fM_{\alpha_2} \ar{r}{\Phi_{\alpha_1,\alpha_2}} & \fM_\alpha
  \end{tikzcd} \]
The condition that \eqref{sp:comp-ass:bE-isom} is an isomorphism can be checked by looking at \eqref{eq:pairs-master-3-relative-Tvir} combined with an argument producing \cite[(8.29)]{Bo25}. For \eqref{sp:comp-ass:bE-isom}, we observe that \eqref{eq:EE=Delta-Theta} together with bilinearity of $\Theta$ implies
\begin{align*}
I^*\big(\pi_{\fM^{\vec\Fr}_{\al}}\big)^* \bE_{\al}|_{\bbLambda_{\al_1,\al_2}}  = &I^*\big(\pi_{\fM^{\Fr_1}_{\al_1}}\big)^*\bE_{\al_1}\oplus I^*\big(\pi_{\fM^{\Fr_2}_{\al_2}}\big)^*\bE_{\al_2}\oplus z\big(\pi_{\fM^{\Fr_1}_{\al_1}}\times \pi_{\fM^{\Fr_2}_{\al_2}}\big)^*\Theta_{\al_1\al_2}\\
&\oplus z^{-1}\big(\pi_{\fM^{\Fr_1}_{\al_1}}\times \pi_{\fM^{\Fr_2}_{\al_2}}\big)^*\sig^* \Theta_{\al_2\al_1}\,.
\end{align*}
This shows that $I^*\big(\pi_{\fM^{\vec\Fr}_{\al}}\big)^* \bE_{\al}|^f_{\bbLambda_{\al_1,\al_2}}$ is indeed isomorphic to the obstruction theory from $\fM_{\al_1}\times \fM_{\al_2}$. Since the above splitting is the same one as was used in \cite[(4.3)]{Bo25}, we deduce that the comparison of orientations \eqref{as:eq:orientation-sum-discrepancy} is identified with the one from \eqref{sp:comp-ass:bE-isom}, so $\varepsilon_Z  = \varepsilon_{\al_1,\al_2}$.

 From \eqref{eq:pairs-master-3-relative-Tvir} we see that
\[ \Omega_{\pi_{\fM_\alpha} \circ I_{\alpha,1}}^{>}\Big|_{\bbLambda_{\alpha_1,\alpha_2}} = z\cV_1 \otimes \cFr_2(\cE_1)^{\vee} \]
has rank $\fr_2(\alpha_1)$. Thus we conclude that
\[ \big[\tilde{\cO}^\vir_{\bbLambda_{\alpha_1,\alpha_2}}\big] = \varepsilon_{\alpha_1,\alpha_2} (-1)^{\fr_1(\alpha_2)} \big[\hat{\cO}^\vir_{\bbLambda_{\alpha_1,\alpha_2}}\big]. \]

Note that the $\bG_m$-action on the first factor of $\fM_{\al_1}\times\fM_{\al_2}$ rescales the automorphisms making $\Phi_{\al_1,\al_2}$ non-trivially $\bG_m$-equivariant. This is where the action of $D(z)$ from §\ref{sec:LA-quotient} acting on the first factor appears from. Putting it all together, the contribution
\eqref{eq:pairs-master-term-3} becomes $\rho_z$ of
\begin{align*}
  \varepsilon_{\alpha_1,\alpha_2}  &(\Phi_{\al_1,\al_2})_*\big(D(z)\otimes \id\big)\Big((\pi_{\fM_{\alpha_1}^{\Fr_1}})_*\big(I_*\sZ_{\alpha_1,1}^{\Fr_1}(\tau^{\JS}) \cap\hat{\fe}(\Omega_{\pi_{\fM_{\alpha_1}^{\Fr_1}}}^\vee)\big)\\
  &\boxtimes (\pi_{\fM_{\alpha_2}^{\Fr_2}})_*\big(I_*\sZ_{\alpha_2}^{\Fr_2}(\tau^{\JS}) \cap\hat{\fe}(\Omega_{\pi_{\fM_{\alpha_2}^{\Fr_2}}}^\vee)\big)\cap \hat{\fe}_{\bT}(z\Theta_{\al_1\al_2})^{-1}\Big)\,.
\end{align*}
Using the Lie bracket from §\ref{sec:LA-quotient} and the defining equation \eqref{eq:pairs-invariant}, this is equal to
$$  \fr(\al_1)\fr(\al_2)[\breve
\sZ_{\alpha_1}^{\Fr_1}(\tau^{\JS}), \breve
  \sZ_{\alpha_2}^{\Fr_2}(\tau^{\JS})]\,.$$
\subsubsection{}

Putting it all together, we have obtained the formula
\[ \breve \sZ_\alpha^{\Fr_2}(\tau^{\JS}) - \breve \sZ_\alpha^{\Fr_1}(\tau^{\JS})    =  \sum_{\substack{\alpha = \alpha_1+\alpha_2\\\al_1,\al_2\in R_{\al}}} \frac{\fr_1(\al_1)\fr_2(\al_2)}{\fr_1(\al)\fr_2(\al)}\cdot  [\breve \sZ_{\alpha_1}^{\Fr_1}(\tau^{\JS}), \breve \sZ_{\alpha_2}^{\Fr_2}(\tau^{\JS})]. \]
This is the equivalent of formula \cite[(9.51)]{Joyce2021} necessary for running the full argument in \cite[§9.3]{Joyce2021}.\footnote{One only needs to remember that we additionally factored out $\fr(\al)$ in \eqref{eq:pairs-invariant}.} See also \cite[§3.3]{KLT25}. Thus we have proved the general claim of Theorem \ref{thm:sst-invariants}.

\section{Wall-Crossing}
\label{sec:wall-crossing}
\subsection{Strategy}
\label{sec:strategy}
The master-space localization techniques used in this section to prove the wall-crossing formula go back to \cite{mochizuki}, and have been developed into a general wall-crossing framework for abelian categories in \cite{Joyce2021}. These methods have then been adapted in \cite{KLT25} for wall-crossing in general equivariant CY$3$ categories. For some (potentially equivariant) CY$4$ categories including representations of CY4 quivers and sheaves and pairs on local CY fourfolds, this proof appeared in  \cite{Bo25}. The present paper proves the general CY4 wall-crossing formula under the assumptions in Section~\ref{sec:assumptions-all}. We combine the machinery developed in the above papers with the techniques worked out in §\ref{sec:sym-pullback}, which allows us to run the familiar master space localization argument. This recovers the horizontal flag wall-crossing formula \eqref{eq:simple.flag-wall-crossing}. As explained in \cite[§8.1]{Bo25}, this formula together with its projection to $\fM^{\pl}_{\cat{A}}$ are sufficient to recover the general statement in Theorem \ref{thm:general-wall-crossing-intro}.

\subsubsection{}
\label{sec:piecing-wc-together}
First, using Assumptions~\ref{ass:stab} (b), (c), (e), and (f), as in \cite[\S8.1 (1)\&(4)]{Bo25} which summarizes and adapts \cite[\S11]{Joyce2021}, we can piece together the general wall-crossing formula between two stability conditions, by moving along a continuous path $\gamma_t$ of stability conditions chosen in \ref{ass:stab} (b). This consists of crossing dominant walls at finitely many points $t_i$, thus moving onto and off of a wall where additional objects may become destabilized. This is illustrated in the following figure representing the dominant wall-crossing at each wall $t_i$.
\begin{equation*}
    \vcenter{\hbox{\includegraphics[width=0.45\textwidth]{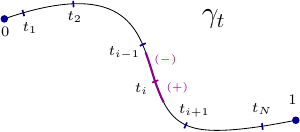}}}\label{wc:fig:piecingWCtogether}
\end{equation*}
Here, at each $t=t_i$, we fix $\mathring t$ in the segment ${\color{violet}(-)}$ and wall-cross onto the wall, and then fix $\mathring t$ in the segment ${\color{violet}(+)}$ to wall-cross again off the wall. To recover the general wall-crossing formula, we iterate this  procedure for every wall $t_i$ on the path between $\gamma_0$ and $\gamma_{1}$. 

\subsubsection{}\label{wc:sec:dominant-wc-strategy}
Fix the stability conditions $\tau := \gamma_{t_i}$ on the wall and $\mathring\tau:=\gamma_{\mathring t}$ off the wall for $t_i, \mathring t$ from §\ref{sec:piecing-wc-together}. Using Assumptions~\ref{ass:stab} (d), (e), (f), (g), (h), and (i) , as in \cite[\S8.1 (2)]{Bo25}, \cite[\S4.1]{KLT25}, and \cite[\S10]{Joyce2021}, we can lift dominant wall-crossing between $\tau$ and $\mathring\tau$ to the auxiliary flag moduli stacks constructed as in §\ref{bg:sec:aux-stacks} over $\fM_{\cat{A}}$ for the quiver 
$$\label{wc:eq:flag-quiver}
 \includegraphics{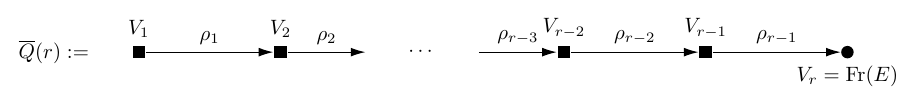}
$$
where the class $(\alpha,\vec{d})$ in \eqref{eq:barK-of-auxiliary} satisfies\footnote{This is a slightly more restrictive condition than the \textit{flag} condition in \cite[4.1.8]{KLT25}, but it is sufficient for our purposes.} 
\begin{equation}
\label{eq:flag-dimension-vector}
d_a \leq  d_{a+1}\leq d_{a}+1  \quad \text{for}\quad 1\leq a\leq r-2 \,,\qquad d_1 = 1\,,\qquad\textnormal{and}\qquad d_{r-1} = \fr(\al)-1\,.\end{equation}
 Using the stability conditions from \eqref{eq:barphi}, one refines the more complicated wall at each $t_i$ in §\ref{sec:piecing-wc-together} to a series of simple walls at $s_k\in (0,1)$ represented as horizontal flag wall-crossing in Figure \ref{fig:dominant-wc-strategy-1}.
\begin{figure}[!ht]
  \centering
  \begin{tikzpicture}
    \draw[->,thick] (-0.5,0)--(11.7,0) node[below]{$s$};
    \draw (0.8,0)--(0.8,.15);
    \draw (5.5,0)--(5.5,.15);
    \draw (10.2,0)--(10.2,.15);
    
    \node[above] at (0.8,0) {$s_1$};
    \node[above] at (2.5,0) {$\cdots$};
    \node[above] at (5.5,0) {$s_k$};
    \node[above] at (8.5,0) {$\cdots$};
    \node[above] at (10.2,0) {$s_p$};
    
    \node[vertex] at (-0.3,0) {};
    \node[below] at (-0.3,-0.1) {$0$};
    \node[vertex] at (11.2,0) {};
    \node[below] at (11.2,-0.1) {$1$};
    
    \node[above] at (-1,0) {$\sZ_{\alpha,\mathbf{d}}(\tau^0)$};
    \node[] at (-1,-3) {$\breve{\sZ}_{\alpha}(\tau)$};
    \node[above] at (12,0) {$\sZ_{\alpha,\mathbf{d}}(\tau^1)$};
    \node[] at (12,-3) {$\breve{\sZ}_{\alpha}(\mathring{\tau})$};
    \node[] at (-1,-5.5) {$\sz_{\alpha}(\tau)$};
    \node[] at (12,-5.5) {$\sz_{\alpha}(\mathring\tau)$};

    \node[below] at (4.5,0) {$\sZ_{\alpha,\mathbf{d}}(\tau^{s_-})$};
    \node[below] at (6.5,0) {$\sZ_{\alpha,\mathbf{d}}(\tau^{s_+})$};

    \node[] at (4.4,-3) {$\breve{\sZ}_{\alpha,\mathbf{d}}(\tau^{s_-})$};
    \node[] at (6.6,-3) {$\breve{\sZ}_{\alpha,\mathbf{d}}(\tau^{s_+})$};

    \draw[<-, thick, Green, text=darkgray] (4.5,-2.7)--(4.5,-0.6);
    \draw[<-, thick, Green, text=darkgray] (6.5,-2.7)--(6.5,-0.6);
    
    \draw[<->, thick, decorate, decoration=snake] (-1,-3.3)--(-1,-5.2);
    \draw[<->, thick, decorate, decoration=snake] (12,-3.3)--(12,-5.2);
    \draw[<->, thick, red, text=darkgray] (-0.4,-5.5)--node[above]{desired {\it dominant wall-crossing}} (11.4,-5.5);
    
    \draw[<-, thick, Green, text=darkgray] (-1,-2.7)--node[left,align=center]{flag\\push-\\forward}(-1,0.1);
    \draw[<-, thick, Green, text=darkgray] (12,-2.7)--node[right, align=center]{flag\\push-\\forward}(12,0.1);
    
    \node at (0.4,0.1) (sm1) {};
    \node at (1.2,0.1) (sp1) {};
    \draw[<->, thick, orange, text=darkgray, bend left=90, distance=22] (sm1) to (sp1);
    \node at (5.1,0.1) (sm2) {};
    \node at (5.9,0.1) (sp2) {};
    \draw[<->, thick, orange, text=darkgray, bend left=90, distance=22] (sm2) to (sp2);
    \node at (9.8,0.1) (sm3) {};
    \node at (10.6,0.1) (sp3) {};
    \draw[<->, thick, orange, text=darkgray, bend left=90, distance=22] (sm3) to (sp3);
    \draw[<->, thick, orange] (5.2,-3) to (5.8,-3);

    \node[] at (5.5,-2.55) {\S\ref{prop:OmegaWC}};
    \node[] at (6.5,0.5) {\eqref{eq:simple.flag-wall-crossing}};
    
    \node[above, text=darkgray] at (5.5,0.8) {horizontal flag wall-crossing};
    
    \draw[<->,thick,blue] (11.4,-5) to (11.4,-3.4) --node[below]{our path} (-0.4,-3.4) -- (-0.4,-5);
  \end{tikzpicture}
  \caption{}
  \label{fig:dominant-wc-strategy-1}
\end{figure} 
In Section~\ref{sec:horizontalflagWC}, we prove the \textit{simple flag wall-crossing} at each $s_k$ separately. The dominant wall-crossing formula can be recovered from the simple wall-crossings and a pushforward formula via a combinatorial procedure explained in \cite[\S 4]{KLT25}, following \cite[\S 10]{Joyce2021}. Instead, we follow \cite[§8.1 (3), §8.5]{Bo25} here, which shows that this may be simplified in our case by pushing the resulting expression forward to $\fM^{\pl}_{\cat{A}}$ at every simple wall-crossing step. This produces a wall-crossing formula between $\breve{\sZ}_{\alpha}^{\Fr}(\tau)$ and $\breve{\sZ}_{\alpha}^{\Fr}(\mathring{\tau})$  from \eqref{eq:pairs-invariant} which is a composition of the single wall-crossing steps in Proposition \ref{prop:OmegaWC}.\footnote{This approach is extracted from \cite[§10.5]{Joyce2021} and skips comparing the combinatorics in each step there.} This step, represented by the \textit{flag pushforward} arrows in Figure \ref{fig:dominant-wc-strategy-1}, is addressed in Section~\ref{sec:flagprojection} and is a direct application of Corollary \ref{cor:flag-pushforward}.
\subsubsection{}
    \begin{remark}
        Note that the only geometric wall-crossing in the above procedure is the horizontal flag wall-crossing that take place inside of $\cat{A}^{Q(\vec\Fr)}_{\tau,\phi(\al)}$. When it comes to wall-crossing for stable pairs with fixed-determinant obstruction theories, Assumption \ref{ass:pairWC} was set up in such a way that the corresponding replacement category $\cat{B}^{Q(\vec\Fr)}_{\tau,\phi(\beta)}$ does not distinguish between the left and the right-hand side of \eqref{eq:equivalence-cat-pairs}. In particular, one can still wall-cross with the stabilities $\tau$ in the heart $\cat{B}\subset D^b(X)$ and use the fixed-determinant obstruction theories, but one also has the necessary moduli spaces of flags constructed using the framing $\Fr$ from Assumption \ref{ass:pairWC} (d) as in \eqref{eq:framedstablemodulispace}. Thus, the same arguments as above still apply to prove the corresponding wall-crossing formula.
    \end{remark}

\subsection{Horizontal Flag Wall-Crossing}\label{sec:horizontalflagWC}
\subsubsection{}
\label{sec:simple-wall-set-up}
For a class $(\al,\vec d)$ of the quiver $\bar{Q}(r)$ as in \eqref{eq:flag-dimension-vector}, the stability $\tau^{\lambda}_{\vec{\delta}}$ used in the horizontal wall-crossing is given for $\vec{\delta}$ satisfying the condition
$$
1\gg \delta_1\gg \delta_2\gg \cdots  \gg \delta_{r-2}\gg \delta_{r-1} >0
$$
explained in more detail in \cite[(10.5)]{Joyce2021}. The group-homomorphism $\lambda: \bar{K}(\cat{A})\to \bR$ is given by $\lambda = s\lambda^{t,t'}_{\al}$ for $t,t'\in [0,1]$ as in §\ref{sec:piecing-wc-together} and $s\in [0,1]$, where $s$ is precisely the variable in the horizontal flag wall-crossing of Figure \ref{fig:dominant-wc-strategy-1}. We call the resulting stability simply $\tau^s$, omitting $\vec{\delta}$ as its role is to make the maps $\rho_{i}$ injective. In particular, we may always choose $\vec{d}$ so that 
\begin{equation}
\label{eq:injective-implies}
d_a = a\qquad \textnormal{for } 1\leq a\leq {r-1} \,.
\end{equation}

The result in \cite[Proposition 10.16]{Joyce2021} shows that the simple walls at $s_k$ in Section~\ref{wc:sec:dominant-wc-strategy} occur precisely when a $\tau^{s_k}$-semistable object splits into two $\tau^{s_k}$-semistables with classes on the right-hand side of
\begin{equation*}
  (\al,\vec d)=(\beta,\vec e)+(\gamma,\vec f)
\end{equation*}
also satisfying \eqref{eq:flag-dimension-vector}. Moreover, for a fixed $s_k$, the set 
\begin{equation}
\label{eq:splitting-pair-Jo}
J_k:= \Big\{\Big((\beta,\vec e), (\gamma,\vec f)\Big)\Big\}
\end{equation}
of such pairs is finite.

Setting\footnote{In Figure \ref{fig:dominant-wc-strategy-1}, we set $s_0=0$ and $s_{p+1} = 1$.}
\begin{equation}
\label{eq:definition-s-pm}
s^-\in (s_{o-1}, s_{o})\qquad \textnormal{and}\qquad s^+\in (s_{o}, s_{o+1})\,,
\end{equation}
we thus want to prove a wall-crossing formula relating the virtual invariant of $M_{\al,\vec d}^{\bar{Q}(r)}(\tau^{s_-})$ to the one of $M_{\al,\vec d}^{\bar{Q}(r)}(\tau^{s_+})$ in terms of the invariants of $M^{\bar{Q}(r)}_{\beta,\vec e}(\tau^{s_k})$ and $M^{\bar{Q}(r)}_{\gamma,\vec f}(\tau^{s_k})$. The corresponding master space is recalled below.

\subsubsection{}
\begin{definition}[{\cite[Definition 10.19]{Joyce2021}}]\label{wc:def:horizontal-master-space}
  Fix $(\al, \vec d)$ as in \eqref{eq:injective-implies} and let $s_k \in (0,1)$ for $o\in \{1,\cdots,p\}$ be as in §\ref{sec:simple-wall-set-up}. Let $\{\vec e, \vec f\}$ be ordered so that $\vec e > \vec f$ in lexicographical ordering. Let $b$ be the smallest index such that $f_b > 0$, and consider\footnote{Note that, unlike \cite{Bo25} and the present work, \cite{Joyce2021, KLT25} consider the quiver with $\rho_{-1}$ from the smallest index $a$ such that $e_a>0$, but since that means the first vector spaces are zero-dimensional, and semistable invariants can be defined using $\vec 1$ framings of any length following \cite[\S4.3]{KLT25}, \cite[(5.14)]{Bo25}, this ends up being equivalent, and the computation is unchanged.} the quiver $\widehat{Q}(r)$ given by
  \begin{equation}
\vcenter{\hbox{\includegraphics[width=0.9\textwidth]{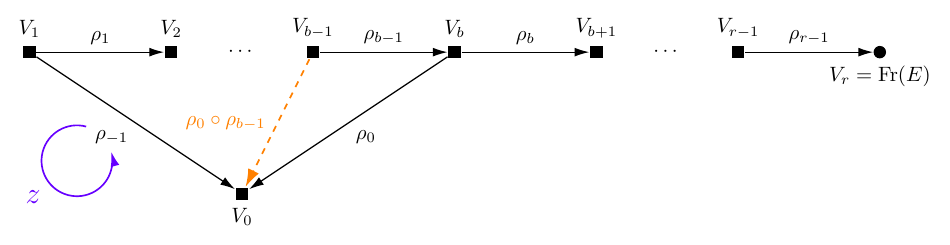}}}\label{wc:fig:wc-ms-quiver}
  \end{equation}
  where the dotted arrow is not an edge but rather the relation $\rho_0 \circ \rho_{b-1} = 0$. Explicitly, let
  \[ \fM^{\hat{Q}(r)}_{\al,\hat{\vec d}} \coloneqq \left\{\left[(E,\vec V, \vec \rho)\right] : E \in \fM_{\al}, \; \dim \vec V = \hat{\vec d}\right\} \]
  be the moduli stack of triples. We only work with $\hat{\vec d}
  \coloneqq (1, \vec d)$, i.e. $\dim V_0 = 1$. On the resulting auxiliary moduli stack, consider the stability condition $\hat{\tau}:=\tau^{\lambda}_{\hat{\vec\delta}}$ where $\lambda=s_k\lambda^{t,t'}$ and $\hat{\vec\delta} = (-\epsilon, \vec\delta)$ for a sufficiently small $\epsilon > 0$ as in \cite[§10.6]{Joyce2021}. Since the dotted arrow is a relation, rather than a map, we take as the \textit{master space} the closed substack
  \[ \bM \subset \fM^{\hat{Q}(r)}_{\al,\hat{\vec d}}(\hat{\tau}) \]
  of the semistable locus, where $\rho_0 \circ \rho_{b-1} = 0$.
\end{definition}

\subsubsection{}
Since $d_b = d_{b-1} +1$ and all semistable objects have injective morphisms $\rho_j$, see \cite[Prop. 10.3]{Joyce2021} and \cite[4.1.12]{KLT25}, this is the substack where $\rho_0$ factors through a morphism $V_b/V_{b-1} \to V_0$. We can hence describe the master space $\bM$ as an open subspace
\begin{equation*}
  \bM\mathring{\subset} \tot\bigg(\big(\bigoplus_{i=1}^{r-1}\cV_i^\vee\otimes\cV_{i+1}\big)\oplus \left(\cV_1^\vee\otimes\cV_0\right)\oplus\left((\cV_b/\cV_{b-1})^\vee\otimes\cV_0\right)\bigg) \xrightarrow{f_\bM} \fM^{\Fr,\pl}_{\alpha} \times \prod_{v \in \hat{Q}_0^f} [\pt/\GL(\hat{d}_v)].
\end{equation*}
We write $p_1$ and $p_2$ for the projections of $\fM^{\Fr,\pl}_{\alpha} \times \prod_{v} [\pt/\GL(d_v)]$ to $\fM^{\Fr,\pl}_{\alpha}$ and $\prod_{v} [\pt/\GL(d_v)]$ respectively, and 
\begin{equation*}
  \pi:\bM\to \fM_\alpha^{\Fr,\pl}
\end{equation*}
for the smooth projection. Using the vector bundle presentation of $\bM$ above, the argument in Section~\ref{bg:sec:aux-stacks} then shows that the relative cotangent complex of $\pi$ is
\begin{equation}\label{wc:eq:master-space-relative-cotangent}
  \bL_\pi \cong \left[\bigg(\bigoplus_{i=1}^{r-1}\cV_i\otimes\cV_{i+1}^\vee\bigg)\oplus \left(\cV_1\otimes\cV_0^\vee\right)\oplus\left((\cV_b/\cV_{b-1})\otimes\cV_0^\vee\right)\to\bigg(\bigoplus_{v}\cV_v\otimes\cV_v^\vee\bigg)\right],
\end{equation}
which we will use in the proof of Proposition~\ref{prop:master-space-fixed-loci} below.

\subsubsection{}
The $\bGm$-action on $\bM$ is induced by scaling the map $\rho_{-1}$ with weight denoted by $z$ as represented by the purple loop in \eqref{wc:fig:wc-ms-quiver}. Since there is an underlying $\sT$-action, we write $\bT\coloneqq \bGm\times\sT$ and work $\bT$-equivariantly on $\bM$.

By Assumption \ref{ass:stab} (i), the master space is a quasi-compact, separated algebraic space with $\bT= \bG_m\times \sT$-equivariant resolution property. By applying Theorem~\ref{sp:thm:symm-pullback-localization} to the smooth projection map $\pi$ and the obstruction theory $\bE^{\rig}$ from \eqref{eq:EE-rigidified}, we obtain a virtual class $\big[\hat\cO^\vir_{\bM}\big]$ amenable to $\bT$-equivariant localization.
\begin{proposition}[{\cite[Prop. 8.6]{Bo25}, \cite[Props. 10.20, 10.21]{Joyce2021}}]
\label{prop:master-space-fixed-loci}
  The $\bGm$-fixed point locus of $\bM$ is the disjoint union of the following pieces:
  \begin{enumerate}
  \item Let $\bM_{-} \coloneqq \{\rho_{-1} = 0\}
    \subset \bM$. Then $\rho_0 \neq 0$ holds due to $\hat{\tau}$-stability. For $s_-$
    from \eqref{eq:definition-s-pm}, there is a natural isomorphism of stacks $\bM_{-}\xrightarrow{\sim} M_{\al,\vec d}^{\bar{Q}(r)}(\tau^{s_-})$ acting by
    $$
      \left[(E, (V_0, \vec V), (0, \rho_0, \vec \rho))\right] \mapsto \left[(E, \vec V, \vec \rho)\right],
  $$
  which is illustrated by the following quiver
    \begin{equation}
    \vcenter{\hbox{\includegraphics[width=0.85\textwidth]{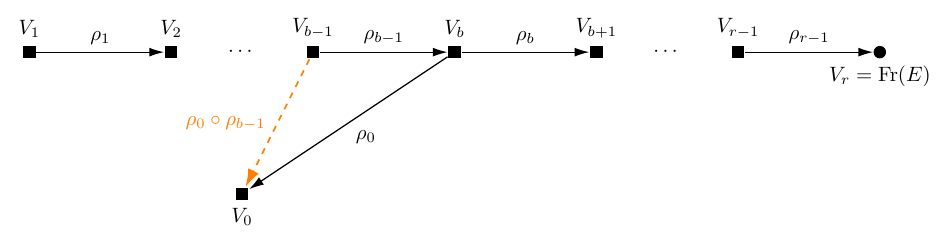}}}\label{wc:fig:wc-ms-fixed-locus-1}
    \end{equation}
    Writing $\nu_-:M_{\al,\vec d}^{\bar{Q}(r)}(\tau^{s_-})\to \fM_{\al}^{\Fr,\pl}$ for the forgetful morphism, we get $\Omega_\pi|_{\bM_-}\cong \Omega_{\nu_-}\oplus z^{-1}\big(\cV_0^\vee\otimes \cV_1\big)$. The virtual normal bundle of $\bM_{-}$ is
    \begin{equation}\label{wc:eq:ms-vir-normal-bdl-rho-1}
        \bN_{-}=z\big(\cV_1^\vee\otimes\cV_0\big)\oplus z^{-1}\big(\cV_0^\vee\otimes \cV_1[-2]\big).
    \end{equation}
    \item Let $\bM_{+} \coloneqq \{\rho_0 = 0\} \subset
    \bM$. Then $\rho_{-1} \neq 0$ holds due to $\hat{\tau}$-stability. For $s_+$ from \eqref{eq:definition-s-pm}, there is a natural isomorphism of stacks $\bM_{+} \xrightarrow{\sim} M_{\al,\vec d}^{\bar{Q}(r)}(\tau^{s_+})$ acting by
    \begin{align*}
      \left[(E, (V_0, \vec V), (\rho_{-1}, 0, \vec \rho))\right] &\mapsto \left[(E, \vec V, \vec \rho)\right],
    \end{align*}
   which is illustrated by the following quiver:
    \begin{equation}
    \vcenter{\hbox{\includegraphics[width=0.85\textwidth]{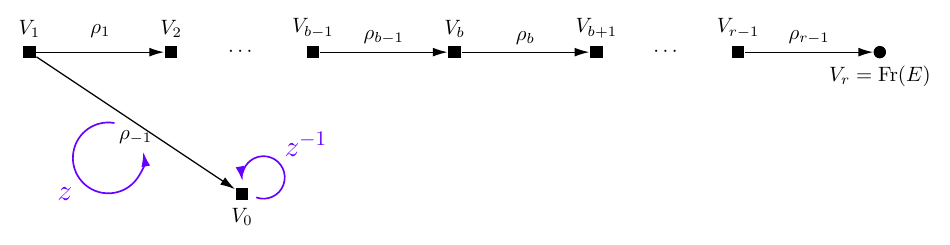}}}\label{wc:fig:wc-ms-fixed-locus-2}
    \end{equation}
    Writing $\nu_+:M_{\al,\vec d}^{\bar{Q}(r)}(\tau^{s_+})\to \fM_\al^{\Fr,\pl}$ for the forgetful morphism, we get $\Omega_\pi|_{\bM_+}\cong\Omega_{\nu_+}\oplus z\big(\cV_0^\vee\otimes (\cV_b/\cV_{b-1})\big)$ and its virtual normal bundle is
    \begin{equation}\label{wc:eq:ms-vir-normal-bdl-rho0}
        \bN_{+}=z^{-1}\big((\cV_b/\cV_{b-1})^\vee\otimes\cV_0\big)\oplus z\big(\cV_0^\vee\otimes (\cV_b/\cV_{b-1})[-2]\big).
    \end{equation}
  \item For each splitting
    $(\al, \vec d) = (\beta, \vec e) + (\gamma, \vec f)$ for a pair in $J_k$ from \eqref{eq:splitting-pair-Jo}, let
    \begin{equation}
        \bM_{(\beta,\vec e)}^{(\gamma, \vec f)} \coloneqq \left\{\left[(E' \oplus E'', (V_0, \vec V' \oplus \vec V''), (\rho_{-1}, \rho_0, \vec \rho' \oplus \vec \rho''))\right] : \begin{array}{c} \rho_{-1} \neq 0 \\ \rho_0\big|_{V'_b} = 0, \; \rho_0\big|_{V''_b} \neq 0 \end{array}\right\} \subset \bM
    \end{equation}
    where the class of $((E', \vec V', \vec \rho'))$ is $(\beta, \vec e)$ and
    the class of $((E'', \vec V'', \vec \rho''))$ is $(\gamma, \vec f)$. There are
    natural isomorphisms of stacks\footnote{In the proof of wall-crossing, the target stacks have no strictly $\tau^{s}$-semistable objects, see \cite[Lemma 4.4.5]{KLT25}.} $ \bM_{(\beta,\vec e)}^{(\gamma, \vec f)} \xrightarrow{\sim}  M^{\bar{Q}(r)}_{\beta,\vec e}(\tau^{s_k}) \times M^{\bar{Q}(r)}_{\gamma,\vec f}(\tau^{s_k}),\label{eq:cross-term-fixed-locus-isomorphism}$ acting by
    \begin{align}
      \left[(E, (V_0, \vec V), (\rho_{-1}, \rho_0, \vec\rho))\right] &\mapsto \left(\left[(E', \vec V', \vec \rho')\right], \left[(E'', \vec V'', \vec \rho'')\right]\right), \nonumber
    \end{align}
   which is illustrated by the following quiver
    \begin{equation}
    \vcenter{\hbox{\includegraphics[width=0.85\textwidth]{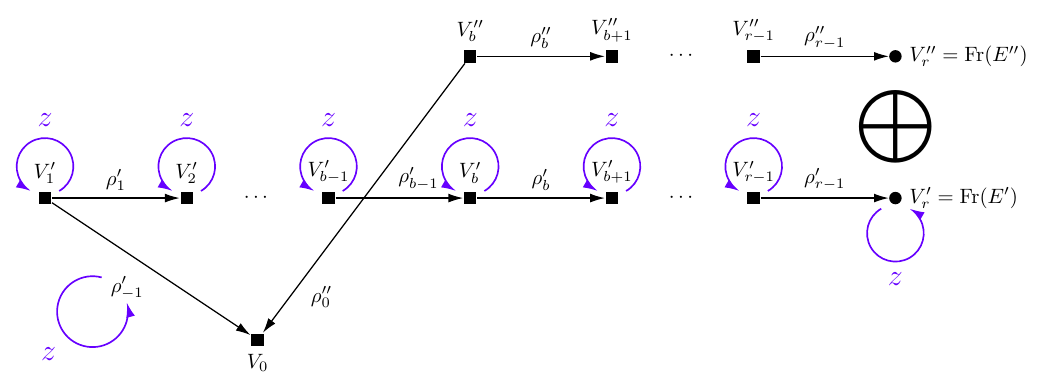}}}\label{wc:fig:wc-ms-fixed-locus-3}
    \end{equation}
    Writing $\nu_\times:M^{\bar{Q}(r)}_{\beta,\vec e}(\tau^{s_k}) \times M^{\bar{Q}(r)}_{\gamma,\vec f}(\tau^{s_k})\to \fM_\beta^{\Fr,\pl}\times\fM_\gamma^{\Fr,\pl}$ for the forgetful morphism and using \eqref{eq:framed-stack-forgetful-map-cotangent}, we have
    \begin{equation}
    \label{eq:Omega-pi-locus-3}
    \Omega_\pi|_{\bM_{(\beta,\vec e)}^{(\gamma,\vec f)}}\cong \Omega_{\nu_\times}\oplus z\big(\bF^{\bar{Q}(r)}_{\beta\gamma}(\vec e,\vec f)\big)^{\vee}\oplus z^{-1}(12)^*\big(\bF_{\gamma\beta}^{\bar{Q}(r)}(\vec f,\vec e)\big)^\vee\,.\end{equation}
   In terms of $\Theta$ from \eqref{eq:ThetaofA}, the positive half of the virtual normal bundle in the sense of §\ref{sp:sec:loc-setup} is identified with
    \begin{align*}
      \big[\bN^{>}_{\times}\big]&= z\nu_\times^*\Theta_{\beta\gamma} + z\bF^{\bar{Q}(r)}_{\beta\gamma}(\vec e,\vec f)^\vee + z(12)^*\bF_{\gamma\beta}^{\bar{Q}(r)}(\vec f,\vec e) \,
    \end{align*}
    under the isomorphism above.
  \end{enumerate}
\end{proposition}
\subsubsection{}

\begin{proof}
  The fixed point loci have been determined in \cite[Prop. 10.20]{Joyce2021}, which also describes the action of $\bGm$ on the vertices used to neutralize the action on the edges as represented by the purple loops in the diagrams \eqref{wc:fig:wc-ms-fixed-locus-1}, \eqref{wc:fig:wc-ms-fixed-locus-2}, and \eqref{wc:fig:wc-ms-fixed-locus-3}. 
  
  In our case, we know the maps $\pi$, $\nu_\pm$, and $\nu_\times$ are smooth. So, using the actions depicted in \eqref{wc:fig:wc-ms-quiver}, \eqref{wc:fig:wc-ms-fixed-locus-2}, and \eqref{wc:fig:wc-ms-fixed-locus-3} respectively, \eqref{wc:eq:master-space-relative-cotangent} lets us read off that
  \begin{align*}
    \Omega_\pi|_{\bM_-} &\cong \Omega_{\nu_-}\oplus z^{-1}\cV_0^\vee\otimes \cV_1,\\ 
    \Omega_\pi|_{\bM_+}&\cong\Omega_{\nu_+}\oplus z\cV_0^\vee\otimes (\cV_b/\cV_{b-1}),\text{ and}\\
    \Omega_\pi|_{\bM_{(\beta,\vec e)}^{(\gamma,\vec f)}}&\cong \Omega_{\nu_\times}\oplus z\big(\bF^{\bar{Q}(r)}_{\beta\gamma}(\vec e,\vec f)\big)^{\vee}\oplus z^{-1}(12)^*\big(\bF_{\gamma\beta}^{\bar{Q}(r)}(\vec f,\vec e)\big)^\vee,
  \end{align*}
  where the last line can be read off from inserting the direct sum depicted in \eqref{wc:fig:wc-ms-fixed-locus-3} into \eqref{wc:eq:master-space-relative-cotangent} and multiplying out the resulting expression. Details of this can be found in the argument leading up to \cite[(8.29)]{Bo25}.
 
  In the first two cases, the virtual normal bundle is then simply constructed from the relative cotangent bundles above, since the embedding of $\bM_\pm$ into $\bM$ composed with $\pi$ is simply $\nu_\pm$. Note also that \eqref{eq:Omega-pi-locus-3} implies that $\bF^{\bar{Q}(r)}_{\beta\gamma}(\vec e,\vec f)$ and $(12)^*\bF_{\gamma\beta}^{\bar{Q}(r)}(\vec f,\vec e)$ when restricted to $M^{\bar{Q}(r)}_{\beta,\vec e}(\tau^{s_k})  \times M^{\bar{Q}(r)}_{\gamma,\vec f}(\tau^{s_k})$ are both vector bundles.

  Since the obstruction theory $\bE$ can be expressed using the bilinear elements $\Theta_{\alpha,\al}$ as in \eqref{eq:EE=Delta-Theta} we obtain
  \begin{equation}\label{wc:eq:virtual-tangent-pulled-back-crossing-fixed-locus}
    \pi^*\bE_{\al}|_{\bM_{(\beta,\vec e)}^{(\gamma,\vec f)}}\cong \nu_\times^*\big(\bE_{\beta}\boxplus \bE_{\gamma}\big) \oplus z\nu_\times^*\Theta_{\beta\gamma}\oplus z^{-1}\nu_\times^*(12)^*\Theta_{\gamma\beta}\,.
  \end{equation}
 This combines with the relative cotangent bundle computations above to give the desired expression
 \begin{align*}
  \left[\bN_\times^>\right] &= \pi^*\bE_\alpha|_{\bM_{(\beta,\vec e)}^{(\gamma,\vec f)}}^>+\Omega_\pi|_{\bM_{(\beta,\vec e)}^{(\gamma,\vec f)}}^>+T_\pi|_{\bM_{(\beta,\vec e)}^{(\gamma,\vec f)}}^>[2]\\
  &= z\nu_\times^*\Theta_{\beta\gamma} + z\bF^{\bar{Q}(r)}_{\beta\gamma}(\vec e,\vec f)^\vee + z(12)^*\bF_{\gamma\beta}^{\bar{Q}(r)}(\vec f,\vec e),
 \end{align*}
 where the first term uses symmetry of $\Theta$, and the third term can be found by dualizing the expression for $\Omega_\pi|_{\bM_{(\beta,\vec e)}^{(\gamma,\vec f)}}$ above.
\end{proof}
\subsubsection{}
\begin{remark}
    Note that in the setting of Assumption \ref{ass:pairWC} we do not have an obstruction theory on $\fN_{d,\al}(\tau^p)$ when $d=1$ as it is only provided on its rigidification. However, we still have the direct sum morphism 
    $$
    \mu:\fN_{1,\al_1}(\tau^p)^{\pl}\times \fN_{0,\al_2}(\tau^p)\to \fN_{1,\al_1+\al_2}(\tau^p)^{\pl}
    $$
    constructed using the de-rigidification map from Definition \ref{def:de-rigidification}. Let $I$ be an object of $\fN_{1,\al_1}(\tau^p)^{\pl}$ and $F$ and object of $\fN_{0,\al_2}(\tau^p)$. Then we see that
    $$
    \operatorname{Rhom}_X(I\oplus G,I\oplus G)_0 = \operatorname{Rhom}_X(I,I)_0\oplus \operatorname{Rhom}_X(I,G)\oplus \operatorname{Rhom}_X(G,I)\oplus \operatorname{Rhom}_X(G,G)
    $$
  so the analog of \eqref{wc:eq:virtual-tangent-pulled-back-crossing-fixed-locus} still holds in this setting. Thus the computation undertaken below will also apply in the case of Assumption \ref{ass:pairWC}.
\end{remark}
\subsubsection{}
\label{sec:flag-wall-crossing}
Recall the Lie algebra $L(\fM^{\bar{Q}(r)})_\loc$ on $\fM^{\bar{Q}(r)}$ outlined in Section~\ref{sec:homology-theories-LA}. For any $(\beta,\vec e)$, where the moduli space $M^{\bar{Q}(r)}_{\beta,\vec e}(\tau^{s})$ has a virtual class, consider the associated universal invariants
\begin{equation*}
  j_*\sZ_{M^{\bar{Q}(r)}_{\beta,\vec e}(\tau^{s})}\in H^\sT(\fM^{\bar{Q}(r),\pl}_{\beta,\vec e})_{\loc}
\end{equation*} 
of \eqref{eq:universal-inv-general}. Since we only work with dimension vectors with $d_1=1$, following Definition~\ref{def:de-rigidification}, we have a canonical de-rigidification map, which we call $I$ for every $(\beta,\vec e)$ by abuse of notation. We then denote
\begin{equation*}
  \sZ_{\beta,\vec e}(\tau^s)\coloneqq \big[I_*j_*\sZ_{M^{\bar{Q}(r)}_{\beta,\vec e}(\tau^{s})}\big]\in L(\fM^{\bar{Q}(r)})_{\loc}
\end{equation*}
for the class of the pushforward along $I$ in the Lie algebra $L(\fM^{\bar{Q}(r)})_{\loc}$, which we recall is defined as a quotient of $H^\sT(\fM^{\bar{Q}(r)})_{\loc}$. In this notation, we have the following identity.
\begin{theorem}
\label{thm:horizontal-flag-wc}
For any $s_k,s_{\pm}\in (0,1)$ as in \eqref{eq:definition-s-pm}, the simple flag wall-crossing formula
    \begin{equation}
    \label{eq:simple.flag-wall-crossing}
        \sZ_{\al,\vec d}(\tau^{s_+})  = \sZ_{\al,\vec d}(\tau^{s_-}) + \sum_{j\in J_k} \left[\sZ_{\be_j,\vec e_j}(\tau^{s_k}), \sZ_{\gamma_j,\vec f_j}(\tau^{s_k})\right]
    \end{equation}
    holds in $L(\fM^{\bar{Q}(r)})_{\loc}$.
\end{theorem}

\subsubsection{}
\begin{proof}
 We can use Theorem~\ref{sp:thm:symm-pullback-localization} to obtain the virtual classes $\big[\tilde{\mathcal{O}}_{Z}^{\vir}\big]$ from Theorem~\ref{sp:thm:symm-pullback-localization}(a) for each fixed-point locus in Proposition \ref{prop:master-space-fixed-loci}. They satisfy the localization formula
\begin{equation}\label{wc:eq:master-space-localization-base-form}
    \big[\widehat{\cO}_{\bM}^{\vir}\big]=\iota_*\left(\frac{\big[\tilde{\cO}_{\bM_{-}}^\vir\big]}{\widehat{\fe}_{\bT}\big(\bN^{>}_{-}\big)} + \frac{\big[\tilde{\cO}_{\bM_{+}}^\vir\big]}{\widehat{\fe}_{\bT}\big(\bN^{>}_{+}\big)} + \sum_{j\in J_k} \frac{\big[\tilde{\cO}_{\bM_{\times,j}}^\vir\big]}{\widehat{\fe}_{\bT}\big(\bN^{>}_{\times}\big)}\right),
\end{equation}
where we write $\bM_{\times, j}\coloneqq \bM_{(\beta_j,\vec e_j)}^{(\gamma_j,\vec f_j)}$. Each of the three types of fixed point loci has an alternative description using flag spaces by Proposition~\ref{prop:master-space-fixed-loci}, which yields another virtual class $\big[\widehat{\mathcal{O}}_{Z}^{\vir}\big]$ from Theorem~\ref{sp:thm:symm-pullback-localization}(a). We want to use Proposition~\ref{prop:comparison} to compare these virtual classes with the ones in the localization formula above.

\subsubsection{}\label{wc:sec:hor-flag-comparison}
For the first two types of loci, we have the diagram
\begin{equation*}
    \begin{tikzcd}
     M_{\al,\vec d}^{\bar{Q}(r)}(\tau^{s_\pm}) \ar{d}{\nu_\pm} \ar[hookrightarrow]{r}{\iota_\pm} & \bM \ar[loop right]{r}{\bG_m} \ar{d}{\pi} \\
      \fM_\al^\pl \ar[r,equals] & \fM_\al^\pl,
    \end{tikzcd}
\end{equation*}
where $\nu_\pm$ is the forgetful map $\nu_\pm:M_{\al,\vec d}^{\bar{Q}(r)}(\tau^{s_\pm})\to \fM_{\al}$. Since the bottom morphism is the identity, we see that \eqref{sp:comp-ass:bE-isom} is an isomorphism of CY4 obstruction theories. Proposition~\ref{prop:master-space-fixed-loci} shows that \eqref{sp:comp-ass:cotangent-isom} is satisfied, so applying Proposition~\ref{prop:comparison} tells us that 
\begin{equation*}
  \big[\tilde{\cO}_{\bM_{-}}^\vir\big]=(-1)^{\omega_{\pi,-}^>}\Big[\widehat{\cO}_{M^{\bar{Q}(r)}_{\al,\vec d}(\tau^{s_-})}^\vir\Big], \qquad \big[\tilde{\cO}_{Z_{+}}^\vir\big]=(-1)^{\omega_{\pi,+}^>}\Big[\widehat{\cO}_{M^{\bar{Q}(r)}_{\al,\vec d}(\tau^{s_+})}^\vir\Big],
\end{equation*}
where we compute $\omega_{\pi,-}^>=0$ and $\omega_{\pi,+}^>=1$ by Proposition~\ref{prop:master-space-fixed-loci} above.

For the third type of fixed point loci, we note that all dimension vectors involved have first entry equal to $1$, so that, following Definition~\ref{def:de-rigidification}, we have canonical de-rigidification maps all denotes by $I$ by abuse of notation. Hence, we have $\Pi\circ \tilde{\pi}\circ I=\pi$ and $\Pi\circ \tilde{\nu}_\times\circ I=\nu_\times$, where $\tilde{\pi}$ and $\tilde{\nu}_\times$ are the corresponding maps of un-rigidified stacks. Now use Lemma \ref{lem:Mrig-and-M-no-difference} to note that $\big[\tilde{\cO}_{\bM_{(\beta,\vec e)}^{(\gamma, \vec f)}}^\vir\big]$ can equivalently be constructed by starting from the obstruction theory $\bE$ on $\fM_{\cat{A}}$ and applying Theorem \ref{sp:thm:symm-pullback-localization}~(a) along the projection $\tilde{\pi}\circ I:\bM\to \fM_{\al}$. Thus, we may apply Proposition \ref{prop:comparison} to the diagram
\begin{equation*}
    \begin{tikzcd}
      \bM_{(\beta,\vec e)}^{(\gamma, \vec f)}\cong M^{\bar{Q}(r)}_{\beta,\vec e}(\tau^{s_k})  \times M^{\bar{Q}(r)}_{\gamma,\vec f}(\tau^{s_k}) \ar{d}{\tilde{\nu}_\times\circ I} \ar[hookrightarrow]{r}{\iota_\times} & \bM \ar[loop right]{r}{\bG_m} \ar{d}{\tilde{\pi}\circ I} \\
      \fM_\beta\times\fM_\gamma \ar[r,"\Phi"] & \fM_\al,
    \end{tikzcd}
\end{equation*}
where $\Phi$ is the morphism induced by direct sums. By Proposition~\ref{prop:master-space-fixed-loci} and the computation in \eqref{wc:eq:virtual-tangent-pulled-back-crossing-fixed-locus}, the conditions \eqref{sp:comp-ass:cotangent-isom} and \eqref{sp:comp-ass:bE-isom} are satisifed. Comparing the splitting \eqref{wc:eq:virtual-tangent-pulled-back-crossing-fixed-locus} to \cite[(4.3)]{Bo25}, we deduce that the comparison of orientations \eqref{as:eq:orientation-sum-discrepancy} is identified with the one from \eqref{sp:comp-ass:bE-isom}, so the additional sign is $ \varepsilon_{\be,\gamma}$. By Proposition~\ref{prop:comparison}, we then get\footnote{See \cite[(8.9)]{Bo25} for the case when obstruction theories exist.}
\begin{equation*}
  \Big[\tilde{\cO}_{\bM_{(\beta,\vec e)}^{(\gamma, \vec f)}}^\vir\Big] = \varepsilon_{\beta,\gamma}(-1)^{\omega_{\pi,\vec e, \vec f}^>} \Big[\widehat{\cO}^\vir_{M^{\bar{Q}(r)}_{\beta,\vec e}(\tau^{s_k}) }\Big]\boxtimes \Big[\widehat{\cO}^\vir_{M^{\bar{Q}(r)}_{\gamma,\vec f}(\tau^{s_k})}\Big].
\end{equation*}
By the computation of the moving part of the relative cotangent sheaf in Proposition~\ref{prop:master-space-fixed-loci} above, we get $\omega_{\pi,\vec e, \vec f}^>=\rk\big(\bF^{\bar{Q}(r)}_{\beta\gamma}(\vec e,\vec f)\big)$.

\subsubsection{}
We now insert the simplifications from \ref{wc:sec:hor-flag-comparison} into \eqref{wc:eq:master-space-localization-base-form}, and consider the corresponding universal invariants. All dimension vectors involved have first entry equal to $1$, so that, following Definition~\ref{def:de-rigidification}, we have canonical de-rigidification maps. For clarity, we suppress these and the open embeddings of semistable loci from the notation below. We push forward our universal invariants along the forgetful map
\begin{equation*}
  \eta:\bM\to \fM^{\bar{Q}(r),\pl},
\end{equation*}
which satisfies $\eta\circ \iota_\pm=\id$ and $\eta\circ \iota_\times=\Phi^{\bar{Q}(r)}$, where $\Phi^{\bar{Q}(r)}$ is the morphism induced by direct sums on $\fM^{\bar{Q}(r),\pl}$ via de-rigidification. 

To make the K-theoretic residues well-behaved, we introduce the insertion
\begin{equation*}
  \cL_-\coloneqq z\cV_1^\vee\otimes \cV_0,\qquad \cL_+\coloneqq z\cV_0^\vee\otimes(\cV_b/\cV_{b-1})\,.
\end{equation*}
Since the action of $\bG_m$ on $ \fM^{\bar{Q}(r),\pl}$ is trivial, we may take the residue\footnote{In the case of the equivariant homology theories from §\ref{sec:equivariant-homology-bivariant}, this is possible due to Example \ref{ex:trivialThom}.}
\begin{equation*}
  \rho_z\left\{\eta_*\left(\sZ_\bM\cap (\cL_-\otimes\cL_+)^{\frac{1}{2}}\right)\right\}
\end{equation*}
which vanishes due to the properness in Assumption \ref{ass:stab} (i) (see \cite[§8.4]{Bo25}, \cite[Lemma 2.6.4]{KLT25}). Note that $\cL_-$ is trivial on $\bM_+$ and all $\bM_{\times,j}$, whereas $\cL_+$ is trivial on $\bM_-$ and all $\bM_{\times,j}$. Hence, using Proposition~\ref{prop:master-space-fixed-loci} and the computations in Section~\ref{wc:sec:hor-flag-comparison} above, we get from \eqref{wc:eq:master-space-localization-base-form}
\begin{align*}
  0 = &\rho_z\id_*\left(\sZ_{\al,\vec d}(\tau^{s_-})\cap \frac{\cL_-^{\frac{1}{2}}}{\widehat{\fe}_{\bT}\big(\cL_-\big)}\right)-\rho_z\id_*\left(\sZ_{\al,\vec d}(\tau^{s_+})\cap \frac{\cL_+^{\frac{1}{2}}}{\widehat{\fe}_{\bT}\big(\cL_+\big)}\right)\\
  &+\sum_{j\in J_k}\varepsilon_{\beta_j,\vec e_j}^{\gamma_j,\vec f_j}\rho_z\Phi^{\bar{Q}(r)}_*(D(z)\times\id)\Bigg(\frac{\sZ_{\beta_j,\vec e_j}(\tau^{s_k})\boxtimes\sZ_{\gamma_j,\vec f_j}(\tau^{s_k})}{\widehat{\fe}_{\bT}\big(z\nu_\times^*\Theta_{\beta_j\gamma_j} + z\bF^{\bar{Q}(r)}_{\beta_j\gamma_j}(\vec e_j,\vec f_j)^\vee + z(12)^*\bF_{\gamma_j\beta_j}^{\bar{Q}(r)}(\vec f_j,\vec e_j)\big)}\Bigg)
\end{align*}
using the signs $\varepsilon_{(\beta_j,\vec e_j)}^{(\gamma_j, \vec f_j)}$ from Definition~\ref{def:framed-epsilons}. Here, the $D(z)\times\id$ operator enters the third term after localization, since the first entry is scaled by $z$ as illustrated in figure \eqref{wc:fig:wc-ms-fixed-locus-3}. We recognize the class $\Theta_{(\beta_j,\vec e_j)}^{(\gamma_j,\vec f_j)}$ from \eqref{eq:Theta-flag} in the denominator of the third term, which leads to
\begin{equation*}
  \varepsilon_{\beta_j,\vec e_j}^{\gamma_j,\vec f_j}\rho_z\Phi^{\bar{Q}(r)}_*(D(z)\times\id)\Bigg(\frac{\sZ_{\beta_j,\vec e_j}(\tau^{s_k})\boxtimes\sZ_{\gamma_j,\vec f_j}(\tau^{s_k})}{\widehat{\fe}_{\bT}\big(\Theta_{(\beta_j,\vec e_j)}^{(\gamma_j,\vec f_j)}\big)}\Bigg)=\left[\sZ_{\be_j,\vec e_j}(\tau^{s_k}), \sZ_{\gamma_j,\vec f_j}(\tau^{s_k})\right]
\end{equation*}
using the definition of the Lie bracket in \eqref{bg:eq:aux-lie-bracket}. Applying 
\begin{equation*}
  \rho_z \frac{(z\cL)^{\frac{1}{2}}}{\hat{\fe}(z\cL)} = 1\,,
\end{equation*}
which holds for any equivariantly trivial line bundle $\cL$, to the definition of $\cL_{\pm}$, we can also evaluate the residues in the first two terms to obtain the desired expression \eqref{eq:simple.flag-wall-crossing}.
\end{proof}

\subsection{Projecting along Flags}\label{sec:flagprojection}
\subsubsection{}
Here, we follow the arguments of \cite[§8.5]{Bo25} using Corollary \ref{cor:flag-pushforward}. We want to prove a wall-crossing formula between $\breve{\sZ}_{\alpha}^{\Fr}(\tau)$ and $\breve{\sZ}_{\alpha}^{\Fr}(\mathring{\tau})$ where $\tau$ and $\mathring{\tau}$ were fixed in §\ref{wc:sec:dominant-wc-strategy}. For this, recall from \cite[(8.2)]{Bo25}, which follows from \cite[Lemma 11.4, Proposition 10.2 and 10.5]{Joyce2021}, that
$$
M_{\al,\vec d}^{\bar{Q}(r)}(\tau^{1}) = M_{\al,\vec d}^{\bar{Q}(r)}\big((\mathring{\tau})^{0}\big)\qquad \textnormal{and}\qquad \Big[\wh{\cO}^{\vir}_{M_{\al,\vec d}^{\bar{Q}(r)}(\tau^{1})}\Big] = \Big[\wh{\cO}^{\vir}_{M_{\al,\vec d}^{\bar{Q}(r)}\big((\mathring{\tau})^{0}\big)}\Big]\,.
$$
For $\zeta = \tau,\mathring{\tau}$, there is a commutative diagram of forgetful maps
\begin{equation}
\label{eq:NrigstoMrig}
\begin{tikzcd}[column sep=large]
\arrow[dr," \pi^{\bar{Q}}_{\fM^{\Fr}_\alpha}"']\fM_{\al,\vec d}^{\bar{Q}(r)}\arrow[r,"\pi^{\Fr}_{\al,\vec{d}/1}"]& \fM^{Q^{\JS}}_{1,\al}\arrow[d," \pi^{\JS}_{\fM^{\Fr}_\alpha}"]\\
    &\fM_{\al},
\end{tikzcd}\qquad
\begin{tikzcd}[column sep=large]
\arrow[dr," \pi^{\bar{Q}}_{\fM^{\Fr,\pl}_\alpha}"']M_{\al,\vec d}^{\bar{Q}(r)}(\zeta^{0})\arrow[r,"\pi^{\Fr,\pl}_{\al,\vec{d}/1}"]& M^{\Fr}_{1,\al} (\zeta^{\JS})\arrow[d," \pi^{\JS}_{\fM^{\Fr,\pl}_\alpha}"]\\
    &\fM^{\pl}_{\al},
\end{tikzcd}
\end{equation}
where the horizontal map only retains the composed map $V_1\to\Fr(E)$, and the second diagram is induced by the first via rigidification and restricting to semistable loci. The second diagram above satisfies the assumption of Proposition \ref{prop:smoothpushforward}. Moreover, the map $\pi_{\al,\vec{d}/1}^{\Fr,\pl}$ is the full flag-bundle of the vector bundle $\cFr(\cE)/\cV_1$ of rank $\fr(\al)-1$, so we can apply Corollary \ref{cor:flag-pushforward} to prove that 
\begin{align*}
&\big(\pi^{\bar{Q}}_{\fM^{\Fr}_\alpha}\big)_*\Big(\sZ_{\al,\vec d}^{\bar{Q}(r)}(\zeta^{0})\cap\, \hat \fe\big(T_{\pi^{\bar{Q}}_{\fM^{\Fr}_\alpha}}\big)\Big)&\\
= &\big(\pi^{\JS}_{\fM^{\Fr}_\alpha}\big)_*\bigg(\big(\pi^{\Fr}_{\al,\vec{d}/1}\big)_*\Big(\sZ_{\al,\vec d}^{\bar{Q}(r)}(\zeta^{0})\cap\, \hat{\fe}\big(T_{\pi^{\Fr}_{\al,\vec{d}/1}}\big)\Big)\cap\, \hat{\fe}(T_{\pi^{\JS}_{\fM^{\Fr}_\alpha}})\bigg)&\text{(by \eqref{eq:NrigstoMrig})}\\  
=&\big(\fr(\al)-1\big)!\cdot\big(\pi_{\fM^{\Fr}_\alpha}\big)_*\Big(\sZ_{\alpha,1}^{\Fr}(\zeta^{\JS})\cap\, \hat \fe \big(T_{\pi_{\fM^{\Fr}_\alpha}}\big)\Big)&\text{(by Cor.~\ref{cor:flag-pushforward})}\\
=&\fr(\al)! \cdot  \breve{\sZ}_{\alpha}^{\Fr}(\tau)\,.&\text{(by \eqref{eq:pairs-invariant})}
\end{align*}
 Here, we suppress the canonical de-rigidification maps of Definition~\ref{def:de-rigidification} from the notation.

\subsubsection{}
The above computation motivates the definition of 
\begin{equation}\label{wc:def:breve-Z-alpha-d}
\breve{\sZ}_{\alpha,\vec d}^{\Fr}(\tau^{s}) := \frac{1}{\fr(\al)!}\big(\pi_{\fM^{\Fr}_\alpha}\big)_*\Big(\sZ_{\al,\vec d}^{\bar{Q}(r)}(\tau^{s})\cap\, \hat \fe\big(T_{\pi_{\fM^{\Fr}_\alpha}}\big)\Big)
\end{equation}
whenever there are no strictly $\tau^s$-semistable objects of class $(\al,\vec d)$. The wall-crossing between $\breve{\sZ}_{\alpha}^{\Fr}(\tau)$ and $\breve{\sZ}_{\alpha}^{\Fr}(\mathring{\tau})$ is the sum over all $k\in \{1,\ldots,p\}$ from Figure \ref{fig:dominant-wc-strategy-1} of the wall-crossing contribution in the next proposition, which hence finishes the proof of the dominant wall-crossing between $\tau$ and $\mathring{\tau}$.
\begin{proposition}
\label{prop:OmegaWC}
Fix $k\in \{1,\ldots,p\}$ and $s_k,s_{\pm}\in (0,1)$ as in \eqref{eq:definition-s-pm}. Then 
\begin{equation}
\label{eq:OmegaWC}
\breve{\sZ}_{\alpha,\vec d}^{\Fr}(\tau^{s_+}) - \breve{\sZ}_{\alpha,\vec d}^{\Fr}(\tau^{s_-}) = \mathlarger{\mathlarger{\sum}}_{j\in J_k}{\fr(\al)\choose \fr(\beta_j)
}^{-1}\Big[\breve{\sZ}_{\beta_j,\vec e_j}^{\Fr}(\tau^{s_k}),\breve{\sZ}_{\gamma_j,\vec f_j}^{\Fr}(\tau^{s_k})\Big]\,.
\end{equation}
holds in $L(\fM_{\cat{A}})_{\loc}$.
\end{proposition}
\subsubsection{}
\begin{proof}[Proof of Proposition \ref{prop:OmegaWC}]
To prove \eqref{eq:OmegaWC}, one applies 
\begin{equation}
\label{eq:projection-morphism}
\big(\pi_{\fM^{\Fr}_\alpha}\big)_*\left(-\cap \,\hat{\fc}_{\rk}\big(\bL^{\vee}_{\pi_{\fM^{\Fr}_\alpha}}\big)\right)
\end{equation}
to \eqref{eq:simple.flag-wall-crossing}. This was done in homology explicitly in  \cite[§8.5]{Bo25}. It was also observed there that one can apply \cite[Theorem 2.12]{GJT} directly. This is true even though $\bL^{\vee}_{\pi_{\fM^{\Fr}_\alpha}}$ is not a vector bundle on the entire $\fM^{\bar{Q}(r)}_{\al,\vec d}$. The following restriction
\begin{equation}
\label{eq:summands-of-cotangent}
\Phi^*T_{\pi_{\fM^{\Fr}_\alpha}}|_{M^{\bar{Q}(r)}_{\beta,\vec e}(\tau^{s_k})  \times M^{\bar{Q}(r)}_{\gamma,\vec f}(\tau^{s_k})} = T_{\pi_{\fM^{\Fr}_\beta}}\boxplus T_{\pi_{\fM^{\Fr}_\gamma}}\oplus  \bF^{\bar{Q}(r)}_{\beta\gamma}(\vec e,\vec f)\oplus(12)^*\bF_{\gamma\beta}^{\bar{Q}(r)}(\vec f,\vec e)\,,
\end{equation}
where we suppress the de-rigidification maps, is still a  vector bundle because it is the pullback of \eqref{eq:Omega-pi-locus-3}. Thus each of its summands is again a vector bundles, and the argument that proves \cite[Theorem 2.12]{GJT} still applies so \eqref{eq:projection-morphism} acts as a morphism of Lie algebras when restricted to \eqref{eq:simple.flag-wall-crossing}. Since the proof of \cite[Theorem 2.12]{GJT} is not publically available anywhere, we summarize the computation from \cite[§8.5]{Bo25}.

For this, we focus on the single summand 
\begin{equation}
\label{eq:single-bracket-puhsforward}
(\pi_{\fM^{\Fr}_{\al}})_*\left(\left[\sZ_{\be,\vec e}(\tau^{s_k}), \sZ_{\gamma,\vec f}(\tau^{s_k})\right]\cap\hat{\fc}_{\rk}(\bL^{\vee}_{\pi_{\fM^{\Fr}_\alpha}})\right)
\end{equation}
from \eqref{eq:simple.flag-wall-crossing}. We may write this as
$$
 \varepsilon_{\beta,\vec e}^{\gamma,\vec f}\rho_z \big(\pi_{\fM^{\Fr}_{\al}}\circ \Phi^{\bar{Q}(r)}\big)_*(D(z)\times\id)\Bigg(\frac{\sZ_{\beta,\vec e}(\tau^{s_k})\boxtimes\sZ_{\gamma,\vec f}(\tau^{s_k})}{\widehat{\fe}_{\bT}\big(\Theta_{(\beta,\vec e)}^{(\gamma,\vec f)}\big)}\cap\hat{\fc}_{\rk}\big((\Psi\times \id)^*\Phi^*\bL^{\vee}_{\pi_{\fM^{\Fr}_\alpha}}\big)\Bigg)
$$
where we identified the K-theory class of the universal line bundle on $B\bG_m$ with $z$. Using \eqref{eq:summands-of-cotangent} and including the $B\bG_m$-weights, $\cap\,\hat{\fc}_{\rk}\big((\Psi\times \id)^*\Phi^*\bL^{\vee}_{\pi_{\fM^{\Fr}_\alpha}}\big)$ becomes
$$
\cap\left(\hat{\fe}\big(T_{\pi_{\fM^{\Fr}_\beta}}\big)\boxtimes \hat{\fe}\big(T_{\pi_{\fM^{\Fr}_\gamma}}\big)\otimes  \hat{\fe}\big(z^{-1}\bF^{\bar{Q}(r)}_{\beta\gamma}(\vec e,\vec f)\big)\otimes\hat{\fe}\big(z(12)^*\bF_{\gamma\beta}^{\bar{Q}(r)}(\vec f,\vec e)\big)\right)\,.
$$
Combining with the formula \eqref{eq:Theta-flag}, we can cancel out the last two contributions in \eqref{eq:summands-of-cotangent} while picking up precisely the sign $(-1)^{\rk\big(\bF_{\beta\gamma}^{Q(\vec \Fr)}(\vec e,\vec f)\big)}$. Thus we may write \eqref{eq:single-bracket-puhsforward} as 
$$
 \varepsilon_{\beta,\gamma}\rho_z \Phi_*(D(z)\times\id)\left(\frac{({\pi}_{\fM^{\Fr}_{\be}})_*\left(\sZ_{\beta,\vec e}(\tau^{s_k})\cap \hat{\fe}(T_{\pi_{\fM^{\Fr}_\be}})\right)\boxtimes({\pi}_{\fM^{\Fr}_{\gamma}})_*\left(\sZ_{\gamma,\vec f}(\tau^{s_k})\cap \hat{\fe}(T_{\pi_{\fM^{\Fr}_\gamma}})\right)}{\widehat{\fe}_{\bT}\big(\Theta_{(\beta,\gamma)}\big)}\right)\,.
$$
Using \eqref{wc:def:breve-Z-alpha-d} then immediately implies the statements of the proposition once the factorials are included.
\end{proof}
\subsubsection{}
As explained in \S\ref{wc:sec:dominant-wc-strategy}, using the defining relation \eqref{eq:sstable-def} of semistable invariants, we can prove the dominant wall-crossing formula for a fixed $\Fr$
\begin{equation}  
\label{eq:Fr-dependent-wall-crossing}
  \sz^{\Fr}_\alpha(\mathring{\tau}) = \sum_{\substack{n>0\\\alpha = \alpha_1 + \cdots + \alpha_n\\\forall i:\, \fM_{\alpha_i}(\tau)\neq \emptyset}}\tilde U\left(\alpha_1,\dots,\alpha_n;\tau,\mathring{\tau}\right)\left[\left[\cdots\left[\sz^{\Fr}_{\alpha_1}(\tau),\sz^{\Fr}_{\alpha_2}(\tau)\right],\cdots\right],\sz^{\Fr}_{\alpha_n}(\tau)\right]\,.
\end{equation}
The detailed combinatorics of collecting the \eqref{eq:OmegaWC} to get the above wall-crossing formula is explained in \cite[\S10.5]{Joyce2021}. Using the independence of Theorem~\ref{thm:sst-invariants}, this wall-crossing extends to any two stability condition $\tau,\mathring{\tau}\in W$, finishing the proof of Theorem~\ref{thm:general-wall-crossing-intro}.

As explained in \cite[Remark 1.3]{Bo25}, sometimes one can instead directly piece the wall-crossing formula together as in \S\ref{sec:piecing-wc-together} for a fixed $\Fr$. This produces for example \eqref{eq:JSwall-crossingformula}.

\phantomsection
\addcontentsline{toc}{section}{References}

\begin{small}
\bibliographystyle{alpha}
\bibliography{paper}
\end{small}

\end{document}